\documentclass[11pt,leqno]{amsart}
\usepackage{graphicx}
\usepackage{amsmath}
\usepackage{amssymb}
\usepackage{amsthm}
\usepackage{stmaryrd}
\usepackage{mathrsfs}
\usepackage{bm}
\usepackage{color}
\usepackage{enumitem}
\usepackage{verbatim}
\usepackage{mathtools}

\DeclareFontFamily{U}{mathx}{\hyphenchar\font45}
\DeclareFontShape{U}{mathx}{m}{n}{
      <5> <6> <7> <8> <9> <10>
      <10.95> <12> <14.4> <17.28> <20.74> <24.88>
      mathx10
      }{}
\DeclareSymbolFont{mathx}{U}{mathx}{m}{n}
\DeclareFontSubstitution{U}{mathx}{m}{n}
\DeclareMathAccent{\widecheck}{0}{mathx}{"71}

\newcommand{\abs}[1]{\vert#1\vert}

\newcommand{\Abs}[1]{\Vert#1\Vert}

\newcommand{\jump}[1]{\llbracket#1\rrbracket}

\newcommand{\uu}{\underline{u}}

\newcommand{\dtan}{{\mathfrak d}_{\rm tan}}
\newcommand{\dnor}{{\mathfrak d}_{\rm nor}}
\newcommand{\dpar}{{\mathfrak d}_{\parallel}}
\newcommand{\gr}{{\mathtt g}}

\newcommand{\R}{{\mathbb R}}
\newcommand{\N}{{\mathbb N}}
\newcommand{\dt}{\partial_t}

\newcommand{\dx}{\partial_x}

\newcommand{\cI}{{\mathcal I}}
\newcommand{\cE}{{\mathcal E}}

\newcommand{\eps}{\varepsilon}

\makeatletter
\newcommand{\opnorm}{\@ifstar\@opnorms\@opnorm}
\newcommand{\@opnorms}[1]{%
  \left|\mkern-1.5mu\left|\mkern-1.5mu\left|
   #1
  \right|\mkern-1.5mu\right|\mkern-1.5mu\right|
}
\newcommand{\@opnorm}[2][]{%
  \mathopen{#1|\mkern-1.5mu#1|\mkern-1.5mu#1|}
  #2
  \mathclose{#1|\mkern-1.5mu#1|\mkern-1.5mu#1|}
}

\newtheorem{theorem}{Theorem}
\newtheorem{proposition}{Proposition}
\newtheorem{lemma}{Lemma}
\newtheorem{corollary}{Corollary}
\newtheorem{definition}{Definition}
\newtheorem{assumption}{Assumption}

\newtheorem{remark}{Remark}
\newtheorem{example}{Example}

\begin{document}
\title[The $2D$ nonlinear shallow water equations with a partially immersed obstacle]%
{The $2D$ nonlinear shallow water equations with a partially immersed obstacle}

\author{Tatsuo Iguchi and David Lannes}

\address{Tatsuo Iguchi, Department of Mathematics, Faculty of Science and Technology, Keio University, 
3-14-1 Hiyoshi, Kohoku-ku, Yokohama 223-8522, Japan}

\address{David Lannes, Institut de Math\'ematiques de Bordeaux, Universit\'e de Bordeaux et CNRS UMR 5251, 
351 Cours de la Lib\'eration, 33405 Talence Cedex, France}
\maketitle

\begin{abstract}
This article is devoted to the proof of the well-posedness of a model describing waves propagating in shallow water 
in horizontal dimension $d=2$ and in the presence of a fixed partially immersed object. 
We first show that this wave-interaction problem reduces to an initial boundary value problem 
for the nonlinear shallow water equations in an exterior domain, 
with boundary conditions that are fully nonlinear and nonlocal in space and time. 
This hyperbolic initial boundary value problem is characteristic, does not satisfy the constant rank assumption on the boundary matrix, 
and the boundary conditions do not satisfy any standard form of dissipativity. 
Our main result is the well-posedness of this system for irrotational data and at the quasilinear regularity threshold. 
In order to prove this, we introduce a new notion of weak dissipativity, that holds only after integration in time and space. 
This weak dissipativity allows high order energy estimates without derivative loss; 
the analysis is carried out for a class of linear non-characteristic hyperbolic systems, 
as well as for a class of characteristic systems that satisfy an algebraic structural property that allows us to define a generalized vorticity. 
We then show, using a change of  {unknowns}, that  {it} is possible to transform the linearized wave-interaction  {problem} into 
a non-characteristic system,  {which} satisfies this structural property and for which the boundary conditions are weakly dissipative. 
We can therefore use our general analysis to derive linear, and then nonlinear, a priori energy estimates. 
Existence for the linearized problem is obtained by a regularization procedure that makes the problem non-characteristic and strictly dissipative, 
and by the approximation of the data by more regular data satisfying higher order compatibility conditions for the regularized problem. 
Due to the fully nonlinear nature of the boundary conditions, it is also necessary to implement a quasilinearization procedure. 
Finally, we have to lower the standard requirements on the regularity of the coefficients of the operator in the linear estimates 
to be able to reach the quasilinear regularity threshold in the nonlinear well-posedness result. 
\end{abstract}



\section{Introduction}
%
\subsection{General setting}
Motivated by applications to renewable marine energies such as offshore wind turbines or wave energy convertors, 
as well as environmental issues such as the modeling of sea-ice, 
many recent articles are devoted to the study of the interactions between ocean waves and floating structures. 
There are several approaches to describe such interactions. Offshore industry has traditionally employed tank testing, 
but simulations based on computational fluid dynamics (CFD) are becoming increasingly popular and tend to replace model tests. 
CFD-based simulations are particularly relevant to study turbulence-related issues, wave impact, strain on mooring systems, etc. 
However, while CFD is a common tool in aeronautics or the automotive industry, it is not yet the case for wave-structure interactions, 
as explained in \cite{Kim_etal} where the main difficulties are pointed out, 
among which an open ocean environment that requires the computation of a large volume of fluid and the computation of complicated wave fields. 

Less costly than CFD computations is the so-called fully nonlinear potential flow (FNPF) approach \cite{Penalba}; 
viscous effects are neglected, and the flow is assumed to be irrotational. 
The velocity field therefore derives from a scalar velocity potential that can be found by solving the Laplace equation in the fluid domain. 
This approach is of course less precise than CFD, as it misses turbulence-related effects, 
but it is interesting for its ability to capture the nonlinear effects in the wave-structure interactions. 

Both CFD and FNPF approaches are unable to describe wave-structure interactions on a large scale and 
cannot for instance be used to study the impact of a group of floating structures on the dynamics of the waves. 
This is not surprising because the simulation of ocean waves on a domain whose size is of the order of several wave-lengths of the waves 
is out of reach with these methods, even if there is no floating structure. 
For practical applications, waves are indeed not computed using the free-surface Navier--Stokes or irrotational Euler equations, 
which are the underlying mathematical equations for the CFD and FNPF approaches, but with simpler approximate models. 

For instance, most of the computations made to simulate wave farms, which are arrays of up to dozens of wave energy convertors, 
are based on Cummins' equation \cite{Cummins} which is an integro-differential equation for the six degrees of freedom of the floating object. 
This equation relies on several approximations performed on the FNPF equations: 
one neglects the variations of the surface elevation for the domain of definition to the velocity potential 
as well as the variations of the immersed part of the object, 
and the pressure is recovered using a linear approximation of Bernoulli's equation in terms of the velocity potential, 
which is itself computed using Green's functions. 
The resulting model is very efficient from a numerical point of view but, due to these many approximations, 
cannot be used when nonlinear effects are important, in extreme sea conditions for instance. 

More recently, a strategy was proposed in \cite{Lannes2017} to derive wave-structure models based on other approximations of the FNPF system. 
More precisely, the idea was to use the important advances performed in the last two decades on the derivation and 
justification of shallow water asymptotic models; 
  {see} for instance \cite{Iguchi2009,Alvarez,Iguchi2018,Iguchi2018b} and the review \cite{Lannes2020}, 
and to use these models to build new wave-structure interactions systems. 
Compared to Cummins' equation, the limitation is that the range of validity of these models is restricted to shallow water 
while the advantages are that they are able to capture nonlinear effects and that very efficient numerical codes have been developed 
for such wave models and are therefore likely to be used for the ``wave part'' of such wave-structure models. 
To describe better this approach, let us briefly describe it on the physical configuration considered in this article, namely, 
a fixed vertical cylinder partially immersed in a $2+1$ layer of fluid with a flat bottom; see Figure \ref{fig:SideWalls}. 
Denoting by $x\in \R^2$ and $z \in \R$ the horizontal and vertical coordinates, the bottom is located at $z=-H_0$ and 
the surface can be represented at time $t$ by the graph of a function $\zeta(t,\cdot)$. 
At time $t$, the fluid domain is therefore given by 
\[
\Omega(t)=\{(x,z)\in \R^{2+1} \,|\, -H_0< z <\zeta(t,x)\}.
\]
The horizontal plane $\R^2$ is decomposed as $\R^2={\mathcal I}\cup {\mathcal E}\cup \mathit{\Gamma}$, 
where ${\mathcal I}$ is the projection on the horizontal coordinates of the wetted part of the object; 
this region is called the {\it interior} region; 
the projection of the part of the surface which is in contact with the air is called {\it exterior} region and denoted ${\mathcal E}$; 
finally, ${\mathit \Gamma}$ denotes the projection of the contact line. 
Finally, the bottom of the object is parametrized by a function $\zeta_{\rm w}$, where the subscript ${\rm w}$ is for ``wetted''. 

\bigskip
\begin{figure}[h]
\begin{center}
\includegraphics[width=0.8\linewidth]{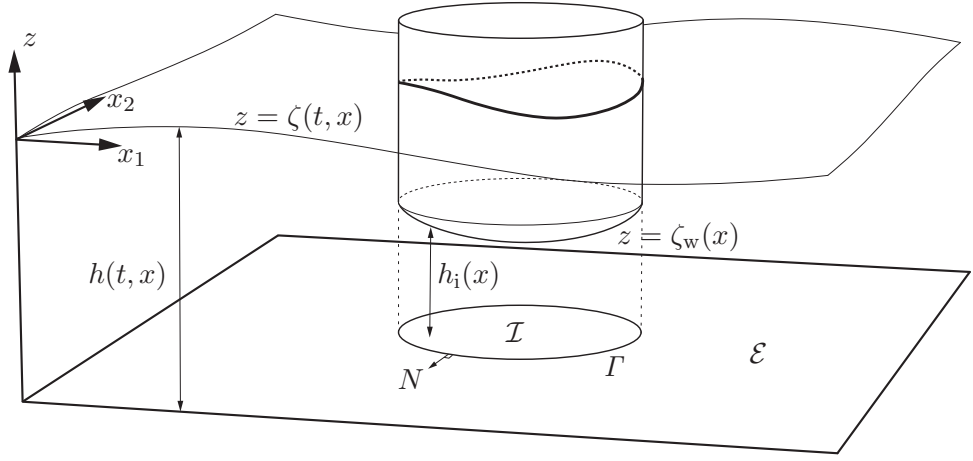}
\end{center}
\setlength{\unitlength}{1pt}
\begin{picture}(0,0)
\put(8,58){$\cI$}
\put(100,50){$\cE$}
\put(45,45){$\mathit{\Gamma}$}
\put(-175,170){$z$}
\put(-139,126){$x_1$}
\put(-143,147){$x_2$}
\put(-150,80){$h(t,x)$}
\put(-18,80){$h_{\rm i}(x)$}
\put(-95,140){$z=\zeta(t,x)$}
\put(50,95){$z=\zeta_{\rm w}(x)$}
\put(-33,40){$N$}
\end{picture}
\vspace{-3ex}
\caption{A fixed solid body with vertical sidewalls}\label{fig:SideWalls}
\end{figure}

There are three steps in the strategy of \cite{Lannes2017}: 
\begin{enumerate}
\item[(i)]
{\it Choice of the wave model}. 
For the configuration mentioned above, the FNPF wave-structure interaction model can be formulated equivalently 
as a system of evolution equations on $\zeta$ and on the vertically averaged horizontal velocity $v$. 
A very common approximation, obtained by neglecting non-hydrostatic pressure terms and 
quadratic interactions of the fluctuations of the horizontal velocity around its vertical average ${v}$, 
results in the {\it nonlinear shallow water equations}, 
\[
\begin{cases}
 \dt \zeta + \nabla\cdot (h{v})=0, \\
 \dt {v}+{v}\cdot \nabla {v} + \gr \nabla\zeta=-\frac{1}{\rho} \nabla \underline{P},
\end{cases}
\]
where $h=H_0+\zeta$ is the depth of the water, $\gr$ is the gravitational constant, $\rho$ is the constant density of the water, 
and $\underline{P}$ is the pressure on the water surface. 
This is the wave model that we consider in this article. 
\item[(ii)]
{\it Identification of the constraints in ${\mathcal E}$ and ${\mathcal I}$.} 
We assume that the nonlinear shallow water equations are solved in both ${\mathcal E}$ and ${\mathcal I}$, 
but with symmetric constraints on $\underline{P}$ and $\zeta$: 
\begin{itemize}
\item
In the exterior region ${\mathcal E}$, there is a constraint on the pressure, namely, $\underline{P}=P_{\rm atm}$, 
where $P_{\rm atm}$ is the atmospheric pressure assumed to be constant; 
on the other hand, there is no constraint on $\zeta$. 
We must therefore solve 
\begin{equation}\label{introSW2}
\begin{cases}
 \dt \zeta+\nabla\cdot (h{v})=0 &\mbox{in}\quad (0,T)\times\cE, \\
 \dt {v}+{v}\cdot \nabla {v}+\gr \nabla\zeta=0  &\mbox{in}\quad (0,T)\times\cE.
\end{cases}
\end{equation}
\item
In the interior region ${\mathcal I}$, this is the reverse: 
there is no constraint on the surface pressure which is one of the unknowns of the problem, 
but there is a constraint on the surface elevation since, by definition, we must have $\zeta_{\rm i}=\zeta_{\rm w}$ in ${\mathcal I}$. 
Here and in what follows, we denote with a subscript ${\rm i}$ the quantities evaluated in the interior region. 
Since $\zeta_{\rm w}$ is time independent in our configuration, we must therefore solve 
\begin{equation}\label{introSW3}
\begin{cases}
 \nabla\cdot (h_{\rm i}{v}_{\rm i})=0 &\mbox{in}\quad (0,T)\times\cI, \\
 \dt {v}_{\rm i}+{v}_{\rm i}\cdot \nabla {v}_{\rm i}+\gr \nabla\zeta_{\rm w}=-\frac{1}{\rho} \nabla \underline{P}_{\rm i}
  &\mbox{in}\quad (0,T)\times\cI
\end{cases}
\end{equation}
with $h_{\rm i}=H_0+\zeta_{\rm w}$; these equations have the same structure as the {\it incompressible} Euler equations; 
the surface pressure $\underline{P}_{\rm i}$ can therefore be understood as the Lagrange multiplier associated with the constraint $\zeta=\zeta_{\rm w}$, 
while it is found through Bernoulli's equation in the standard FNPF approach.
\end{itemize}
\item[(iii)]
{\it   {Matching} conditions on $\mathit{\Gamma}$.} 
Neither \eqref{introSW2} nor \eqref{introSW3} are well-posed without appropriate boundary conditions on the unknowns. 
We must therefore provide such conditions. 
This is not always an easy task. 
In the present case, we derive in Section \ref{subsectSWM} the following matching conditions: 
\begin{equation}\label{introSW4}
\begin{cases}
 N\cdot(h{v})=N\cdot(h_{\rm i}{v}_{\rm i})&\quad\mbox{on}\quad (0,T)\times\mathit{\Gamma}, \\
 P_{\rm atm}+\gr \zeta+\frac{1}{2}\abs{{v}}^2=P_{\rm i}+\gr \zeta_{\rm w}+\frac{1}{2}\abs{{v}_{\rm i}}^2
  &\quad\mbox{on}\quad (0,T)\times\mathit{\Gamma}, \\
 N^\perp\cdot {v}=N^\perp\cdot {v}_{\rm i}&\quad\mbox{on}\quad (0,T)\times\mathit{\Gamma},
\end{cases}
\end{equation}
where $N$ denotes the outward normal vector to $\mathit{\Gamma}$. 
We also show that under such boundary conditions, if the initial velocity field is irrotational in ${\mathcal E}$ and in ${\mathcal I}$, 
then it remains irrotational for all times. 
\end{enumerate}

The goal of this article is to prove that the wave-structure model \eqref{introSW2}--\eqref{introSW4} is locally well-posed 
for irrotational data; see Theorem \ref{th:exist}.

\subsection{Related results}\label{sectrelated}
As mentioned above, it is necessary to understand the initial boundary value problem (IBVP), 
which consists in solving the equations given initial and boundary data, for the wave model under consideration. 
Unfortunately, the local well-posedness of IBVPs with nonhomogenous boundary data remains an open question for most wave-models. 
In horizontal dimension $d=1$, the exterior domain ${\mathcal E}$ is the union of two half-lines and the equations 
\eqref{introSW2}--\eqref{introSW4} can be reduced to a transmission problem across the two connected components of ${\mathcal E}$ \cite{Lannes2017}. 
One can check that these transmission conditions satisfy the Kreiss--Lopatinski\u{\i} condition, 
from which the well-posedness can be established \cite{IguchiLannes}. 
In this one-dimensional setting, it is also possible to deal with an object having non-vertical sidewalls: 
this is more complicated because finding the location of $\mathit{\Gamma}$, which are reduced to two points, is then a free boundary problem. 
This was solved in \cite{IguchiLannes} where the object is also allowed to move freely under the action of the waves, 
and where a general theory for hyperbolic boundary value problems in dimension $d=1$ is provided. 
The proof was adapted numerically in \cite{Haidar}. 
In \cite{BHV}, this approach was used to simulate a wave-energy device, that is, the oscillating water column, 
and this one-dimensional approach has also been extended to cover radially symmetric configurations in two-dimensions \cite{Bocchi,Bocchi2}; 
let us also mention \cite{SuTucsnak} where controllability issues are investigated, and  \cite{MSTT} where a viscous fluid is considered. 

In dimension $d=1$, there are also some works dealing with other wave-models, 
and in particular for the Boussinesq equations which can be seen as a dispersive perturbation of the nonlinear shallow water equations. 
The presence of dispersion makes the analysis of the IBVP completely different, 
both theoretically \cite{BLM,BeckLannes} and numerically \cite{LannesWeynans}, with an important role played by dispersive boundary layers. 
For practical applications, most authors use formal approximations near the boundary to avoid this difficulty; 
see for instance \cite{BEER,MIS,Karambas,Gusev}. 

In order to avoid dealing with the IBVP for the wave model, 
some authors also proposed to relax the constraint on the surface elevation in the interior domain: 
the pressure term is then approximated by a pseudo-compressible relaxation for which efficient numerical tools exist \cite{GPSW1,GPSW2}. 
This latter approach highlights the fact that the kind of wave-structure interaction we consider here is a particular instance of {\it congested flow}. 
Such models appeared first as an asymptotic model for two-phase flows \cite{Bouchut} and are also relevant in biology \cite{Perthame}, 
collective dynamics \cite{Degond,Maury} or granular flows for instance; 
it is the maximal packing constraint at the particle level that induces the congestion constraint at the macroscopic level. 
In the literature devoted to congested flows, 
a distinction is often made \cite{Perrin} between soft congestion models where a single model of compressible type 
with a singular pressure law is used in the whole domain, 
and hard congestion models where the two phases are separated by a possibly free boundary. 
The model we study in this article fits therefore in the category of hard congestion models. 
Such models have generally been addressed from the point of view of weak solutions \cite{Berthelin} 
but strong solutions have also been considered more recently in \cite{Ancona,Dalibard1,Dalibard2}, 
where smooth solutions have been constructed in the neighborhood of explicit propagation fronts. 

Contrary to the previous references our present work deals with the case of strong solutions in horizontal dimension $d=2$. 
Even in the broader context of congested flows, this is to our knowledge the first mathematical result of this kind. 
One of the main obstructions in the two-dimensional case is the difficulty to handle the IBVP, 
even in the hyperbolic framework of the nonlinear shallow water equations. 
As we shall see in Section \ref{sectModelling}, the wave-structure interaction problem formed by \eqref{introSW2}--\eqref{introSW4} 
can be reduced to an IBVP for the nonlinear shallow water equations \eqref{introSW2} in $\cE$ 
with nonlocal in space and time boundary conditions on $\mathit{\Gamma}$ that do not fit in any category of boundary conditions 
for which quasilinear hyperbolic IBVP are well-posed. 
Contrary to the one-dimensional case which is well understood \cite{IguchiLannes}, the well-posedness of IBVP for quasilinear hyperbolic systems 
in general domains remains open in many cases, and even the well-posedness of the linearized equations raises important difficulties. 

The well-posedness in Sobolev spaces of linear hyperbolic IBVPs in certain classes of regular domains $\Omega$ is known 
when the boundary is non-characteristic, the differential operator is symmetrizable, 
and with strictly dissipative boundary conditions whose notions are defined in Section \ref{sectapriorilin}; 
even if the boundary conditions are not strictly dissipative, it is possible in certain cases, in particular, 
in the case where the so-called Kreiss--Lopatinski\u{\i} condition holds, 
to construct a Kreiss symmetrizer that reduces the problem to the case of strictly dissipative boundary conditions; 
see the works of Majda \cite{Majda1983}, M\'etivier \cite{Metivier}, and the book of Benzoni and Serre \cite{BenzoniSerre}. 

When the boundary is not non-characteristic, which is the case in this article, 
if the boundary condition is maximal dissipative and under the assumption that the rank of the boundary matrix is constant on the boundary, 
one can use energy methods as in the non-characteristic case for strictly dissipative solutions; 
however, in general, one has to work in conormal Sobolev spaces, 
and there is a loss of half a derivative in the energy estimates with respect to the boundary data. 
This is the reason why most studies assume homogeneous boundary data; 
we refer to the works of Secchi, and in particular to \cite{Secchi} and references therein, 
where the case of quasilinear systems is treated with linear homogeneous boundary conditions. 
In some particular cases, it is not necessary to work with conormal Sobolev spaces 
because some additional information can be obtained from the analysis of the vorticity or a similar quantity; 
this is for instance the case for the compressible Euler equations \cite{Schochet1986} with homogeneous boundary condition on the normal velocity, 
and for systems having a particular structure exhibited by Ohkubo \cite{Ohkubo}. 

When the boundary conditions are not maximally dissipative, 
it is natural to try to extend the construction of Kreiss symmetrizers to the characteristic case; 
this was done by Majda and Osher \cite{MajdaOsher} under the quite restrictive assumption that 
the boundary matrix is of constant rank in the vicinity of the boundary. 
In this situation, the failure of the uniform Kreiss--Lopatinski\u{\i} condition typically induces a loss of derivative in the energy estimate, 
as shown for instance by Coulombel and Secchi for $2D$ compressible vortex sheets \cite{Coulombel}, 
so that solving the nonlinear problem requires the use of a Nash--Moser iterative scheme \cite{Coulombel2008}. 

There are several reasons why we cannot use these results here. 
A first reason is that we have to deal with non-standard boundary conditions that are nonlocal in space and time; 
for such boundary conditions, the notion of Kreiss--Lopatnski\u{\i} condition is not very clear and 
several technical steps such as localization arguments to reduce to the case where the domain $\Omega$ is a half-space are not valid. 
A second reason is that the boundary conditions are fully nonlinear. 
A third reason is that the boundary matrix does not satisfy the constant rank assumption, 
and a fourth one is that the boundary conditions are not maximally dissipative. 

We are however able to transform the equations into a new system that satisfies the constant rank assumption, 
for which we exhibit a generalized vorticity that allows us to work with standard Sobolev spaces instead of conormal ones. 
We also exhibit a new kind of weak dissipativity: 
contrary to the strict or maximal dissipativity conditions, this new form of dissipativity holds only in integral form in space and time, 
but we show that it is sufficient to avoid derivative losses. 
The specific structure of the equations also allows us to quasilinearize the equation by time differentiation only, 
that is, we do not need for spatial differentiation, 
and a priori estimates allow us to obtain a well-posedness result at the critical quasilinear regularity $H^3$.

\subsection{Organization of this article}
Our first goal is to derive the wave-structure interaction problem \eqref{introSW2}--\eqref{introSW4} 
and to further reduce it to an IBVP in the exterior domain $\cE$. 
This is done in Section \ref{sectModelling}. 
The first thing to do is to derive the last two matching conditions in \eqref{introSW4}: 
we show that conservation of the total energy yields the continuity of the Bernoulli pressure across $\mathit{\Gamma}$ 
and then explain why the continuity of the tangential velocity is a natural condition to be imposed if we want that a flow preserves irrotationality. 
By solving the equations in the interior domain, we are then able to reduce the coupled compressible-incompressible system 
\eqref{introSW2}--\eqref{introSW4} into an IBVP for \eqref{introSW2} with boundary conditions that involve the Dirichlet-to-Neumann map 
associated with the interior domain. 

We then gather in Section \ref{sectprelim} some notations and technical results that will be used throughout the article, 
such as the construction of normal-tangential coordinates, the definition of function spaces, 
and some properties of the Dirichlet-to-Neumann map. 

We turn to the study of linear hyperbolic IBVP in Section \ref{sectapriorilin}. 
In this section, we begin to work with the case where the boundary is non-characteristic for a wide class of systems. 
The system that we have to deal with is characteristic and adaptations will be done in the latter part of this section. 
One of the objectives in this section is to insist on the notion of weak dissipativity. 
We show that this notion, that is weaker than the standard strict and maximal dissipativity conditions, 
is sufficient to derive   {higher} order energy estimate without derivative loss. 
Then, we work with a case where the boundary is not necessarily non-characteristic. 
We impose some structural conditions to the equations, which allow us to introduce a generalized vorticity. 
This generalized vorticity compensates the lack of the invertibility of the boundary matrix to evaluate normal derivatives of the solution. 

We then show in Section \ref{sectAPNLWS} that this notion of weak dissipativity can be used to obtain a priori estimates 
for our nonlinear wave-interaction problem. 
Since this system has been built to conserve the total energy, we get easily a control of the $L^2$-norm of the solution. 
In order to get higher order estimates, we need to obtain $L^2$-estimates for the linear IBVP satisfied by the derivatives of the solution. 
However, the boundary conditions for this linearized system do not seem to satisfy any kind of dissipativity 
even in the weak sense of Section \ref{sectapriorilin}; moreover, the boundary is not non-characteristic for this system, 
and the constant rank assumption mentioned in Section \ref{sectrelated} is not satisfied. 
Our strategy to answer these issues is first to use the irrotationality condition to transform the problem into another equivalent one, 
but for which the constant rank assumption is satisfied. 
In this new formulation, the boundary condition on the tangential velocity is also removed. 
We then remark that if we include lower order terms in the linearized equations, 
it is possible to perform a change of unknown that greatly simplifies the boundary condition and that makes it weakly dissipative. 
The only thing left to do is therefore to check that the transformed system satisfies the structural conditions introduced in Section \ref{sectapriorilin}. 
Gathering all these results, we obtain nonlinear a priori estimates at the quasilinear threshold of regularity. 

We then turn to prove the existence of solution. 
With start in Section \ref{sectexistlin} where we show the well-posedness of a class of linear IBVP 
that has the same structure as the linearized equations associated with the wave-interaction problem, 
after the change of unknown exhibited in Section \ref{sectAPNLWS} has been performed; 
in particular, the coefficients of the space derivatives are constant. 
We first propose a regularization of the problem that is non-characteristic and with strictly dissipative boundary conditions 
so that existence follows by standard results; 
a difficulty is to construct, starting from data that are compatible for the original system, 
some approximate data that are more regular and satisfy higher order compatibility conditions for the regularized system. 
This part is made quite delicate by the fact that the original problem is characteristic. 
Once this is done, it suffices to prove uniform estimates for the regularized system, and to use compactness arguments to pass to the limit. 
We must then check the continuity in time of the limit. 
During the whole process, we manage to lower the usual regularity requirements on the coefficients of the operator, 
a fact that proves crucial later to get critical regularity in the nonlinear well-posedness result. 

We can then prove the main result of this article in Section \ref{sectmain}. 
A difficulty is that while the nonlinear shallow water equations are quasilinear, the boundary conditions are {\it fully} nonlinear. 
We must therefore reduce the problem to a system of quasilinear equations. 
Using the specific structure of the equations, 
it is possible to introduce a set of new unknowns that involve only the original unknown and its time derivative, 
and that satisfies a system with is quasilinear. 
We can use the linear existence result of Section \ref{subsectexistlin} to build a standard Picard scheme on this extended system. 
Since this procedure requires additional regularity on the data, we also need to approximate the data by a sequence of more regular data 
satisfying higher order compatibility conditions, from which we get a sequence of approximate solutions. 
We can then use the a priori estimate of Section \ref{sectAPNLWS} to pass to the limit and get a solution that has sharp regularity.

\medskip
\noindent
{\bf Acknowledgement} \\
T. I. was partially supported by JSPS KAKENHI Grant Number JP17K18742 and JP22H01133, 
and D. L. was partially supported by the ANR-18-CE40-0027 Singflows.

\section{Modeling}\label{sectModelling}
The goal of this section is to reduce the whole wave-structure interaction problem described in the introduction 
to an initial boundary value problem for the nonlinear shallow water equations in the exterior domain. 
In Section \ref{subsectSWM}, we present the set of equations in the exterior and interior domains, 
as well as a continuity condition on the normal mass flux at the interface. 
In order to close the system of equations, an additional condition is needed at the interface; 
we show in Section \ref{constotNRJ} that imposing the conservation of the total energy leads to a continuity condition 
for a quantity that we name Bernoulli pressure. 
Even with this additional condition, we show that one cannot expect uniqueness of the solution. 
A way to discard this obstruction to uniqueness is to impose irrotationality of the solution; 
as shown in Section \ref{subsectsup} this leads to impose continuity of the tangential velocity at the interface. 
We then show in Section \ref{sectanalvort} that this irrotationality assumption is reasonable 
inasmuch as it is propagated from the initial condition; 
this is not a trivial point because the vorticity equation is not non-characteristic in our setting. 
We are then in a position to reduce the wave-structure interaction problem to an initial boundary value problem in the exterior domain; 
this is done in Section \ref{subsectreduc}; 
note that the boundary conditions for this problem are not standard, 
as they are nonlocal and involve the Dirichlet-to-Neumann map associated with the interior domain.

\subsection{Shallow water model}\label{subsectSWM}
We recall that we investigate the interactions between waves of the water and a partially immersed solid body which is placed on the water surface. 
We assume that the solid body is fixed and has a vertical sidewall whose projection on the horizontal plane 
is a smooth Jordan curve $\mathit{\Gamma}$ of length $L$; see Figure \ref{fig:SideWalls}. 
We call interior region and denote by $\cI$ the interior of the domain delimited by this curve $\mathit{\Gamma}$, 
and similarly call exterior region and denote by $\cE$ the outer domain delimited by this curve. 
We also denote by $N$ the unit normal vector on $\mathit{\Gamma}$ pointing from $\cI$ to $\cE$.

We suppose also that the water surface at time $t$, the bottom of the wetted part of the solid body, and the bottom of the fluid layer are 
represented as $z=\zeta(t,x)$ for $x\in\cE$, $z=\zeta_{\rm w}(x)$ for $x\in\cI$, and $z=-H_0$ with a positive constant $H_0$, respectively. 
Therefore, the depth of the water under the water surface and under the solid body are given by 
$h(t,x)=H_0+\zeta(t,x)$ and $h_{\rm i}(x)=H_0+\zeta_{\rm w}(x)$, respectively. 
We denote by $v(t,x)$ and $v_{\rm i}(t,x)$ the vertically averaged horizontal velocity fields under the water surface and the solid body, respectively. 
Then, the nonlinear shallow water model proposed in \cite{Lannes2017} consists of the nonlinear shallow water equations in the exterior region 
\begin{equation}\label{SWEinE}
\begin{cases}
\dt\zeta + \nabla\cdot(hv) = 0 &\mbox{in}\quad (0,T)\times\cE, \\
\dt v + (v\cdot\nabla)v + \gr\nabla\zeta = 0 &\mbox{in}\quad (0,T)\times\cE,
\end{cases}
\end{equation}
the nonlinear shallow water equations in the interior region 
\begin{equation}\label{SWEinI}
\begin{cases}
\nabla\cdot(h_{\rm i}v_{\rm i}) = 0 &\mbox{in}\quad (0,T)\times\cI, \\
\dt v_{\rm i} + (v_{\rm i}\cdot\nabla)v_{\rm i} + \gr\nabla\zeta_{\rm w}
 + \frac{1}{\rho}\nabla\underline{P}_{\rm i} = 0 &\mbox{in}\quad (0,T)\times\cI,
\end{cases}
\end{equation}
and the matching condition on the curve $\mathit{\Gamma}$
\begin{equation}\label{MConM}
N\cdot(hv) = N\cdot(h_{\rm i}v_{\rm i}) \quad\mbox{on}\quad (0,T)\times\mathit{\Gamma},
\end{equation}
where $\gr$ is the gravitational constant, $\rho$ is the constant density of the water, and 
$\underline{P}_{\rm i}(t,x)$ is the pressure measured on the bottom of the solid body. 
This matching condition represents the continuity of mass flux across the interface $\mathit{\Gamma}$. 
Here, we note that this matching condition is not enough to close the equations and to ensure the local well-posedness of the initial value problem. 
In particular, in order to determine the pressure $\underline{P}_{\rm i}$ from the interior equations \eqref{SWEinI}, 
boundary data must be provided at the interface $\mathit{\Gamma}$. 
Such boundary conditions for the pressure can be obtained, as in the one-dimensional case \cite{MSTT,Bocchi2020,BLM}, 
by imposing the conservation of the total energy.

\subsection{Conservation of the total energy}\label{constotNRJ}
Since the object is fixed, the conservation of the total energy of the system, that is, the water and the solid, is 
equivalent to the conservation of the mechanical energy, that is, the sum of the kinetic and the potential energies of the fluid, 
denoted by $\mathfrak{E}_\mathrm{fluid}(t)$ at time $t$. 
The mechanical energy $\mathfrak{E}_\mathrm{fluid}(t)$ is the sum of the mechanical energy of the water below the water surface 
and below the fixed solid structure 
\[
\mathfrak{E}_\mathrm{fluid}(t) = \int_\cE\mathfrak{e}(t,\cdot) + \int_\cI\mathfrak{e}_{\rm i}(t,\cdot),
\]
where the densities of the energies $\mathfrak{e}$ and $\mathfrak{e}_{\rm i}$ are given by 
\[
\mathfrak{e} = \tfrac12\rho h\abs{v}^2 + \tfrac12\rho\gr\zeta^2, \qquad
\mathfrak{e}_{\rm i} = \tfrac12\rho h_{\rm i}\abs{v_{\rm i}}^2+\tfrac12\rho\gr\zeta_{\rm w}^2.
\]
Now, suppose that $(\zeta,v)$ and $v_{\rm i}$ satisfy the nonlinear shallow water model \eqref{SWEinE}--\eqref{MConM}. 
It follows from the nonlinear shallow water equations \eqref{SWEinE} and \eqref{SWEinI} that 
\[
\begin{cases}
 \dt\mathfrak{e} + \nabla\cdot {\mathfrak F}= 0 &\mbox{in}\quad (0,T)\times\cE, \\
 \dt\mathfrak{e}_{\rm i} + \nabla\cdot {\mathfrak F}_{\rm i}=0,
  &\mbox{in}\quad (0,T)\times\cI, 
\end{cases}
\]
where ${\mathfrak F}$ and ${\mathfrak F}_{\rm i}$ are the energy fluxes in the exterior and interior regions respectively, that is, 
\[
{\mathfrak F}=\Pi hv,\qquad
{\mathfrak F}_{\rm i}= \Pi_{\rm i}h_{\rm i}v_{\rm i},
\]
where $\Pi$ and $\Pi_{\rm i}$ denote the Bernoulli pressures in the exterior and interior domains respectively and are defined as 
\begin{equation}\label{defiPiPii}
\Pi=\rho\big(\gr\zeta+\tfrac12\abs{v}^2\big),\qquad 
\Pi_{\rm i}=\underline{P}_{\rm i}-P_{\rm atm}+\rho\big(\gr\zeta_{\rm w}+\tfrac12\abs{v_{\rm i}}^2\big),
\end{equation}
where $P_{\rm atm}$ is the atmospheric pressure, which is assumed to be constant. 
These equations together with the matching condition \eqref{MConM} yield that the conservation of the energy 
$\frac{\rm d}{{\rm d}t}\mathfrak{E}_\mathrm{fluid}(t)=0$ is equivalent to 
\[
\int_\mathit{\Gamma}  \big( \Pi-\Pi_{\rm i}\big)hN\cdot v = 0.
\]
Taking this into account 
we impose additionally the continuity of the Bernoulli   {pressures} across $\mathit{\Gamma}$ 
\begin{equation}\label{MConP}
\Pi= \Pi_{\rm i}
 \quad\mbox{on}\quad (0,T)\times\mathit{\Gamma},
\end{equation}
under which we have the conservation of the total energy.

\begin{remark}
If the pressure were hydrostatic in the fluid, then the boundary condition for $\underline{P}_{\rm i}$ should be 
\[
\underline{P}_{\rm i} = P_{\rm atm} + \rho\gr\zeta-\rho\gr\zeta_{\rm w}
 \quad\mbox{on}\quad (0,T)\times\mathit{\Gamma};
\]
this differs from \eqref{MConP} by nonlinear terms that account for the hydrodynamical pressure and are necessary to ensure total energy conservation. 
\end{remark}

\subsection{A supplemental matching condition}\label{subsectsup}
The matching conditions \eqref{MConM} and \eqref{MConP} are still underdetermined, that is, 
under these matching conditions the nonlinear shallow water equations \eqref{SWEinE} and \eqref{SWEinI} with the initial conditions 
\[
(\zeta,v)_{\vert_{t=0}}=(\zeta^{\rm in},v^{\rm in}) \quad\mbox{in}\quad \cE,
\]
cannot determine uniquely the solution $(\zeta,v,v_{\rm i})$. 
In fact, we have the following simple example of multiple solutions.

Let us consider the case where $\mathit{\Gamma}$ is the unit circle $\abs{x}=1$, the bottom of the solid body is flat so that 
$\zeta_{\rm w}$ is a negative constant, and the initial data are identically zero $(\zeta^{\rm in},v^{\rm in})=0$. 
Then, we can construct solutions $(\zeta,v,v_{\rm i})$ satisfying $(\zeta,v)=0$ and $v_{\rm i}\ne0$ as follows. 
In this case, $v_{\rm i}$ should satisfy 
\begin{equation}\label{eqsuppl1}
\begin{cases}
\nabla\cdot v_{\rm i} = 0 &\mbox{in}\quad (0,T)\times\cI, \\
\dt v_{\rm i} + (v_{\rm i}\cdot\nabla)v_{\rm i} 
 + \frac{1}{\rho}\nabla\underline{P}_{\rm i} = 0 &\mbox{in}\quad (0,T)\times\cI,
\end{cases}
\end{equation}
with the boundary conditions 
\begin{equation}\label{eqsuppl2}
\begin{cases}
 N\cdot v_{\rm i}=0 &\mbox{on}\quad (0,T)\times\partial\cI, \\
 \underline{P}_{\rm i} = P_{\rm atm} -\rho(\tfrac12\abs{v_{\rm i}}^2+\gr\zeta_{\rm w})
  &\mbox{on}\quad (0,T)\times\partial\cI.
\end{cases}
\end{equation}
It is easy to check that this problem has a one-parameter family of solutions 
\[
\begin{cases}
 v_{\rm i}(x)=ax^\perp, \\
 \underline{P}_{\rm i}(x)=P_{\rm atm}-\rho(\frac12a^2(2-\abs{x}^2)+\gr\zeta_{\rm w})
\end{cases}
\]
with a parameter $a\in\R$, where we write $v^\perp=(-v_2,v_1)^{\rm T}$ for a vector $v=(v_1,v_2)^{\rm T}$.

One of the ways to exclude these multiple solutions is to impose an initial condition on $v_{\rm i}$ 
since the equations \eqref{eqsuppl1} for $(v_{\rm i},\underline{P}_{\rm i})$ are nothing but the incompressible Euler equations. 
The initial boundary value problem associated with these equations in $\cI$ and Neumann boundary condition 
$N\cdot v_{\rm i}=0$ on $\mathit{\Gamma}$ is well posed, in particular, 
the value of $\underline{P}_{\rm i}$ is determined by the initial data up to an additive constant and cannot be imposed as the second 
condition in \eqref{eqsuppl2}. 
In other words, the initial value problem associated with \eqref{eqsuppl1} and \eqref{eqsuppl2} is overdetermined.

Another way to exclude these multiple solutions is to impose the irrotationality of $v_{\rm i}$, namely, that $\nabla^\perp\cdot v_{\rm i}=0$ in $\cI$. 
We also impose irrotationality of $v$ in the exterior domain $\cE$, and to avoid the presence of a vortex line on $\mathit{\Gamma}$, 
we impose continuity of the tangential velocity across $\mathit{\Gamma}$, that is, 
\begin{equation}\label{MConV}
N^\perp\cdot v = N^\perp\cdot v_{\rm i} \quad\mbox{on}\quad (0,T)\times\mathit{\Gamma}
\end{equation}
as a supplemental matching condition. 
Under this condition, we can exclude the above multiple solutions; the parameter $a$ must be equal to zero. 
Moreover, we show in the next section that with this matching condition, irrotationality is propagated from the initial data.

\subsection{Analysis of the vorticity}\label{sectanalvort}
We check here that the irrotationality of $v_{\rm i}$ in $\cI$ and of $v$ in $\cE$ is preserved under the time evolution 
by the nonlinear shallow water model \eqref{SWEinE}--\eqref{MConP} and \eqref{MConV}. 
To this end, suppose that $(\zeta,v,v_{\rm i})$ is a solution to the model and put 
\[
\omega=\nabla^\perp\cdot v, \qquad \omega_{\rm i}=\nabla^\perp\cdot v_{\rm i}.
\]
Then, the second equation in the nonlinear shallow water equations \eqref{SWEinE} and \eqref{SWEinI} can be written as 
\begin{equation}\label{MomEqs}
\begin{cases}
 \dt v + \omega v^\perp + \frac{1}{\rho}\nabla\Pi=0 &\mbox{in}\quad (0,T)\times\cE, \\
 \dt v_{\rm i} + \omega_{\rm i} v_{\rm i}^\perp + \frac{1}{\rho}\nabla\Pi_{\rm i}=0 &\mbox{in}\quad (0,T)\times\cI, 
\end{cases}
\end{equation}
respectively, where the Bernoulli pressures $\Pi$ and $\Pi_{\rm i}$ are defined as in \eqref{defiPiPii}. 
Applying the curl operator $\nabla^\perp\cdot$ to these equations, we have 
\[
\begin{cases}
 \dt\omega + \nabla\cdot(\omega v)=0 &\mbox{in}\quad (0,T)\times\cE, \\
 \dt\omega_{\rm i} + \nabla\cdot(\omega_{\rm i} v_{\rm i})=0 &\mbox{in}\quad (0,T)\times\cI,
\end{cases}
\]
which together with the first equations in the nonlinear shallow water equations \eqref{SWEinE} and \eqref{SWEinI} yield 
\[
\begin{cases}
 \dt\bigl(\frac{\omega}{h}\bigr) + v\cdot\nabla\bigl(\frac{\omega}{h}\bigr)=0 &\mbox{in}\quad (0,T)\times\cE, \\
 \dt\bigl(\frac{\omega_{\rm i}}{h_{\rm i}}\bigr)
  + v_{\rm i}\cdot\nabla\bigl(\frac{\omega_{\rm i}}{h_{\rm i}}\bigr)=0 &\mbox{in}\quad (0,T)\times\cI,
\end{cases}
\]
so that 
\[
\begin{cases}
 \dt\bigl(\frac{\omega^2}{h}\bigr) + \nabla\cdot\bigl(\frac{\omega^2}{h}v\bigr)=0 &\mbox{in}\quad (0,T)\times\cE, \\
 \dt\bigl(\frac{\omega_{\rm i}^2}{h_{\rm i}}\bigr)
  + \nabla\cdot\bigl(\frac{\omega_{\rm i}^2}{h_{\rm i}}v_{\rm i}\bigr)=0 &\mbox{in}\quad (0,T)\times\cI.
\end{cases}
\]
On the other hand, taking the inner products of the equations in \eqref{MomEqs} with $N^\perp$ and using the matching 
conditions \eqref{MConP} and \eqref{MConV} we have 
\[
\omega N\cdot v = \omega_{\rm i}N\cdot v_{\rm i} \quad\mbox{on}\quad (0,T)\times\mathit{\Gamma}.
\]
Therefore, we see that 
\begin{align*}
\frac{\rm d}{{\rm d}t}
 \left( \int_\cE\frac{\omega^2}{h} + \int_\cI\frac{\omega_{\rm i}^2}{h_{\rm i}} \right)
&= \int_{\mathit{\Gamma}}\left( \frac{\omega^2}{h}N\cdot v - \frac{\omega_{\rm i}^2}{h_{\rm i}}N\cdot v_{\rm i} \right) \\
&= \int_{\mathit{\Gamma}}\left( \frac{\omega}{h} + \frac{\omega_{\rm i}}{h_{\rm i}} \right)
 (\omega N\cdot v - \omega_{\rm i}N\cdot v_{\rm i}) \\
&= 0,
\end{align*}
where we used the matching condition \eqref{MConM} in the second equality. 
This shows that if $\omega(0,\cdot)=0$ and $\omega_{\rm i}(0,\cdot)=0$, then $\omega(t,\cdot)=0$ and $\omega_{\rm i}(t,\cdot)=0$ for all $t$. 
In other words, irrotationality in $\cI$ and $\cE$ is preserved under the time evolution by the nonlinear shallow water model 
\eqref{SWEinE}--\eqref{MConP} and \eqref{MConV}. 
In the following of this article, we will always consider irrotational initial data so that the following conditions hold. 
\begin{equation}\label{Irrotational}
\begin{cases}
 \nabla^\perp\cdot v = 0 &\mbox{in}\quad (0,T)\times\cE, \\
 \nabla^\perp\cdot v_{\rm i} = 0 &\mbox{in}\quad (0,T)\times\cI.
\end{cases}
\end{equation}

\subsection{Reduction to an initial boundary value problem}\label{subsectreduc}
Suppose that $(\zeta,v,v_{\rm i})$ is a solution to the nonlinear shallow water model \eqref{SWEinE}--\eqref{MConP} and \eqref{MConV} 
and satisfies the irrotationality conditions \eqref{Irrotational}. 
Since $v_{\rm i}(t,\cdot)$ is irrotational in $\cI$ which is simply connected, 
$v_{\rm i}(t,\cdot)$ has a single valued velocity potential $\phi_{\rm i}(t,\cdot)$, that is, $v_{\rm i}=\nabla\phi_{\rm i}$. 
Even though the exterior region $\cE$ is not simply connected, we have 
\begin{align*}
\int_\mathit{\Gamma}N^\perp\cdot v
 = \int_\mathit{\Gamma}N^\perp\cdot v_{\rm i}
 = \int_\cI\nabla^\perp\cdot v_{\rm i}
 =0,
\end{align*}
so that $v(t,\cdot)$ has also a single valued velocity potential $\phi(t,\cdot)$, that is, $v=\nabla\phi$. 
Therefore, by choosing appropriately additive functions of $t$ to the velocity potentials $\phi$ and $\phi_{\rm i}$ 
we can rewrite the second equations in the nonlinear shallow water equations \eqref{SWEinE} and \eqref{SWEinI}, or equivalently \eqref{MomEqs}, as 
\[
\begin{cases}
 \dt\phi+\frac{1}{\rho}\Pi=0 &\mbox{in}\quad (0,T)\times\cE, \\
 \dt\phi+\frac{1}{\rho}\Pi_{\rm i}=0 &\mbox{in}\quad (0,T)\times\cI,
\end{cases}
\]
which together with the matching condition \eqref{MConP} yield 
\[
\dt\phi=\dt\phi_{\rm i} \quad\mbox{on}\quad (0,T)\times\mathit{\Gamma}.
\]
On the other hand, it follows directly from the matching condition \eqref{MConV} that 
\[
N^\perp\cdot\nabla\phi=N^\perp\cdot\nabla\phi_{\rm i} \quad\mbox{on}\quad (0,T)\times\mathit{\Gamma}.
\]
Since $N^\perp\cdot\nabla$ is a tangential derivative on the curve $\mathit{\Gamma}$, 
these equations imply that the function $\phi-\phi_{\rm i}$ is constant on $(0,T)\times \mathit{\Gamma}$. 
Therefore, by adding an appropriate constant to $\phi$ or $\phi_{\rm i}$ we may assume without loss of generality that $\phi$ and $\phi_{\rm i}$ 
coincide on $\mathit{\Gamma}$, and we denote by $\psi_{\rm i}$ their common trace, that is, 
\[
\psi_{\rm i}:=\phi=\phi_{\rm i} \quad\mbox{on}\quad (0,T)\times\mathit{\Gamma}.
\]
Taking the trace on $\mathit{\Gamma}$ of the equation for $\phi$, and recalling that $v=\nabla\phi$, 
we see that $\psi_{\rm i}$ is determined by an ODE in time forced by the traces on $\mathit{\Gamma}$ of $v$ and $\zeta$, that is, 
\[
\dt \psi_{\rm i}+\gr \zeta+\tfrac{1}{2}\abs{v}^2=0 \quad\mbox{on}\quad (0,T)\times\mathit{\Gamma}. 
\]
The knowledge of $\psi_{\rm i}$ is enough to reconstruct the velocity $v_{\rm i}=\nabla\phi_{\rm i}$ in $\cI$. 
Indeed, the first equation in \eqref{SWEinI} can be restated as 
\begin{equation}\label{EEinI1}
\nabla\cdot(h_{\rm i}\nabla\phi_{\rm i}) = 0\quad \mbox{ in}\quad (0,T)\times\cI, 
\end{equation}
which, for all time $t\in (0,T)$, is an elliptic equation that can be solved with the boundary condition $\phi_{\rm i}=\psi_{\rm i}$ on $\mathit{\Gamma}$. 
It is also convenient to introduce the Dirichlet-to-Neumann map $\Lambda$ as 
\[
\Lambda \psi_{\rm i} := \bigl( N\cdot(h_{\rm i}\nabla\phi_{\rm i}) \bigr)_{\vert_\mathit{\Gamma}};
\]
see Definition \ref{defDN} below for a precise definition of the map $\Lambda$. 
The matching conditions \eqref{MConM} and \eqref{MConV} can therefore be understood as non-homogeneous boundary conditions 
on the velocity $v$ in the exterior region that are expressed in terms of $\psi_{\rm i}$, that is, 
\[
\begin{cases}
 N\cdot(hv) = \Lambda \psi_{\rm i} &\mbox{on}\quad (0,T)\times\mathit{\Gamma}, \\
 N^\perp\cdot v = \dtan\psi_{\rm i} &\mbox{on}\quad (0,T)\times\mathit{\Gamma},
\end{cases}
\]
where $\dtan=N^\perp\cdot \nabla$ is a tangential derivative on $\mathit{\Gamma}$; see Section \ref{sectnoprmtang} below.

To summarize the above considerations, the nonlinear wave-structure interaction problem formed by \eqref{SWEinE}--\eqref{MConP} and \eqref{MConV} 
reduces to the nonlinear shallow water equations in $\cE$, 
\begin{equation}\label{PB1}
\begin{cases}
 \dt \zeta + \nabla\cdot (hv) = 0 &\mbox{in}\quad (0,T)\times\cE, \\
 \dt v + v\cdot \nabla v + \gr\nabla\zeta = 0 &\mbox{in}\quad (0,T)\times\cE
\end{cases}
\end{equation}
with irrotationality condition
\begin{equation}\label{PBirrot}
\nabla^\perp\cdot v=0 \quad\mbox{in}\quad (0,T)\times\cE,
\end{equation}
and boundary conditions 
\begin{equation}\label{PB2}
\begin{cases}
 N\cdot(hv) = \Lambda \psi_{\rm i} &\mbox{on}\quad (0,T)\times\mathit{\Gamma}, \\
 N^\perp\cdot v = \dtan\psi_{\rm i} &\mbox{on}\quad (0,T)\times\mathit{\Gamma},
\end{cases}
\end{equation}
where $\psi_{\rm i}$ is found by solving the forced ODE on $\mathit{\Gamma}$ 
\begin{equation}\label{PB3}
\dt\psi_{\rm i} = - \gr \zeta-\tfrac{1}{2}\abs{v}^2  \quad\mbox{on}\quad (0,T)\times\mathit{\Gamma}.
\end{equation}

\begin{remark}
We note that in this formulation of the nonlinear shallow water model the velocity $v_{\rm i}$ 
and the pressure $\underline{P}_{\rm i}$ have been eliminated. 
However, as explained above $v_{\rm i}=\nabla\phi_{\rm i}$ can be deduced from $\psi_{\rm i}$ by solving an elliptic boundary value problem 
for $\phi_{\rm i}$, while the pressure can be recovered by 
\[
\underline{P}_{\rm i}
 = P_{\rm atm} - \rho\bigl(\dt\phi_{\rm i} + \tfrac12\abs{\nabla\phi_{\rm i}}^2 + \gr\zeta_{\rm w}\bigr). 
\]
\end{remark}

\begin{remark}
In the one-dimensional case considered in \cite{Lannes2017,MSTT,BHV}, the exterior domain consists of two half lines, 
$\cE=(-\infty,\ell)\cup(\ell,+\infty)$ for some $\ell>0$, 
and the wave-structure interaction problem reduces to a transmission problem between the two connected components of $\cE$ 
\begin{equation}\label{PB1-1d}
\begin{cases}
 \dt \zeta+\dx (h v)=0 &\mbox{ in}\quad (0,T)\times ((-\infty,\ell)\cup(\ell,\infty)), \\
 \dt v+v\dx v +\gr\dx\zeta=0 &\mbox{ in}\quad (0,T)\times ((-\infty,\ell)\cup(\ell,\infty))
\end{cases}
\end{equation}
with boundary conditions
\begin{equation}\label{PB2-1d}
(h v)_{\vert_{x=\pm\ell}} = q_{\rm i} \quad\mbox{on}\quad (0,T)\times\{x=\pm\ell\},
\end{equation}
where $q_{\rm i}=q_{\rm i}(t)$ is found by solving the forced ODE 
\begin{equation}\label{PB3-1d}
\alpha\frac{\rm d}{{\rm d}t}q_{\rm i}=- \jump{\gr\zeta+\tfrac{1}{2}{v}^2}  \quad\mbox{in}\quad (0,T),
\end{equation}
with the notation $\jump{f}=f(\ell)-f(-\ell)$, and with $\alpha=\int_{-\ell}^\ell \frac{1}{h_{\rm i}}$. 
In order to see that \eqref{PB1}--\eqref{PB3} is a natural generalization of \eqref{PB1-1d}--\eqref{PB3-1d}, 
we can rewrite the boundary conditions \eqref{PB2-1d}--\eqref{PB3-1d} as 
\[
\pm(h v)_{\vert_{x=\pm\ell}} = \pm\frac{1}{\alpha}(\psi_{+}-\psi_-)
\quad \mbox{ with }\quad
\frac{\rm d}{{\rm d}t}\psi_\pm=-(\gr \zeta+\tfrac{1}{2}v^2)_{\vert_{x=\pm\ell}}  \quad\mbox{in}\quad (0,T).
\]
The local well-posedness of the initial value problem to this system is a consequence of the general results for 
$1d$ nonlinear hyperbolic initial boundary value problems of \cite{IguchiLannes}. 
\end{remark}

\section{Notations and preliminary results}\label{sectprelim}
We introduce in Section \ref{sectnota} some notations that are used throughout this article. 
In Section \ref{sectnoprmtang}, we construct normal-tangential coordinates in the neighborhood of $\mathit{\Gamma}$ 
and then define in Section \ref{sectfunctionspace} the functional spaces that we shall need. 
Finally, we introduce in Section \ref{sectDN} the Dirichlet-to-Neumann map and related objects, and prove some of their most important properties.

\subsection{General notations} \label{sectnota}
\begin{itemize}[leftmargin=\parindent]
\item[--]
We denote generically by $C$ a constant of no importance, whose value may differ from one line to another. 
\item[--]
We write $f\lesssim g$ if $f\leq C g$. 
\item[--]
For a two-dimensional vector $v=(v_1,v_2)^{\rm T}$, we write $v^\perp=(-v_2,v_1)^{\rm T}$. 
\item[--]
We denote by $x=(x_1,x_2)$ the horizontal coordinates and by $\partial_1$ and $\partial_2$ partial derivation with respect to $x_1$ and $x_2$. 
\item[--]
We denote by ${\bf e_1}$ and ${\bf e}_2$ the unit vectors of $(0x_1)$ and $(0x_2)$ respectively. 
\item[--]
If we say that a quantity depends on $\boldsymbol{\partial} u$, this means that it depends on $\dt u$, $\partial_1 u$, and $\partial_2 u$. 
\item[--]
We use Einstein's convention on repeated indices, so that $a_jb_j=a_1b_1+a_2b_2$. 
\item[--]
If $u=(u_1,u_2,u_3)^{\rm T}\in \R^3$, we write $u^{\rm I}=u_1$ and $u^{\rm II}=(u_2,u_3)^{\rm T}$.
\end{itemize}

\subsection{Normal and tangential coordinates}\label{sectnoprmtang}
We assume throughout this article that $\mathit{\Gamma}$, the projection of the contact line on the horizontal plane, 
is a positively oriented Jordan curve of $C^\infty$-class. 
To fix the ideas, we suppose that $\mathit{\Gamma}$ is parametrized by the arc length $s$ as 
$x=\bm{\gamma}(s)=({\gamma}_1(s),{\gamma}_2(s))^{\rm T}$ for $0\leq s<L$. 
As usual, we can regard $x=\bm{\gamma}(s)$ as a function of $s\in\mathbb{T}_L\simeq\R/L$. 
Then, it holds that $\abs{\bm{\gamma}'(s)} \equiv 1$ and that the unit normal vector $\bm{n}(s)$ at $\bm{\gamma}(s)$ 
to the curve $\mathit{\Gamma}$ pointing towards the exterior is given by 
\[
\bm{n}(s) = -\bm{\gamma}'(s)^\perp = ({\gamma}_2'(s),-{\gamma}_1'(s))^{\rm T}. 
\]
We also denote by ${\cI}$ the region enclosed by $\mathit{\Gamma}$, and by ${\cE}$ the outer region.

Near the boundary $\mathit{\Gamma}$ of the exterior and interior domains $\cE$ and $\cI$, 
it is convenient to work with normal-tangential coordinates rather than the cartesian ones. 
Recalling that $\mathit{\Gamma}$ is parametrized by its arc length $s$ through an $L$-periodic function $\bm{\gamma}$, 
we have that $\bm{\gamma}''(s)=\kappa(s)\bm{\gamma}'(s)^\perp=-\kappa(s)\bm{n}(s)$, where $\kappa(s)$ 
is the scalar curvature of the curve $\mathit{\Gamma}$ at $\bm{\gamma}(s)$. 
For $r_0>0$ we define a map $\theta: (-r_0,r_0)\times\mathbb{T}_L\ni(r,s) \mapsto \theta(r,s)\in\R^2$ by 
\begin{equation}\label{diffeo-theta}
\theta(r,s) = \bm{\gamma}(s) + r\bm{n}(s) = 
\begin{pmatrix}
 {\gamma}_1(s) + r{\gamma}_2'(s) \\
 {\gamma}_2(s) - r {\gamma}_1'(s)
\end{pmatrix}.
\end{equation}
Then, it holds that 
\[
\det\left( \frac{\partial\theta(r,s)}{\partial(r,s)} \right) = 1+r\kappa(s).
\]
Therefore, if we take $r_0>0$ so small that $r_0 \abs{\kappa}_{L^\infty(\mathbb{T}_L)}<1$, 
and define a tubular neighborhood $U_{\mathit{\Gamma}}$ of $\mathit{\Gamma}$ by 
\begin{equation}\label{defUGamma}
U_{\mathit{\Gamma}} = \{ x=\theta(r,s)\in\R^2 \,|\, (r,s)\in (-r_0,r_0)\times\mathbb{T}_L \},
\end{equation}
then the map 
$\theta: (-r_0,r_0)\times\mathbb{T}_L \to U_{\mathit{\Gamma}}$ is a $C^\infty$-diffeomorphism. 
Each point of $U_{\mathit{\Gamma}}$ can therefore be uniquely determined by its normal-tangential coordinates $(r,s)$. 
We will also repeatedly use a cutoff function $\chi_{\rm b}$ with support in $U_{\mathit{\Gamma}}$, equal to $1$ 
in a neighborhood of $\mathit{\Gamma}$ and depending only on the normal variable; such a function can be defined on $U_{\mathit{\Gamma}}$ as 
\begin{equation}\label{defchib}
(\chi_{\rm b}\circ \theta)(r,s)=\chi(r),
\end{equation}
where $\chi$ is a smooth cutoff function with support in $(-r_0,r_0)$ and equal to $1$ in the neighborhood of the origin; 
it is extended by $0$ in $\R^2\backslash U_{\mathit{\Gamma}}$.

Associated with these normal-tangential coordinates, 
we can define normal and tangential derivatives of functions $f$ defined in the neighborhood $\mathit{\Gamma}$ by 
\[
(\dnor f)\circ\theta = \partial_r(f\circ\theta), \quad
(\dtan f)\circ\theta = \partial_s(f\circ\theta). 
\]
Then, we have 
\[
\dnor = {N}\cdot\nabla,\quad
\dtan = {T}\cdot\nabla,
 \]
where ${T}\circ\theta = (1+r\kappa(s))\bm{\gamma}'(s)$ and ${N}\circ\theta = \bm{n}(s) = (-\bm{\gamma}'(s))^\perp$. 
We note that $\dnor$ and $\dtan$ commute with each other. 
Conversely, we have a decomposition 
\begin{equation}\label{diff deco}
\nabla = {N}\dnor + \frac{1}{\abs{T}^2}{T}\dtan.
\end{equation}
For a vector field $\bm{f}$ defined in the neighborhood $\mathit{\Gamma}$ we see that 
\begin{align*}
(\nabla\cdot\bm{f})\circ\theta
&= \tilde{J}^{-1}\nabla_{r,s}\cdot (\tilde{J}(\partial\theta)^{-1}\bm{f}\circ\theta) \\
&= \tilde{J}^{-1}\{ \partial_r(\tilde{J}(N\cdot\bm{f})\circ\theta)) + \partial_s(\tilde{J}^{-1}(T\cdot\bm{f})\circ\theta) \},
\end{align*}
where $\partial\theta$ stands for the Jacobian matrix of the map $\theta$ and $\tilde{J}=\det(\partial\theta)=1+r\kappa(s)$. 
This implies that 
\begin{equation}\label{diff exp1}
\nabla\cdot\bm{f}
= J^{-1}\{ \dnor(JN\cdot\bm{f}) + \dtan(J^{-1}T\cdot\bm{f})\},
\end{equation}
where $J\circ\theta=\tilde{J}$. 
Moreover, putting $\bm{f}=J^{-1}Nf$ and $J^{-1}Tf$ in this formula and noting $\abs{T}=J$ we have also 
\begin{equation}\label{diff exp2}
\begin{cases}
 J^{-1}\dnor f = \nabla\cdot(J^{-1}Nf), \\
 J^{-1}\dtan f = \nabla\cdot(J^{-1}Tf).
\end{cases}
\end{equation}

\subsection{Function spaces}\label{sectfunctionspace}
We introduce function spaces for functions of space, of time, and of space and time.

\subsubsection{Norms in space}
We denote by $L^p(\cE)$, $L^p(\cI)$, and $L^p(\mathit{\Gamma})$ with $1\leq p\leq \infty$, 
the standard Lebesgue spaces on $\cE$, $\cI$, and $\mathit{\Gamma}$, respectively. 
Double bars are used to denote norms on the two-dimensional domains $\cE$ or $\cI$ and simple bars on the curve $\mathit{\Gamma}$, for instance, 
\[
\Abs{u}_{L^2(\Omega)}=\left(\int_\Omega \abs{u(x)}^2 {\rm d}x \right)^{1/2}, \qquad
\abs{g}_{L^2(\mathit{\Gamma})}=\left(\int_\mathit{\Gamma} \abs{f(x)}^2 {\rm d} \mathit{\Gamma}_x \right)^{1/2}
\]
with $\Omega=\cE$ or $\cI$. 
When no confusion is possible on the domain of integration, we write $\Abs{u}_{L^p}$ or even $\Abs{u}_p$ instead of $\Abs{u}_{L^p(\Omega)}$ 
and similarly $\abs{g}_{L^p}$ or $\abs{g}_p$ instead of $\abs{g}_{L^p(\mathit{\Gamma})}$.

We define $L^p$ Sobolev spaces of order $m\in \N$ as 
\[
W^{m,p}(\Omega)=\{ u\in L^p(\Omega) \,|\, \partial_1^j\partial_2^k u \in L^p(\Omega), 0\leq j+k\leq m\} \qquad (\Omega=\cE \mbox{ or }\cI)
\]
endowed with its canonical norm. 
We put $H^m(\Omega)=W^{m,2}(\Omega)$. 
Note that in $U_\mathit{\Gamma}$, Sobolev norms can be defined equivalently in terms of cartesian and normal-tangential derivatives. 
In the proposition below, we write ${\mathfrak d}^\alpha=\dtan^{\alpha_1}\dnor^{\alpha_2}$ for all $\alpha=(\alpha_1,\alpha_2)\in \N^2$.

\begin{proposition}\label{propequivnorm}
For all $m\in \N$, there exists a constant $c_m>0$ depending only on $m$ and $\mathit{\Gamma}$ such that for all regular function $f$ 
supported in $U_\mathit{\Gamma}$, we have 
\[
\frac{1}{c_m} \Abs{f}_{H^m(U_\mathit{\Gamma})} \leq \sum_{\alpha\in \N^2, \abs{\alpha}\leq m} \Abs{{\mathfrak d}^\alpha f}_{L^2(U_\mathit{\Gamma})}
\leq {c_m} \Abs{f}_{H^m(U_\mathit{\Gamma})}. 
\]
\end{proposition}

We can define also fractional Sobolev spaces of order $s\in \R$ on $\mathit{\Gamma}$ by using the fact that the curve $\mathit{\Gamma}$ is 
parameterized by the arc length $s$ as $x=\bm{\gamma}(s)$, which induces a natural isometry between $L^2({\mathbb T}_L)$ and $L^2(\mathit{\Gamma})$. 
More precisely, we define 
\[
H^s(\mathit{\Gamma}) = \{ g\in \mathscr{D}'(\mathit{\Gamma}) \,|\, g\circ\bm{\gamma} \in H^s({\mathbb T}_L)\}
\]
endowed with its canonical norm, where $H^s({\mathbb T}_L)$ is classically defined through Fourier series.

\subsubsection{Weighted norms in time}
We say that $f\in L^p_{\lambda,t}=L^p_{\lambda}(0,t)$ with $1\leq p<\infty$ and $\lambda>0$, if $e^{-\lambda t'}f\in L^p(0,t)$ with canonical norm 
\[
\abs{f}_{L^p_{\lambda,t}}=\left( \int_0^t e^{- p \lambda t'}\abs{f(t')}^p{\rm d}t' \right)^{1/p}. 
\]
We also denote by $I_{\lambda,t}(\cdot)$ the norm of $L^\infty_\lambda(0,t)\cap L^2_\lambda(0,t)$ defined as 
\begin{equation}\label{defIlambda}
I_{\lambda,t}(f) = \sup_{t'\in[0,t]}e^{-\lambda t'}\abs{f(t')} + \sqrt{\lambda}\abs{f}_{L^2_{\lambda,t}},
\end{equation}
which is not the canonical norm of $L^\infty_\lambda(0,t)\cap L^2_\lambda(0,t)$ because of the factor $\sqrt{\lambda}$ in front of the second term. 
The reason why we introduce it is because this quantity arises naturally in hyperbolic energy estimates; see Section \ref{sectapriorilin} below. 
We denote by $S_{\lambda,t}^*(\cdot)$ its dual norm for the $L^2_\lambda(0,t)$ scalar product, that is, 
\begin{equation}\label{defSstar}
S_{\lambda,t}^*(f) = \sup_{\varphi} \biggl\{ \biggl| \int_0^t e^{-2\lambda t'}f(t')\varphi(t'){\rm d}t' \biggr|
 \,;\, I_{\lambda,t}(\varphi) \leq 1 \biggr\}.
\end{equation}
In particular, for all $\varphi\in L^\infty_\lambda(0,t)\cap L^2_\lambda(0,t)$, we have 
\[
\int_0^t e^{-2\lambda t'}f(t')\varphi(t'){\rm d}t'\leq I_{\lambda,t}(\varphi)
S_{\lambda,t}^*(f).
\]
From this definition, we get directly the following upper bounds 
\begin{equation}\label{propSstar}
S^*_{\lambda,t}(f)\leq \abs{f}_{L^1_{\lambda,t}}
\quad\mbox{ and }\quad
S^*_{\lambda,t}(f)\leq \frac{1}{\sqrt{\lambda}}\abs{f}_{L^2_{\lambda,t}}. 
\end{equation}

\subsubsection{Mixed norms}
For all $m\in \N$, we denote by $C^m_{\rm b}(\overline{(0,T)\times\Omega})$ with $\Omega=\cI$ or $\cE$ the space of functions 
that are continuous and bounded in $\overline{(0,T)\times\Omega}$, as well as their space and time derivatives of order up to $m$.

Regularity for solutions to hyperbolic systems on $(0,T)\times \cE$ is typically measured in the spaces $\mathbb{W}^m_T$ with $m\in \N$ defined as 
\[
\mathbb{W}^m_T = \bigcap_{j=0}^mC^j([0,T];H^{m-j}(\cE))
\]
endowed with its canonical norm 
\begin{equation}\label{defopnorm}
\Abs{u}_{\mathbb{W}^m_T}:=\sup_{t\in[0,T]}\opnorm{u(t)}_m \quad\mbox{and}\quad \opnorm{u(t)}_m:=\sum_{j=0}^m \Abs{\dt^j u (t)}_{H^{m-j}(\cE)}. 
\end{equation}
We also introduce similar quantities, but with only non-normal derivatives involved. 
For $\alpha=(\alpha_0,\alpha_1)\in \N^2$, we write 
\[
\dpar^\alpha = \dt^{\alpha_0}\dtan^{\alpha_1}
\]
and define, for functions $u$ supported in the tubular neighborhood $U_\mathit{\Gamma}$ where the normal-tangential coordinates are well defined, 
\begin{equation}\label{defopnormpar}
\opnorm{u(t)}_{m,\parallel}:=\sum_{\alpha\in \N^2, \abs{\alpha}\leq m} \Abs{\dpar^\alpha u(t)}_{L^2(\cE)}. 
\end{equation}
For a technical reason we use an $L^p$ version of the space $\mathbb{W}^m_T$ defined as 
\[
\mathbb{W}^{m,p}_T = \bigcap_{j=0}^mC^j([0,T];W^{m-j,p}(\cE))
\]
endowed with its canonical norm $\Abs{\cdot}_{\mathbb{W}^{m,p}_T}$ together with $\opnorm{\cdot}_{m,p}$.

For functions defined on $\mathit{\Gamma}$ and depending also on time, 
we introduce the space $L^p_{\lambda,t}H^{s}(\mathit{\Gamma})$ defined for $1\leq p\leq \infty$ and $s,\lambda\in \R$ as 
\begin{equation}\label{defHm0}
L^p_{\lambda,t}H^s(\mathit{\Gamma})=L^p_\lambda(0,t;H^s(\mathit{\Gamma}));
\end{equation}
as well $L^p_{\lambda,t}H^s_{(m)}(\mathit{\Gamma})$ with $m\in \N$ defined as 
\begin{equation}\label{defHm}
L^p_{\lambda,t}H^{s}_{(m)}(\mathit{\Gamma})
 = \{ g\in L^p(0,t;H^{s}(\mathit{\Gamma})) \,|\, \abs{g(\cdot)}_{H^{s}_{(m)}(\mathit{\Gamma})}\in L^p_\lambda(0,t)\}, 
\end{equation}
where
\begin{equation}\label{defHm2}
 \abs{g(t)}_{H^{s}_{(m)}(\mathit{\Gamma})} = \sum_{\alpha\in \N^2, \abs{\alpha}\leq m} \abs{\dpar^\alpha g(t)}_{H^{s}(\mathit{\Gamma})},
\end{equation}
and endowed with their canonical norms. 
We introduce also the space $\mathbb{W}_{{\rm b},T}^{m+1/2}$, to which the solution $\psi_{\rm i}$ to the nonlinear wave-structure interaction model 
\eqref{PB1}--\eqref{PB3} belongs, as 
\[
\mathbb{W}^{m+1/2}_{{\rm b},T} = \bigcap_{j=0}^mC^j([0,T];H^{m-j+1/2}(\Gamma)).
\]

\subsection{Extension operator and Dirichlet-to-Neumann map}\label{sectDN}
Functions defined on the curve $\mathit{\Gamma}$ can be extended as functions defined in the exterior domain $\cE$ or in the interior domain $\cI$. 
For the extension to the exterior domain, we will repeatedly use the fact that there is a regularizing extension.

\begin{proposition}\label{propextext}
Let $\theta$ be the diffeomorphism constructed in \eqref{diffeo-theta}. 
For all $m\in \N$ and $\psi\in H^{m-1/2}(\mathit{\Gamma})$, there exists an extension denoted $\psi^{\rm ext}$ of $\psi$, 
compactly supported in $\overline{\cE}\cap U_\mathit{\Gamma}$ such that $\psi^{\rm ext}\circ \theta \in C([0,r_0);H^{m-1/2}({\mathbb T}_L))$ 
and $\psi^{\rm ext} \in H^{m}(\cE)$, and that  satisfies 
\[
{\psi^{\rm ext}}_{\vert_\mathit{\Gamma}}=\psi \quad \mbox{ and }\quad \Abs{\psi^{\rm ext}}_{H^{m}(\cE)}\leq C \abs{\psi}_{H^{m-1/2}(\mathit{\Gamma})}
\]
for some constant $C$ that depends only on $\mathit{\Gamma}$ and $m$. 
\end{proposition}

\begin{proof}
Denoting $\langle{D}\rangle=(1-\partial_s^2)^{1/2}$, such an extension satisfying the properties of the proposition is given by,
\[
\psi^{\rm ext}\circ\theta(r,\cdot)=\chi(r)\chi(r\langle D\rangle)\big( \psi\circ \theta(0,\cdot )\big),
\]
where $r\in [0,r_0]$ and $\chi$ denote a smooth function supported in $[0,r_0)$ satisfying $\chi(0)=1$.
\end{proof}

We can also extend functions defined on $\mathit{\Gamma}$ to functions defined on the closure $\overline{\cI}$ of the interior domain. 
A natural extension of a function $\psi$ defined on $\mathit{\Gamma}$ is provided by a {\em harmonic} extension $\psi^{\mathfrak h}$. 
More precisely, for any $\psi\in H^{1/2}(\mathit{\Gamma})$, $\psi^\mathfrak{h}$ is defined as a unique solution in $H^1(\cI)$ 
to the elliptic boundary value problem 
\begin{equation}\label{BVPinI}
\begin{cases}
\nabla\cdot(h_{\rm i}\nabla \psi^{\mathfrak h})=0 &\mbox{in}\quad \cI, \\
{\psi^{\mathfrak h}}=\psi&\mbox{on}\quad \mathit{\Gamma}.
\end{cases}
\end{equation}
This is a natural extension in our context, because if $\phi_{\rm i}$ denotes the velocity potential associated with the average velocity 
under the object of a solution to the nonlinear wave-structure problem, 
then \eqref{EEinI1} shows that $\phi_{\rm i}$ solves the above elliptic equation in the interior domain with the boundary data $\psi=\psi_{\rm i}$. 
Note that existence and uniqueness of the solution $\psi^{\mathfrak h}$ follow from standard variational arguments under the assumption that 
$h_{\rm i}$ does not vanish. 
We therefore make the following assumption throughout this article.

\begin{assumption}\label{ass:hi}
There exists a positive constant $c_0$ such that the followings hold. 
\begin{enumerate}
\item[{\rm (i)}]
$h_{\rm i} \in C^\infty(\overline{\cI})$. 
\item[{\rm (ii)}]
$h_{\rm i}(x) \geq c_0$ for $x\in\cI$. 
\end{enumerate}
\end{assumption}

Under this assumption, for each $\psi\in H^{1/2}(\mathit{\Gamma})$ the boundary value problem \eqref{BVPinI} has a unique solution 
$\psi^\mathfrak{h} \in H^1(\cI)$. 
Moreover, $( N\cdot(h_{\rm i}\nabla \psi^{\mathfrak h}) )_{\vert_\mathit{\Gamma}}$ can be defined as a function in $H^{-1/2}(\mathit{\Gamma})$. 
Therefore, the following definition makes sense.

\begin{definition}\label{defDN}
Under Assumption \ref{ass:hi}, we define the \emph{Dirichlet-to-Neumann} map $\Lambda$ as 
\[
\Lambda: \begin{array}{lcl}
H^{1/2}(\mathit{\Gamma}) &\to& H^{-1/2}(\mathit{\Gamma}) \\
\psi & \mapsto & ( N\cdot(h_{\rm i}\nabla \psi^{\mathfrak h}) )_{\vert_\mathit{\Gamma}},
\end{array}
\]
where $\psi^\mathfrak{h}$ is a unique solution to the boundary value problem \eqref{BVPinI}. 
\end{definition}

The following proposition gathers some properties of the Dirichlet-to-Neumann map; 
we use the notation $\langle\cdot,\cdot\rangle_{H^{-1/2}\times H^{1/2}}$ for the standard 
$H^{-1/2}(\mathit{\Gamma})$--$H^{1/2}(\mathit{\Gamma})$ duality bracket.

\begin{proposition}\label{propDN}
Under Assumption \ref{ass:hi}, the following properties hold. 
\begin{enumerate}
\item[{\rm (i)}]
There exist positive constants $\underline{\mathfrak c}$ and $\underline{\mathfrak C}$ that depend only on $\mathit{\Gamma}$, $c_0$, and 
$\Abs{ h_{\rm i} }_{L^\infty(\cI)}$ such that for any $\psi\in H^{1/2}(\mathit{\Gamma})$, we have 
\[
\underline{\mathfrak c} \abs{\dtan\psi}_{H^{-1/2}}^{2}
 \leq \langle \Lambda\psi,\psi \rangle_{H^{-1/2}\times H^{1/2}}
 \leq \underline{\mathfrak C}\abs{\dtan\psi}_{H^{-1/2}}^{2}.
\]
\item[{\rm (ii)}]
For any $\psi_1,\psi_2 \in H^{1/2}(\mathit{\Gamma})$, we have 
\[
\langle \Lambda\psi_1,\psi_2 \rangle_{H^{-1/2}\times H^{1/2}} = \langle \Lambda\psi_2,\psi_1 \rangle_{H^{-1/2}\times H^{1/2}}. 
\]
\item[{\rm (iii)}]
For any $k\in \N$, $\Lambda$ maps $H^{k+1/2}(\mathit{\Gamma})$ to $H^{k-1/2}(\mathit{\Gamma})$ and we have 
\[
\abs{\Lambda \psi}_{H^{k-1/2}}\leq C \abs{\dtan \psi}_{H^{k-1/2}},
\]
where the constant $C$ depends only on $k$, $\mathit{\Gamma}$, $c_0^{-1}$, and $\Abs{ h_{\rm i} }_{W^{k,\infty}(\cI)}$. 
\item[{\rm (iv)}]
For any positive integer $k$, we have 
\[
\begin{cases}
 \abs{[\dtan^k,\Lambda]\psi}_{H^{-1/2}}\leq C_1 \abs{\dtan\psi}_{H^{k-3/2}}, \\
 \abs{[\dtan^k,\Lambda]\psi}_{H^{1/2}}\leq C_2 \abs{\dtan\psi}_{H^{k-1/2}},
\end{cases}
\]
where the constant $C_1$ depends only on $k$, $\mathit{\Gamma}$, $c_0^{-1}$, and $\Abs{ h_{\rm i} }_{W^{k,\infty}(\cI)}$, 
and the constant $C_2$ depends also on $\Abs{ h_{\rm i} }_{W^{k+1,\infty}(\cI)}$. 
\end{enumerate}
\end{proposition}

\begin{proof}
The first two points are standard and follow easily from the relation 
\[
\langle \Lambda\psi_1,\psi_2 \rangle_{H^{-1/2}\times H^{1/2}}=\int_{\cI} h_{\rm i}\nabla \psi_1^{\mathfrak h}\cdot \nabla\psi_2^{\mathfrak h}. 
\]
The third point follows from the classical elliptic estimate
\begin{equation}\label{DNp0}
\Abs{ \nabla\psi^\mathfrak{h} }_{H^k(\cI)} \lesssim \abs{ \dtan\psi }_{H^{k-1/2}};
\end{equation}
for more details on this estimate, see for instance \cite{McLean}. 
For the fourth point, we note first that 
$\Lambda\psi = (N\cdot(h_{\rm i}\nabla\psi^\mathfrak{h}))_{\vert_\mathit{\Gamma}} = (h_{\rm i}\dnor\psi^\mathfrak{h})_{\vert_\mathit{\Gamma}}$, so that 
\begin{equation}\label{DNp1}
[\dtan^k,\Lambda]\psi
= \bigl( N\cdot(h_{\rm i}\nabla(\chi_{\rm b}\dtan^k\psi^\mathfrak{h} - (\dtan^k\psi)^\mathfrak{h})) \bigr)_{\vert_\mathit{\Gamma}}
 + ([\dtan^k,h_{\rm i}]\dnor\psi^\mathfrak{h})_{\vert_\mathit{\Gamma}}, 
\end{equation}
where $\chi_{\rm b}$ is the smooth cutoff function with support in $U_\mathit{\Gamma}$ defined in \eqref{defchib} 
and we have used the fact that $\dtan$ commutes with $\dnor = N\cdot\nabla$. 
The quantity $\chi_{\rm b} \dtan^k \psi^{\mathfrak h}$ is defined and compactly supported in $\overline{\cI}\cap U_\mathit{\Gamma}$; 
we extend it by zero and consider it as a function defined in the whole interior region $\cI$. 
Then, we note also that in view of \eqref{diff exp1} the elliptic equation in \eqref{BVPinI} can be written as 
\[
\dtan(J^{-1}h_{\rm i}\dtan\psi^\mathfrak{h}) + \dnor(Jh_{\rm i}\dnor\psi^\mathfrak{h}) = 0 
 \quad\mbox{in}\quad \cI\cap U_\mathit{\Gamma}, 
\]
so that we have 
\begin{align*}
& \dtan(J^{-1}h_{\rm i}\dtan(\dtan^k\psi^\mathfrak{h})) + \dnor(Jh_{\rm i}\dnor(\dtan^k\psi^\mathfrak{h})) \\
&= - \dtan([\dtan^k,J^{-1}h_{\rm i}]\dtan\psi^\mathfrak{h}) - \dnor([\dtan^k,Jh_{\rm i}]\dnor\psi^\mathfrak{h})
 \quad\mbox{in}\quad \cI\cap U_\mathit{\Gamma}. 
\end{align*}
In view of \eqref{diff exp2} we can rewrite this equation as 
\[
\nabla\cdot(h_{\rm i}\nabla(\dtan^k\psi^\mathfrak{h})) = \nabla\cdot F_0
 \quad\mbox{in}\quad \cI\cap U_\mathit{\Gamma}, 
\]
where 
\[
F_0 = -J^{-1}( T[\dtan^k,J^{-1}h_{\rm i}]\dtan\psi^\mathfrak{h} + N[\dtan^k,Jh_{\rm i}]\dnor\psi^\mathfrak{h}).
\]
Now, we define $\psi^{(k)}$ as 
\[
\psi^{(k)} := \chi_{\rm b}\dtan^k\psi^\mathfrak{h} - (\dtan^k\psi)^\mathfrak{h}. 
\]
Then, we see that $\psi^{(k)}$ solves the boundary value problem 
\begin{equation}\label{DNp2}
\begin{cases}
 \nabla\cdot(h_{\rm i}\nabla\psi^{(k)}) = \nabla\cdot F+f &\mbox{in}\quad \cI, \\
 \psi^{(k)} = 0 &\mbox{on}\quad \mathit{\Gamma},
\end{cases}
\end{equation}
where 
\begin{align*}
 F &= \chi_{\rm b}F_0 + 2h_{\rm i}(\dtan^k\psi^\mathfrak{h})\nabla\chi_{\rm b}, \\
 f &= -\nabla\chi_{\rm b}\cdot F_0 - (\nabla\cdot(h_{\rm i}\nabla\chi_{\rm b}))\dtan^k\psi^\mathfrak{h}.
\end{align*}
It is easy to see that $(N\cdot F)_{\vert_\mathit{\Gamma}} = -([\dtan^k,h_{\rm i}]\dnor\psi^\mathfrak{h})_{\vert_\mathit{\Gamma}}$, 
which together with \eqref{DNp1} implies 
\begin{equation}\label{DNp3}
[\dtan^k,\Lambda]\psi = (N\cdot(h_{\rm i}\nabla\psi^{(k)}-F))_{\vert_\mathit{\Gamma}}.
\end{equation}
Therefore, for any $\tilde{\psi}\in H^{1/2}(\mathit{\Gamma})$ we see that 
\begin{align*}
\langle [\dtan^k,\Lambda]\psi, \tilde{\psi} \rangle_{H^{-1/2}\times H^{1/2}}
&= \int_{\cI}\nabla\cdot(\tilde{\psi}^\mathfrak{h}(h_{\rm i}\nabla\psi^{(k)}-F)) \\
&= \int_{\cI}\{ \tilde{\psi}^\mathfrak{h}\nabla\cdot(h_{\rm i}\nabla\psi^{(k)}-F) 
 + \nabla\tilde{\psi}^\mathfrak{h}\cdot(h_{\rm i}\nabla\psi^{(k)}-F)\} \\
&= \int_{\cI}( \tilde{\psi}^\mathfrak{h}f - \nabla\tilde{\psi}^\mathfrak{h}\cdot F),
\end{align*}
where we have used \eqref{DNp2} and the fact that $\tilde{\psi}^\mathfrak{h}$ solves the boundary value problem \eqref{BVPinI} with $\psi=\tilde{\psi}$. 
This identity yields 
\begin{align*}
\abs{ \langle [\dtan^k,\Lambda]\psi, \tilde{\psi} \rangle_{H^{-1/2}\times H^{1/2}} }
&\leq \Abs{ (F,f) }_{L^2(\cI)} \Abs{ \tilde{\psi}^\mathfrak{h} }_{H^1(\cI)} \\
&\lesssim \Abs{ \nabla\psi^\mathfrak{h} }_{H^{k-1}(\cI)} \abs{ \psi }_{H^{1/2}},
\end{align*}
which together with \eqref{DNp0} with $k$ replaced by $k-1$ gives the first estimate. 
As for the second one, it is sufficient to evaluate \eqref{DNp3} directly by the trace theorem. 
In fact, by a classical theory on elliptic boundary value problem we see that 
\begin{align*}
\abs{ [\dtan^k,\Lambda]\psi }_{H^{1/2}}
&\lesssim \Abs{ \nabla\psi^{(k)} }_{H^1(\cI)} + \Abs{F}_{H^1(\cI)} \\
&\lesssim \Abs{F}_{H^1(\cI)} + \Abs{f}_{L^2(\cI)} \\
&\lesssim \Abs{ \nabla\psi^\mathfrak{h} }_{H^k(\cI)},
\end{align*}
which together with \eqref{DNp0} gives the second estimate. 
\end{proof}

\section{A priori estimates for linear hyperbolic systems with weakly dissipative boundary conditions in an exterior domain}\label{sectapriorilin}
We consider in this section the following system of $n$ equations cast in the exterior domain $\cE$ 
\begin{equation}\label{eqlinfixed}
\dt u +A_j(t,x)\partial_j u + B(t,x)u = f(t,x) \quad\mbox{in}\quad (0,T)\times\cE,
\end{equation}
where $f$ is a $\R^n$-valued function defined in $(0,T)\times \cE$ for some $T>0$, 
while $A_1$, $A_2$, and $B$ take their values in the space of $n\times n$ real-valued matrices. 
We recall that throughout this article, we use Einstein's notation on repeated indices so that, for instance, in \eqref{eqlinfixed} above, 
$A_j(t,x)\partial_j u$ stands for $A_1(t,x)\partial_1 u +A_2(t,x)\partial_2 u$.

Our goal in this section is to derive a priori estimates for regular solutions to \eqref{eqlinfixed} subject to some boundary conditions. 
Since the boundary conditions we have to deal with in this article do not fit into any known class of boundary conditions, such as 
maximal dissipative or strictly dissipative conditions, for which such a priori estimates are known, 
we introduce a new and more general notion of weak dissipativity that allows the derivation of such estimates.

\medbreak
Throughout this section, we assume that there exists a Friedrichs symmetrizer $S(t,x)$ in the following sense.

\begin{assumption}\label{assFriedrichs}
There exists a $n\times n$ real valued symmetric matrix $S(t,x)$ defined in $(0,T)\times\cE$ such that for any $(t,x)\in(0,T)\times\cE$, 
matrices $S(t,x)A_j(t,x)$ $(j=1,2)$ are symmetric and the following conditions hold. 
\begin{enumerate}
\item[{\rm (i)}]
There exist constants $\alpha_0,\beta_0>0$ such that for any $(v,t,x)\in\R^n\times(0,T)\times\cE$ we have 
\[
\alpha_0\abs{v}^2 \leq v \cdot S(t,x)v \leq \beta_0 \abs{v}^2.
\]
We also denote by $\beta_0^{\rm in}$ the best constant such that the second inequality holds at $t=0$ for all $(v,x)\in \R^n\times\cE$. 
\item[{\rm (ii)}]
There exists a constant $\beta_1$ such that for any $(v,t,x)\in\R^n\times(0,T)\times\cE$ we have 
\[
v\cdot( \dt S(t,x) + \partial_j(S(t,x)A_j(t,x)) - 2S(t,x)B(t,x) )v \leq \beta_1\abs{v}^2.
\]
\end{enumerate}
\end{assumption}

The organization of this section is as follows. 
In Section \ref{sectAEF}, we derive a priori estimates for solutions to \eqref{eqlinfixed} that vanish in a neighborhood of the boundary. 
On the contrary, in Section \ref{sectAENB} we deal with solutions which are supported near the boundary. 
In this case, we have to control the contribution of the boundary term in the energy estimates. 
Various notions of dissipativity can be found in the literature to deal with this boundary contribution, 
but none is adapted to deal with the initial boundary value problem associated with \eqref{PB1}--\eqref{PB3}. 
We therefore introduce, for general hyperbolic systems of the form \eqref{eqlinfixed}, a notion of weak dissipativity 
and show how it can be used to derive a priori estimates. 
Finally, we extend in Section \ref{sectAEGC} this result to the general case where no assumption is made on the support of the solution.

\subsection{A priori energy estimates away from the boundary}\label{sectAEF}
We first provide energy estimates for solutions that vanish near the boundary $\mathit{\Gamma}$. 
This is the simplest configuration since the boundary contribution in the energy estimates then vanishes. 
As in \cite{IguchiLannes} for the one-dimensional case, we provide a sharp dependence on the source term which improves the classical estimates; 
see for instance \cite{BenzoniSerre}. 
In order to do so we use the quantities $I_{\lambda,t}(\cdot)$ and $S_{\lambda,t}^*(\cdot)$ defined in \eqref{defIlambda} and \eqref{defSstar}, 
respectively.

\begin{proposition}\label{propNRJL2far}
There exists an absolute constant $C$ such that under Assumption \ref{assFriedrichs}, 
any regular solution $u$ to \eqref{eqlinfixed} that vanishes in a neighborhood of $\mathit{\Gamma}$ satisfies 
\[
I_{\lambda,t}( \Abs{u(\cdot)}_{L^2(\cE)} )^2
\leq C\tfrac{\beta_0^{\rm in}}{\alpha_0} \Abs{u(0,\cdot)}^2_{L^2(\cE)} + C\tfrac{\beta_0}{\alpha_0} S_{\lambda,t}^*( \Abs{f(\cdot)}_{L^2(\cE)} )^2
\]
for any $\lambda\geq\frac{\beta_1}{\alpha_0}$ and $t\in[0,T]$, where the constants $\alpha_0, \beta_0, \beta_0^{\rm in}$, and $\beta_1$ are 
those in Assumption \ref{assFriedrichs} and $I_{\lambda,t}(\cdot)$ and $S_{\lambda,t}^*(\cdot)$ are defined in \eqref{defIlambda} and \eqref{defSstar}. 
\end{proposition}

\begin{proof}
For the sake of clarity we simply denote $S$, $A_j$, and $B$ instead of $S(t,x)$, $A_j(t,x)$, and $B(t,x)$, 
while $(\cdot,\cdot)_{L^2}$ stands for the standard scalar product in $L^2(\cE)$. 
Since $SA_j$ symmetric, we have 
\begin{equation}\label{IPPterm}
2SA_j \partial_j u\cdot u = \partial_j (SA_j u\cdot u) - (\partial_j(SA_j)) u\cdot u,
\end{equation}
and since $u$ vanishes at the boundary we easily get that for all positive number $\lambda$, 
\begin{align*}
&\frac{\rm d}{{\rm d}t} \bigl\{ e^{-2\lambda t }(Su,u)_{L^2} \bigr\}
 + 2\lambda e^{-2\lambda t}(Su,u)_{L^2} \\
&= 2e^{-2\lambda t}(S f,u)_{L^2} + e^{-2\lambda t} ( (\dt S +\partial_j ({SA_j})-2SB) u,u)_{L^2}.
\end{align*}
Owing to Assumption \ref{assFriedrichs} the last term can be absorbed into the left-hand side if we take $\lambda$ so large that 
$\lambda\geq \frac{\beta_1}{\alpha_0}$. 
Then, integrating it in time yields therefore 
\[
\frac{1}{2}\widetilde{I}_{\lambda,t}( (Su,u)_{L^2}^{1/2} )^2
\leq (Su(0),u(0)\big)_{L^2} + 2\int_0^t e^{-2\lambda t'}(Sf(t'),u(t'))_{L^2}{\rm d}t',
\]
where $\widetilde{I}_{\lambda,t}(f)$ is defined as 
\begin{equation}\label{defItilde}
\widetilde{I}_{\lambda,t}(f) = e^{-\lambda t} \abs{f(t)}+\sqrt{\lambda}\abs{f}_{L^2_{\lambda,t}};
\end{equation}
it differs from the quantity ${I}_{\lambda,t}(f)$ defined in \eqref{defIlambda} by the fact that the first term in the right-hand side is 
$\sup_{t'\in[0,t]}e^{-\lambda t'} \abs{f(t')}$ in \eqref{defIlambda}. 
By definition of $S_{\lambda,t}^*(\cdot)$, we have 
\begin{align*}
2\int_0^t e^{-2\lambda t'}(Sf(t'),u(t'))_{L^2}{\rm d}t'
&\leq 2 S_{\lambda,t}^*((Sf,f)_{L^2}^{1/2}) I_{\lambda,t}( (Su,u)_{L^2}^{1/2} ) \\
&\leq \frac{1}{4}I_{\lambda,t}( (Su,u)_{L^2}^{1/2} )^2 + 4 S_{\lambda,t}^*( (Sf,f)_{L^2}^{1/2} )^2. 
\end{align*}
We have therefore 
\[
\frac{1}{2}\widetilde{I}_{\lambda,t} ( (Su,u)_{L^2}^{1/2} )^2
\leq (Su(0),u(0))_{L^2} + \frac{1}{4}I_{\lambda,t}\big( (Su,u)_{L^2}^{1/2} )^2 + 4 S_{\lambda,t}^*( (Sf,f)_{L^2}^{1/2} )^2. 
\]
Since the right-hand side is an increasing function of time, this inequality is still true if we replace the left-hand side with 
$\frac{1}{2}\widetilde{I}_{\lambda,t'}( (Su,u)_{L^2}^{1/2} )^2 $ for any $t'\in [0,t]$. 
It follows easily that the inequality is also true with $\frac12 I_{\lambda,t}( (Su,u)_{L^2}^{1/2} )^2$ in the left-hand side. 
The proposition follows easily using the properties of $S$ listed in Assumption \ref{assFriedrichs}. 
\end{proof}

We can also derive higher order energy estimates using the spaces $\mathbb{W}^m_T$ introduced in \eqref{defopnorm}; we recall that 
$\Abs{u}_{\mathbb{W}^m_T}=\sup_{0\leq t\leq T}\opnorm{u(t)}_m$ and $\Abs{u}_{\mathbb{W}^{m,p}_T}=\sup_{0\leq t\leq T}\opnorm{u(t)}_{m,p}$, where 
\[
\opnorm{u(t)}_m =\sum_{j=0}^m \Abs{\dt^j u (t)}_{H^{m-j}(\cE)} \quad\mbox{and}\quad
\opnorm{u(t)}_{m,p} =\sum_{j=0}^m \Abs{\dt^j u (t)}_{W^{m-j,p}(\cE)}.
\]
In the statement below, we use the notation $\boldsymbol{\partial} u = (\dt u, \partial_1u,\partial_2u)$. 
Concerning regularities on the coefficient matrices $A_1$, $A_2$, and $B$, we impose the following assumption.

\begin{assumption}\label{ass:regAB}
Let $m$ be a non-negative integer. 
There exist an index $p\in(2,\infty)$ and two constants $0<K_0\leq K$ such that the following conditions hold. 
\begin{enumerate}
\item[{\rm (i)}]
$\|(A_1,A_2)\|_{L^\infty((0,T)\times\cE)} \leq K_0$. 
\item[{\rm (ii)}]
$\Abs{ \boldsymbol{\partial}(A_1,A_2) }_{\mathbb{W}_T^{m-1} \cap \mathbb{W}_T^{1,p}}$, 
$\Abs{B}_{L^\infty((0,T)\times\cE)}$, $\Abs{ \boldsymbol{\partial} B}_{\mathbb{W}_T^{m-1} \cap \mathbb{W}_T^1} \leq K$.
\end{enumerate}
\end{assumption}

\begin{remark}\label{rem:space}
In place of the condition in (i), we may assume $\Abs{ \boldsymbol{\partial}(A_1,A_2) }_{\mathbb{W}_T^{m-1} \cap \mathbb{W}_T^{2}} \leq K$, 
which is slightly restrictive than the above condition due to the Sobolev embedding $H^1(\cE) \hookrightarrow L^p(\cE)$. 
We note also that the requirement $u\in\mathbb{W}_T^3$, so that $\boldsymbol{\partial} u \in \mathbb{W}_T^2$, corresponds to the quasilinear regularity 
in the sense that it is the minimal integer regularity index $m_0$ that ensures the embedding 
$u \in \mathbb{W}^{m_0}_T \hookrightarrow W^{1,\infty}((0,T)\times \cE)$. 
However, in applications, especially in the analysis of the regularity of solutions to nonlinear problems, 
the condition $\boldsymbol{\partial}(A_1,A_2) \in \mathbb{W}_T^{2}$ is rather strong and causes a difficulty in the critical case $m=3$; 
see Section \ref{sect:exist2}. 
\end{remark}

\begin{proposition}\label{propNRHhighfar}
Let $m$ be a non-negative integer. 
There exists a constant $C$ depending only on $m$ such that under Assumptions \ref{assFriedrichs} and \ref{ass:regAB}, 
any regular solution $u$ to \eqref{eqlinfixed} that vanishes in a neighborhood of $\mathit{\Gamma}$ satisfies 
\[
I_{\lambda,t}( \opnorm{ u(\cdot) }_m )^2
\leq C\tfrac{\beta_0^{\rm in}}{\alpha_0} \opnorm{ u(0) }_m^2 + C\tfrac{\beta_0}{\alpha_0} S_{\lambda,t}^*( \opnorm{ f(\cdot) }_m )^2
\]
for any $\lambda\geq\lambda_0$ and $t\in[0,T]$, where $\lambda_0$ depends only on $m$, $p$, $\frac{\beta_1}{\alpha_0}$, and $K$. 
\end{proposition}

\begin{proof}
For a multi-index $\beta=(\beta_0,\beta_1,\beta_2)\in \N^3$, we apply $\partial^\beta=\dt^{\beta_0}\partial_1^{\beta_1}\partial_2^{\beta_2}$ 
to \eqref{eqlinfixed} to obtain 
\begin{equation}\label{systalphafixed}
\dt \partial^\beta u+A_j\partial_j \partial^\beta u + B\partial^\beta u = f_\beta,
\end{equation}
where 
\[
f_\beta = \partial^\beta f-[\partial^\beta, A_j] \partial_j u- [\partial^\beta,B]u.
\]
The commutator terms in $f_\beta$ can be controlled by using the following classical commutator estimates, 
which can be easily obtained by using Sobolev embeddings $W^{1,p}(\cE) \hookrightarrow L^\infty(\cE)$ and $H^1(\cE) \hookrightarrow L^q(\cE)$ 
for any $q\in[2,\infty)$.

\begin{lemma}\label{lemmcommut}
Let $m$ be a non-negative integer, $\beta\in\N^3$ a multi-index satisfying $1\leq \abs{\beta} \leq m$, and $2< p<\infty$. 
Then, for any smooth functions $f$ and $g$ we have 
\[
\Abs{[\partial^\beta, f] g(t)}_{L^2(\cE)}
\lesssim
\begin{cases}
 \opnorm{ \partial f(t) }_{1,p} \opnorm{ g(t) }_{m-1} &\mbox{for}\quad m=1,2, \\
 \opnorm{ \partial f(t) }_{m-1} \opnorm{ g(t) }_{m-1} &\mbox{for}\quad m\geq 3, \\
 \opnorm{ \partial f(t) }_{\max\{m-1,1\}} \opnorm{ g(t) }_m &\mbox{for}\quad m\geq1.
\end{cases}
\]
\end{lemma}

Using this lemma, we obtain the following upper bound for the source term $f_\beta$ 
\[
\sum_{\abs{\beta}\leq m}\Abs{f_\beta(t)}_{L^2} \lesssim \opnorm{f(t)}_{m} + K\opnorm{ u(t) }_{m}. 
\]
Applying Proposition \ref{propNRJL2far} to \eqref{systalphafixed} and summing over all $\beta$ such that $\abs{\beta}\leq m$, we get therefore 
\[
I_{\lambda,t}(\opnorm{u(\cdot)}_m)^2
\leq C\tfrac{\beta_0^{\rm in}}{\alpha_0} \opnorm{u(0)}_{m}^2 + C\tfrac{\beta_0}{\alpha_0} S_{\lambda,t}^*(\opnorm{ f(\cdot) }_{m})^2 
 + CK\lambda^{-2}I_{\lambda,t}(\opnorm{u(\cdot)}_m)^2,
\]
where we used the second inequality in \eqref{propSstar} to derive the last term. 
By taking $\lambda$ large enough, this last term can be absorbed into the left-hand side and the result follows. 
\end{proof}

\subsection{A priori energy estimates near the boundary}\label{sectAENB}
Energy estimates for solutions to \eqref{eqlinfixed} supported away from the boundary $\mathit{\Gamma}$ have been derived in the previous section. 
For solutions that do not vanish near $\mathit{\Gamma}$, an additional boundary term prevents us from getting directly the energy estimate in 
Proposition \ref{propNRJL2far}. 
This boundary term comes from the integration of \eqref{IPPterm} over $\cE$, that is, 
\[
2\int_{\cE} SA_j\partial_j u\cdot u = -\int_\mathit{\Gamma}SA_{\rm nor}u\cdot u - ( (\partial_j(SA_j))u,u )_{L^2(\cE)},
\]
where $A_{\rm nor}$ is the boundary matrix, also called the normal matrix, defined by 
\[
A_{\rm nor} = N_jA_j.
\]
Compared to the energy estimate of Proposition \ref{propNRJL2far}, 
there is therefore an additional boundary term if the solution does not vanish in the neighborhood of the boundary, that is, 
\begin{align}\label{defsupstar}
I_{\lambda,t}( \Abs{u(\cdot)}_{L^2(\cE)} )^2 \leq 
& C\tfrac{\beta_0^{\rm in}}{\alpha_0} \Abs{u(0,\cdot)}^2_{L^2(\cE)} + C\tfrac{\beta_0}{\alpha_0} S_{\lambda,t}^*(\Abs{f(\cdot)}_{L^2(\cE)})^2 \\
& + \frac{4}{\alpha_0} \int_0^t e^{-2\lambda t'} \biggl( \int_\mathit{\Gamma} {\mathfrak B}[u(t')] \biggr) {\rm d}t', \nonumber
\end{align}
where the boundary quadratic form ${\mathfrak B}[u]$ is defined as 
\begin{equation}\label{defB}
{\mathfrak B}[u]:= SA_{\rm nor}u\cdot u.
\end{equation}
This term cannot be controlled in terms of $\Abs{u(t,\cdot)}_{L^2(\cE)}$ by Sobolev embeddings and additional information 
is therefore needed on the solution. 
This additional information comes from the fact that the equations \eqref{eqlinfixed} should be complemented by boundary conditions on $u$.

\begin{example}\label{exBC1}
A typical example consists in complementing \eqref{eqlinfixed} with a set of linear boundary conditions of the form 
\begin{equation}\label{BClin}
M(t,x)u = g(t,x) \quad\mbox{on}\quad (0,T)\times\mathit{\Gamma},
\end{equation}
where $g$ is an $\R^p$-valued function defined on $(0,T)\times\mathit{\Gamma}$, 
while $M$ takes its values in the space of $p\times n$ real-valued matrices, so that $p$ is the number of scalar boundary conditions. 
\end{example}

Finding a good set of boundary conditions and a symmetrizer that provides a control on the boundary contribution 
to the energy estimate is in general a difficult task that we shall not address at this point of the discussion. 
For the moment we just assume that the boundary term ${\mathfrak B}[u]$ is \emph{weakly dissipative} in the following sense; 
see Examples \ref{BC2} and \ref{BC3} below for classical configurations leading to a weakly dissipative boundary term.

\begin{definition} \label{propweakdissip}
We say that the boundary term ${\mathfrak B}[u]$ associated with a regular solution to \eqref{eqlinfixed} is \emph{weakly dissipative} 
if there exist a positive constant $\lambda_0$, a non-negative and non-decreasing function of time $S_{\rm data}[u](\cdot)$ 
depending only on the initial and boundary data possibly imposed on $u$, and 
a non-negative \emph{boundary energy} function $E_{\rm bdry}(\cdot)$ such that for any $\lambda\geq\lambda_0$ and any $t\in[0,T]$ we have
\begin{equation}\label{weakdissip}
\int_0^t e^{-2\lambda t'} \biggl( \int_\mathit{\Gamma} {\mathfrak B}[u(t')] \biggr){\rm d}t'
\leq -E_{\rm bdry}(t)+S_{\rm data}(t) + \frac{\alpha_0}{8} I_{\lambda,t}(\Abs{u(\cdot)}_{L^2(\cE)})^2,
\end{equation}
where $\alpha_0$ is the coercivity constant of the symmetrizer as defined in Assumption \ref{assFriedrichs}. 
\end{definition}

\begin{remark}\label{remalpha0}
The coefficient $\frac{\alpha_0}{8}$ in front of the term $I_{\lambda,t}(\Abs{u(\cdot)}_{L^2(\cE)})^2$ 
in the right-hand side allows one to absorb this term by the left-hand side in the energy estimate. 
\end{remark}

When the boundary term is weakly dissipative, it is possible to state a generalization of Proposition \ref{propNRJL2far} 
for solutions that do not vanish in the neighborhood of $\mathit{\Gamma}$. 
We omit the proof, which is straightforward; we just need to use \eqref{weakdissip} in \eqref{defsupstar}.

\begin{proposition}\label{propNRJnear}
There exists an absolute constant $C$ such that under Assumption \ref{assFriedrichs}, 
if $u$ is a regular solution to \eqref{eqlinfixed} and if the boundary term is weakly dissipative in the sense of Definition \ref{propweakdissip}, 
then for any $\lambda\geq\max\{\frac{\beta_1}{\alpha_0},\lambda_0\}$ and $t\in[0,T]$ we have 
\[
I_{\lambda,t}( \Abs{u(\cdot)}_{L^2(\cE)} )^2 + \tfrac{8}{\alpha_0}E_{\rm bdry}(t)
\leq C\tfrac{\beta_0^{\rm in}}{\alpha_0} \Abs{u(0,\cdot)}^2_{L^2(\cE)} + C\tfrac{\beta_0}{\alpha_0} S_{\lambda,t}^*( \Abs{f(\cdot)}_{L^2(\cE)} )^2
 + \tfrac{8}{\alpha_0}S_{\rm data}(t),
\]
where the constants $\alpha_0, \beta_0, \beta_0^{\rm in}$, and $\beta_1$ are those in Assumption \ref{assFriedrichs} 
and $I_{\lambda,t}(\cdot)$ and $S_{\lambda,t}^*(\cdot)$ are defined in \eqref{defIlambda} and \eqref{defSstar}. 
\end{proposition}

There are two standard notions of dissipative boundary conditions, namely, maximal and strictly dissipative boundary conditions. 
Both of them are algebraic conditions on the quadratic form ${\mathfrak B}[u]$, 
while the weak dissipativity only requires a condition for an integral in space and time of this quantity. 
We shall show that the boundary conditions we have to deal with in this article are neither strictly nor maximal dissipative, 
but that they however lead to weakly dissipative boundary contributions. 
Let us for the moment check that maximal and strictly dissipative boundary conditions satisfy the conditions in Definition \ref{propweakdissip}.

\begin{example}[Maximal dissipative boundary conditions]\label{BC2}
We say that the boundary condition \eqref{BClin} is \emph{maximal dissipative} if there is a constant $\beta_2>0$ such that 
\[
{\mathfrak B}[v]\leq \beta_2 \abs{Mv}^2 \quad\mbox{on}\quad (0,T)\times\mathit{\Gamma}
\]
holds for any $v\in\R^n$; in particular, ${\mathfrak B}[v]\leq 0$ on $\ker M$. 
In this case, \eqref{weakdissip} is satisfied with 
\[
E_{\rm bdry}(t)=0 \quad\mbox{and}\quad S_{\rm data}(t) = \beta_2 \abs{g}_{L^2_{\lambda,t}L^2(\mathit{\Gamma})}^2.
\]
Therefore, Proposition \ref{propNRJnear} provides an energy estimate controlling $\Abs{u(t)}_{L^2(\cE)}$, 
but there is no boundary energy granting further information on the trace of $u$. 
\end{example}

\begin{example}[Strictly dissipative boundary conditions]\label{BC3}
We say that the boundary condition \eqref{BClin} is \emph{strictly dissipative} if there are constants $\alpha_2>0$ and $\beta_2>0$ such that 
\[
{\mathfrak B}[v] \leq \beta_2 \abs{Mv}^2 - \alpha_2 \abs{v}^2 \quad\mbox{on}\quad (0,T)\times\mathit{\Gamma}
\]
holds for any $v\in\R^n$. 
In this case, \eqref{weakdissip} is satisfied with 
\[
E_{\rm bdry}(t) = \alpha_2 \abs{u}_{L^2_{\lambda,t}L^2(\mathit{\Gamma})}^2 \quad\mbox{and}\quad 
S_{\rm data}(t) = \beta_2 \abs{g}_{L^2_{\lambda,t}L^2(\mathit{\Gamma})}^2.
\]
In this particular configuration, the information we get from the boundary energy is a control of the $L^2((0,T)\times\mathit{\Gamma})$-norm 
of the trace of the solution that cannot be deduced through Sobolev embeddings from the control of $\Abs{u(t)}_{L^2(\cE)}$. 
\end{example}

We derived in Proposition \ref{propNRHhighfar} higher order energy estimates for solutions that vanish in the neighborhood of $\mathit{\Gamma}$. 
In the same spirit, for functions that do not necessarily vanish near $\mathit{\Gamma}$, 
we now want to generalize the energy estimate in Proposition \ref{propNRJnear} for higher regularity. 
We will actually consider solutions that are supported near $\mathit{\Gamma}$, that is, in $\overline{\cE} \cap U_\mathit{\Gamma}$, 
where, according to Section \ref{sectnoprmtang}, $U_\mathit{\Gamma}$ is a tubular neighborhood of $\mathit{\Gamma}$ 
in which normal-tangential coordinates can be used. 
We recall in particular that normal and tangential derivatives can be defined in $\overline{\cE} \cap U_\mathit{\Gamma}$ 
using the extensions ${N}$ and ${T}$ of the unit normal and tangential vectors to $\mathit{\Gamma}$ constructed in Section \ref{sectnoprmtang}; 
they are defined through the relations $\dnor={N}\cdot \nabla$ and $\dtan={T}\cdot \nabla$. 
Since ${N}$ is a unit vector while ${T}$ is not, we have the decomposition 
\[
\nabla = \abs{{T}}^{-2}{T}\dtan + {N}{\dnor}.
\]
In the neighborhood $\cE \cap U_\mathit{\Gamma}$, the system \eqref{eqlinfixed} can therefore be written equivalently under the form 
\begin{equation}\label{eqlinfixedbdry}
\dt u + A_{\rm tan}\dtan u + A_{\rm nor}\dnor u + Bu = f \quad\mbox{in}\quad (0,T)\times(\cE \cap U_\mathit{\Gamma}),
\end{equation}
where 
\[
A_{\rm tan} = \abs{T}^{-2}T_jA_j \quad\mbox{and}\quad A_{\rm nor} = N_jA_j.
\]

Two additional assumptions will be needed to derive these higher order energy estimates. 
The first one states that the weak dissipativity of Definition \ref{propweakdissip} is stable by time and tangential differentiations. 
This will allow us to obtain a control of the time and tangential derivatives of the solution using Proposition \ref{propNRJnear}; 
note that a control of the normal derivatives cannot be obtained along this procedure. 
We also recall that for $\alpha=(\alpha_0,\alpha_1)\in \N^2$ we write $\dpar^\alpha = \dt^{\alpha_0}\dtan^{\alpha_1}$.

\begin{definition} \label{propweakdissipm}
Let $m$ be a non-negative integer. 
We say that the boundary term ${\mathfrak B}[u]$ associated with a regular solution to \eqref{eqlinfixed} 
supported in $\overline{\cE}\cap U_\mathit{\Gamma}$ is \emph{weakly dissipative of order} $m$ if for each $\alpha\in\N^2$ satisfying 
$\abs{ \alpha }\leq m$ there exist a non-negative and non-decreasing function of time $S_{{\rm data},\alpha}[u](\cdot)$ 
depending only on the initial and boundary data possibly imposed on $u$, a non-negative \emph{boundary energy} function $E_{{\rm bdry},\alpha}(\cdot)$, 
and a non-negative continuous function $\nu_\alpha$ on $[0,\infty)$ satisfying $\nu_\alpha(0)=0$ such that for any $\lambda\geq\lambda_0$ with a 
positive constant $\lambda_0$ and any $t\in[0,T]$ we have 
\begin{align*}
&\int_0^t e^{-2\lambda t'} \biggl( \int_\mathit{\Gamma} {\mathfrak B}[\dpar^\alpha u(t')] \biggr){\rm d}t' \\
&\leq -E_{{\rm bdry},\alpha}(t)+S_{{\rm data},\alpha}(t) + \frac{\alpha_0}{8} I_{\lambda,t}(\Abs{\dpar^\alpha u(\cdot)}_{L^2(\cE)})^2
 + \alpha_0\nu_\alpha(\lambda^{-1})I_{\lambda,t}(\opnorm{ u(\cdot) }_m)^2,
\end{align*}
where $\alpha_0$ is the coercivity constant of the symmetrizer as defined in Assumption \ref{assFriedrichs}. 
In this case, we put 
\begin{equation}
E^m_{\rm bdry} = \sum_{\abs{\alpha}\leq m}E_{{\rm bdry},\alpha},\qquad
S^m_{\rm data} = \sum_{\abs{\alpha}\leq m}S_{{\rm data},\alpha},\qquad
\nu=\sum_{\abs{\alpha}\leq m}\nu_\alpha.
\end{equation}
\end{definition}

\begin{remark}
The component $\frac{\alpha_0}{8} I_{\lambda,t}(\Abs{\dpar^\alpha u(\cdot)}_{L^2(\cE)})^2$ in the right-hand side is the same 
as the one in Definition \ref{propweakdissip} with $u$ replaced by $\dpar^\alpha u$. 
The additional term $ \alpha_0\nu_\alpha(\lambda^{-1})I_{\lambda,t}(\opnorm{ u(\cdot) }_m)^2$ is here to control commutator terms; 
it contains non-tangential derivatives. 
\end{remark}

In order to deduce a control of the normal derivatives in terms of the time and tangential ones, 
we will use the equations under the assumption that the boundary is non-characteristic. 
In Section \ref{sectAENONCC} we treat this non-characteristic case. 
Note that in the nonlinear wave-structure interaction problem we are interested in here, 
this assumption is not directly satisfied and that an adaptation is necessary. 
To this end, in Section \ref{sectAECC} we consider the case where the boundary is not necessarily non-characteristic, 
but assume some structure of the equations under which one can introduce a generalized vorticity, 
which compensate the equations to control the normal derivatives.

\subsubsection{A priori energy estimates in the non-characteristic case}\label{sectAENONCC}
In this subsection, we consider the case where the boundary is non-characteristic in the following sense.

\begin{assumption}\label{assnonchar}
The problem \eqref{eqlinfixed} is \emph{non-characteristic}, that is, the boundary matrix $A_{\rm nor}$ is invertible in 
$(0,T)\times(\cE \cap U_\mathit{\Gamma})$ and there exists a constant $K_0$ such that 
\[
\abs{A_{\rm nor}(t,x)^{-1}} \leq K_0
\]
holds for any $(t,x) \in (0,T)\times(\cE \cap U_\mathit{\Gamma})$. 
\end{assumption}

\begin{proposition}\label{propfarorderm}
Let $m$ be a non-negative integer. 
Suppose that Assumptions \ref{assFriedrichs}--\ref{assnonchar} are satisfied and that 
the constants $K_0$ and $K$ in Assumptions \ref{ass:regAB} and \ref{assnonchar} are taken such that $\frac{\beta_0}{\alpha_0} \leq K_0$ and 
$\frac{\beta_1}{\alpha_0}\leq K$, where the constants $\alpha_0, \beta_0$, and $\beta_1$ are those in Assumption \ref{assFriedrichs}. 
Then, any regular solution $u$ to \eqref{eqlinfixed} supported in $\overline{\cE}\cap U_\mathit{\Gamma}$ and with a boundary term 
that is weakly dissipative of order $m$ in the sense of Definition \ref{propweakdissipm} satisfies 
\[
I_{\lambda,t}(\opnorm{u(t)}_m)^2 + \tfrac{1}{\alpha_0}E_{\rm bdry}^m(t)
\leq C(K_0) \bigl( \opnorm{u(0)}_m^2 + S^*_{\lambda,t}\big(\opnorm{f(\cdot)}_m\big)^2 + \tfrac{1}{\alpha_0}S^m_{\rm data}(t) \bigr)
\]
for any $\lambda\geq\lambda_0(K)$ and $t\in[0,T]$, where $\lambda_0(K)$ depends also on $\nu$. 
\end{proposition}

\begin{example}
If a boundary condition of the form \eqref{BClin}, with $M$ constant, is maximal or strictly dissipative, 
and if the data $g$ is in $H^m((0,T)\times\mathit{\Gamma})$, then the boundary term is weakly dissipative at order $m$. 
For strictly dissipative boundary conditions, for instance, we have, with the notations in Example \ref{BC3}, 
\[
E^m_{\rm bdry}(t)=\alpha_2 \sum_{\abs{\alpha}\leq m}\abs{\dpar^\alpha u}_{L^2_{\lambda,t}L^2(\mathit{\Gamma})}^2,\qquad
S^m_{\rm data}(t)=\beta_2 \sum_{\abs{\alpha}\leq m}\abs{\dpar^\alpha g}_{L^2_{\lambda,t}L^2(\mathit{\Gamma})}^2,
\]
and Proposition \ref{propfarorderm} therefore gives a control of the trace on $\mathit{\Gamma}$ of all time and tangential derivatives up to order $m$. 
Using the equation as in the proof of Lemma \ref{lemmcontrpar} below, one can deduce a control of the trace of \emph{all} derivatives up to order $m$. 
\end{example}

\begin{proof}[Proof of Proposition \ref{propfarorderm}]
We first note that according to Proposition \ref{propequivnorm}, it is possible, for functions supported in $U_\mathit{\Gamma}$, 
to replace the norm $\opnorm{ u(t)}_m$ with 
\[
\opnorm{u(t)}_{m,*} := \sum_{j+k+l\leq m} \Abs{u^{(j,k,l)}(t)}_{L^2(\cE)},
\]
where $u^{(j,k,l)}=\dt^j\dtan^k\dnor^l u$. 
We also denote by $\opnorm{u(t)}_{m,\parallel}$ the sum of all terms in the above expression that do not contain any normal derivative, that is, 
\[
\opnorm{u(t)}_{m,\parallel}:=\sum_{j+k\leq m} \Abs{u^{(j,k,0)}(t)}_{L^2(\cE)}.
\]
For the sake of conciseness, let us also write 
\[
I^m_{\lambda, t}(u):=I_{\lambda,t}(\opnorm{u(\cdot)}_{m,*}) \quad\mbox{and}\quad
I^{m,\parallel}_{\lambda, t}(u):=I_{\lambda,t}(\opnorm{u(\cdot)}_{m,\parallel}).
\]
The following lemma shows that the fact that the problem is non-characteristic allows us to control the full quantity 
$I^m_{\lambda,t}(u)$ of solutions by its non-normal version $I_{\lambda,t}^{m,\parallel}(u)$.

\begin{lemma}\label{lemmcontrpar}
Under the assumptions of Proposition \ref{propfarorderm}, there exists a constant $\lambda_0=\lambda_0(K)$ such that 
any regular solution $u$ to \eqref{eqlinfixed} supported in $\overline{\cE} \cap U_\mathit{\Gamma}$ satisfies 
\[
I^m_{\lambda,t}(u) \leq C(K_0)\bigl( I_{\lambda,t}^{m,\parallel}(u)+ \opnorm{u(0)}_m+S^*_{\lambda,t}(\opnorm{f(\cdot)}_m) \bigr)
\]
for any $\lambda\geq \lambda_0$ and $t\in[0,T]$. 
\end{lemma}

\begin{proof}
Applying ${\mathfrak d}^{(j,k,l)}=\dt^j\dtan^k\dnor^l $ with $j+k+l\leq m-1$ to the equations, we get 
\begin{equation}\label{equjkl}
u^{(j+1,k,l)}+A_{\rm tan}u^{(j,k+1,l)}+A_{\rm nor}u^{(j,k,l+1)} = f_{j,k,l},
\end{equation}
where 
\begin{equation}\label{fjkl}
f_{j,k,l} = {\mathfrak d}^{(j,k,l)}(f-Bu)-[{\mathfrak d}^{(j,k,l)},A_{\rm tan}]\dtan u-[{\mathfrak d}^{(j,k,l)},A_{\rm nor}]\dnor u.
\end{equation}
Since we assumed that $A_{\rm nor}$ is invertible, we can now use \eqref{equjkl} to write 
\begin{equation}\label{eqAnorinv}
\abs{u^{(j,k,l+1)}} \leq C(K_0) (\abs{u^{(j+1,k,l)}} + \abs{u^{(j,k+1,l)}} + \abs{f_{j,k,l}}),
\end{equation}
from which we deduce 
\[
I_{\lambda,t}^0(u^{(j,k,l+1)})
\leq C(K_0)\bigl( I_{\lambda,t}^0(u^{(j+1,k,l)}) + I_{\lambda,t}^0(u^{(j,k+1,l)}) + I_{\lambda,t}^{0}(f_{j,k,l}) \bigr). 
\]
There is one less normal derivative on $u$ in the right-hand side than in the left-hand one, so that this relation can be used inductively to obtain 
\[
I_{\lambda,t}^m(u) \leq C(K_0)\bigl( I_{\lambda,t}^{m,\parallel}(u)+ \sum_{j+k+l \leq m-1} I_{\lambda,t}^{0}(f_{j,k,l}) \bigr).
\]
We now need the following property 
\begin{equation}\label{eqdtS}
I_{\lambda,t}^{0}(\varphi)\leq C\big( \Abs{\varphi(0)}_{L^2}+S_{\lambda,t}^*(\Abs{\dt\varphi(\cdot)}_{L^2})
\end{equation}
for all smooth enough function $\varphi$ defined on $\R_+\times \cE$; see Lemma 2.16 in \cite{IguchiLannes} for a proof. 
Therefore, 
\[
I_{\lambda,t}^{0}(f_{j,k,l}) \leq C\bigl( \Abs{f_{j,k,l}(0)}_{L^2}+S^*_{\lambda,t}(\Abs{\dt f_{j,k,l}(\cdot)}_{L^2}) \bigr). 
\]
Here, on one hand, by regarding \eqref{equjkl} as an expression of $f_{j,k,l}$ we have 
\[
\Abs{f_{j,k,l}(0)}_{L^2}\leq C(K_0) \opnorm{u(0)}_m. 
\]
On the other hand, by the definition \eqref{fjkl} of $f_{j,k,l}$, we have also 
\[
\Abs{\dt f_{j,k,l}(t)}_{L^2}\leq \opnorm{f(t)}_{m,*}+C(K)\opnorm{u(t)}_{m}. 
\]
Therefore, we obtain 
\[
I_{\lambda,t}^m(u) \leq C(K_0)\bigl( I_{\lambda,t}^{m,\parallel}(u) + \opnorm{u(0)}_m + S^*_{\lambda,t}(\opnorm{f(\cdot)}_m) \bigr)
 + C(K)S^*_{\lambda,t}(\opnorm{u(\cdot)}_m).
\]
We conclude by recalling that, from \eqref{propSstar}, we have $S^*_{\lambda,t}(\opnorm{u(\cdot)}_m) \leq \frac{1}{\lambda}I^m_{\lambda,t}(u)$, 
so that the last term in the above inequality can be absorbed into the left-hand side when $\lambda$ is large enough depending only on $K$. 
\end{proof}

Owing to the lemma, it is enough to control the $L^2$-norm of time and tangential derivatives of the solution. 
For any $\alpha=(\alpha_0,\alpha_1)\in \N^2$, we write $\dpar^\alpha = \dt^{\alpha_0}\dtan^{\alpha_1}$. 
Applying $\dpar^\alpha$ to the equations \eqref{eqlinfixedbdry}, we get 
\begin{equation}\label{equalphanear}
\dt \dpar^\alpha u + A_{\rm tan}\dtan \dpar^\alpha u + A_{\rm nor}\dnor\dpar^\alpha u + B\dpar^\alpha u = f_\alpha,
\end{equation}
where 
\[
f_\alpha = \dpar^\alpha f - [\dpar^\alpha, A_{\rm tan}]\dtan u - [\dpar^\alpha, A_{\rm nor}]\dnor u - [\dpar^\alpha, B]u. 
\]
In particular, proceeding as in Lemma \ref{lemmcommut}, we get 
\[
\sum_{\abs{\alpha}\leq m}\Abs{f_\alpha(t)}_{L^2} \leq \opnorm{f(t)}_{m,*}+C(K)\opnorm{u(t)}_{m}.
\]
Applying Proposition \ref{propNRJnear} to \eqref{equalphanear} and summing over all $\abs{\alpha}\leq m$, we therefore obtain 
\begin{align*}
I^{m,\parallel}_{\lambda,t}(u) + \bigl( \tfrac{1}{\alpha_0}E^m_{\rm bdry} \bigr)^{1/2}
\leq& C(K_0) \bigl( \opnorm{u(0)}_m + S_{\lambda,t}^*(\opnorm{f(\cdot)}_m) + \bigl( \tfrac{1}{\alpha_0}S^m_{\rm data} \bigr)^{1/2} \bigr) \\
& + C(K)S_{\lambda,t}^*(\opnorm{u(\cdot)}_m) + C(K_0)\nu(\lambda^{-1})I^{m}_{\lambda,t}(u).
\end{align*}
Lemma \ref{lemmcontrpar} allows us to replace $I^{m,\parallel}_{\lambda,t}(u)$ with $I^{m}_{\lambda,t}(u)$ in the left-hand side, and 
the last two terms in the right-hand side can be absorbed for $\lambda $ large enough into the left-hand side as in the proof of Lemma \ref{lemmcontrpar}. 
This concludes the proof. 
\end{proof}

\subsubsection{A priori energy estimates in a characteristic case}\label{sectAECC}
We proceed to consider the case where the boundary is not necessarily non-characteristic and impose structural conditions on the equations. 
Since the lower order term $B(t,x)u$ in \eqref{eqlinfixed} does not contribute to the main structure of the equations in our analysis, 
we absorb the term in the right-hand side and consider in this subsection the equations 
\begin{equation}\label{eqlinbis}
\dt u +A_j(t,x)\partial_j u = f(t,x) \quad\mbox{in}\quad (0,T)\times\cE.
\end{equation}
 {Multiplying a matrix $\widetilde{A}_0(t,x)$ to this equation, we have} 
%
\[
\widetilde{A}_0(t,x)\dt u + \widetilde{A}_j(t,x)\partial_j u = \widetilde{f}(t,x) \quad\mbox{in}\quad (0,T)\times\cE,
\]
where $\widetilde{A}_j=\widetilde{A}_0A_j$ $(j=1,2)$ and $\widetilde{f}=\widetilde{A}_0f$. 
The corresponding boundary matrix is given by $\widetilde{A}_{\rm nor}=N_j\widetilde{A}_j=\widetilde{A}_0A_{\rm nor}$ with $A_{\rm nor}=N_jA_j$, 
for which we impose the following assumptions.

\begin{assumption}\label{asschar}
For any $(t,x)\in(0,T)\times(\cE\cap U_{\mathit{\Gamma}})$, $\widetilde{A}_{\rm nor}(t,x)$ has eigenvalues $\lambda_j(t,x)$ $(j=1,2,\ldots,n)$ 
with associated left eigenvectors $\bm{l}_j(t,x)$, which satisfy the following properties: 
\begin{enumerate}
\item[\rm (i)]
The last $n_2$ eigenvalues $\lambda_{n_1+j}$ $(j=1,2,\ldots,n_2)$ are positive or negative definite, where $n_1+n_2=n$. 

\item[\rm (ii)]
The left eigenvectors $\bm{l}_j$ associated with the first $n_1$ eigenvalues $\lambda_j$ $(j=1,2,\ldots,n_1)$ can be written as 
\[
\bm{l}_j(t,x) = N_1(x)\bm{q}_{j,1}(t,x) + N_2(x)\bm{q}_{j,2}(t,x),
\]
where $\R^n$-valued functions $\bm{q}_{j,1}$ and $\bm{q}_{j,2}$ satisfy the relations 
\[
\begin{cases}
 \bm{q}_{j,1}^{\rm T}(\widetilde{A}_1-w_{j,1}\widetilde{A}_0) = \bm{q}_{j,2}^{\rm T}(\widetilde{A}_2-w_{j,2}\widetilde{A}_0) = \bm{0}^{\rm T}, \\
 \bm{q}_{j,1}^{\rm T}(\widetilde{A}_2-w_{j,2}\widetilde{A}_0) + \bm{q}_{j,2}^{\rm T}(\widetilde{A}_1-w_{j,1}\widetilde{A}_0) = \bm{0}^{\rm T}
\end{cases}
\]
for some $\R^2$-valued functions $w_j=(w_{j,1},w_{j,2})^{\rm T}$ for $j=1,2,\ldots,n_1$. 

\item[\rm (iii)]
The $n\times n$ matrix $L(t,x)$ defined by 
$L=(\widetilde{A}_0^{\rm T}\bm{l}_1,\ldots,\widetilde{A}_0^{\rm T}\bm{l}_{n_1},\bm{l}_{n_1+1},\ldots,\bm{l}_{n_1+n_2})$ is invertible for any 
$(t,x)\in(0,T)\times(\cE\cap U_{\mathit{\Gamma}})$. 

\item[\rm (iv)]
We have $N\cdot w_j \leq 0$ on $(0,T)\times\mathit{\Gamma}$ for $j=1,2,\ldots,n_1$. 
\end{enumerate}
\end{assumption}

\begin{remark}
For $j=1,2,\ldots,n_1$ and $k=1,2$, $w_{j,k}(t,x)$ is necessarily an eigenvalue of the matrix  {$A_k(t,x)$} 
associated with a left eigenvector $\widetilde{A}_0(t,x)^{\rm T}\bm{q}_{j,k}(t,x)$ if $\widetilde{A}_0(t,x)^{\rm T}\bm{q}_{j,k}(t,x)\ne\bm{0}$. 
Particularly, we see that $N\cdot w_j$ is an eigenvalue of the boundary matrix  {$A_{\rm nor}$}. 
\end{remark}

Under these assumptions, we can define a generalized vorticity $\omega=(\omega_1,\omega_2,\ldots,\omega_{n_1})^{\rm T}$ by 
\begin{equation}\label{defGVor}
\omega_j = \bm{q}_{j,k}^{\rm T}\widetilde{A}_0\partial_k u
\end{equation}
for $j=1,2,\ldots,n_1$, where we used Einstein's convention on the index $k$. 
It follows from \eqref{eqlinbis} that this generalized vorticity $\omega$ satisfies the equations 
\begin{equation}\label{EqVor}
\dt\omega_j+w_j\cdot\nabla\omega_j=F_j \quad\mbox{in}\quad (0,T)\times(\cE\cap U_{\mathit{\Gamma}})
\end{equation}
for $j=1,2,\ldots,n_1$, where 
\begin{align*}
F_j
&= ((\dt+w_j\cdot\nabla)(\bm{q}_{j,k}^{\rm T}\widetilde{A}_0))\partial_k u
 + \bm{q}_{j,k}^{\rm T}\widetilde{A}_0(\partial_k f - (\partial_kA_l)\partial_l u).
\end{align*}

\begin{example}\label{ExVor}
The nonlinear shallow water equations \eqref{SWEinE} can be written as in the form \eqref{eqlinbis} with $u=(\zeta,v^{\rm T})^{\rm T}$ and 
\[
A_j =
\begin{pmatrix}
 v_j & h{\bf e}_{j}^{\rm T} \\
 \gr {\bf e}_j & v_j\mbox{\rm Id}_{2\times 2}
\end{pmatrix}
\]
for $j=1,2$. 
By choosing $\widetilde{A}_0=\mbox{\rm Id}_{3\times3}$, the boundary matrix $\widetilde{A}_{\rm nor}=A_{\rm nor}$ is given by 
\begin{equation}\label{BM}
A_{\rm nor} = 
\begin{pmatrix}
 N\cdot v & h N^{\rm T} \\
 \gr N & (N\cdot v){\rm Id_{2\times 2}}
\end{pmatrix}
\end{equation}
and has eigenvalues $\lambda_1=N\cdot v$, 
$\lambda_2=N\cdot v+\sqrt{gh}$, and $\lambda_3=N\cdot v-\sqrt{gh}$. 
Under the subcriticality condition 
\begin{equation}\label{subcritical}
\inf_{(t,x)\in(0,T)\times\cE}(\gr h(t,x)-|v(t,x)|^2)>0,
\end{equation}
the last two eigenvalues are positive and negative definite, respectively, so that we can choose $n_1=1$ and $n_2=2$. 
A left eigenvector $\bm{l}_1$ associated to the first eigenvalue $\lambda_1$ is given by $(0,(N^\perp)^{\rm T})^{\rm T}$ so that 
we have $\bm{q}_{1,1}=(0,{\bf e}_2^{\rm T})^{\rm T}$ and $\bm{q}_{1,2}=-(0,{\bf e}_1^{\rm T})^{\rm T}$. 
Therefore, the corresponding generalized vorticity is given by $\omega=\bm{q}_{1,k}^{\rm T}\partial_k u=\nabla^\perp\cdot v$, 
which is nothing but the vorticity of the velocity field $v$. 
This is the reason why we call $\omega$ as a generalized vorticity. 
Moreover, the $\R^2$-valued function $w_1$ is now given by $v$, so that the condition (ii) in Assumption \ref{asschar} is satisfied. 
The left eigenvectors associated with the last two eigenvalues are given by 
$\bm{l}_2=(\sqrt{\gr h},hN^{\rm T})^{\rm T}$ and $\bm{l}_3=(-\sqrt{\gr h},hN^{\rm T})^{\rm T}$, so that we have 
$\det L = 2h\sqrt{\gr h}$ and that the condition (iii) is also satisfied. 
Finally, the condition (iv) is reduced to $N\cdot v\leq0$ on $(0,T)\times\mathit{\Gamma}$. 
\end{example}

We put $Q = (\bm{q}_{1,1},\bm{q}_{1,2},\ldots,\bm{q}_{n_1,1},\bm{q}_{n_1,2})$ and $W = (w_1,\ldots,w_{n_1})$.

\begin{assumption}\label{ass:regLQW}
There exist two constants $0<K_0\leq K$ such that the following conditions hold. 
\begin{enumerate}
\item[\rm (i)]
$\|(\lambda_{n_1}^{-1},\ldots,\lambda_{n_1+n_2}^{-1})\|_{L^\infty((0,T)\times(\cE\cap U_{\mathit{\Gamma}})} \leq K_0$.

\item[\rm (ii)]
$\|(\widetilde{A}_0,A_1,A_2,L,L^{-1},Q)\|_{L^\infty((0,T)\times(\cE\cap U_{\mathit{\Gamma}})} \leq K_0$.
\item[\rm (iii)]
$\|(\widetilde{A}_0,A_1,A_2,Q,W)\|_{W^{1,\infty}((0,T)\times(\cE\cap U_{\mathit{\Gamma}})} \leq K$.
\end{enumerate}
\end{assumption}

\begin{lemma}\label{lemmvor}
Suppose that Assumptions \ref{asschar} and \ref{ass:regLQW} are satisfied. 
Then, any regular solution $u$ to \eqref{eqlinbis} supported in $\overline{\cE}\cap U_\mathit{\Gamma}$ satisfies 
\begin{align*}
I_{\lambda,t}( \|\dnor u(\cdot)\|_{L^2})
\leq C(K_0) \bigl( \opnorm{ u(0) }_1 + I_{\lambda,t}(\|\dpar u(\cdot)\|_{L^2}) \bigr) + C(K)S_{\lambda,t}^*(\opnorm{ f(\cdot) }_1)
\end{align*}
for any $\lambda\geq\lambda_0(K)$ and $t\in[0,T]$. 
\end{lemma}

\begin{proof}
Since $u$ is supported in $\overline{\cE}\cap U_{\mathit{\Gamma}}$, 
we can rewrite the equation under the form $\widetilde{A}_{\rm nor}\dnor u = \widetilde{A}_0(f-\dt u-A_{\rm tan}\dtan u)$. 
Taking the scalar product of this equation with the last $n_2$ left eigenvectors $\bm{l}_j$, we have 
$\bm{l}_j\cdot\dnor u=\lambda_j^{-1}\bm{l}_j\cdot\widetilde{A}_0(f-\dt u-A_{\rm tan}\dtan u)$, which yields 
\[
|\bm{l}_j\cdot\dnor u| \leq C(K_0)(|f|+|\dpar u|) \quad\mbox{for}\quad n_1+1\leq j\leq n_1+n_2.
\]
By the definition \eqref{defGVor} of the generalized vorticity $\omega$, we have 
$\omega_j=\bm{l}_j\cdot\widetilde{A}_0\dnor u + |T|^{-2}T_k\bm{q}_{j,k}\cdot\widetilde{A}_0\dtan u$, which yields 
\[
|\widetilde{A}_0^{\rm T}\bm{l}_j\cdot\dnor u| \leq C(K_0)( |\omega_j| + |\dtan u| ) \quad\mbox{for}\quad 1\leq j\leq n_1.
\]
Since $L=(\widetilde{A}_0^{\rm T}\bm{l}_1,\ldots,\widetilde{A}_0^{\rm T}\bm{l}_{n_1},\bm{l}_{n_1+1},\ldots,\bm{l}_{n_1+n_2})$ is invertible, 
these two estimates imply 
\begin{equation}\label{EstNoru}
I_{\lambda,t}( \|\dnor u(\cdot)\|_{L^2})
\leq C(K_0) \bigl( I_{\lambda,t}( \|f(\cdot)\|_{L^2}) + I_{\lambda,t}( \|\omega(\cdot)\|_{L^2}) + I_{\lambda,t}(\|\dpar u(\cdot)\|_{L^2}) \bigr).
\end{equation}

We proceed to evaluate the generalized vorticity $\omega$. 
Since $\omega_j$ satisfies \eqref{EqVor}, we see that 
\begin{align*}
&\frac{\rm d}{{\rm d}t} \bigl\{ e^{-2\lambda t }\|\omega_j(t)\|_{L^2}^2 \bigr\}
 + 2\lambda e^{-2\lambda t}\|\omega_j(t)\|_{L^2}^2 \\
&= 2e^{-2\lambda t}(\omega_j(t),2F_j(t)+(\nabla\cdot w_j)\omega_j)_{L^2} + e^{-2\lambda t}\int_\mathit{\Gamma} (N\cdot w_j)\omega_j^2.
\end{align*}
In view of (iv) in Assumption \ref{asschar}, we can drop the last integral on $\mathit{\Gamma}$. 
Therefore, by similar calculations to the proof of Proposition \ref{propNRJL2far} we obtain 
\begin{equation}\label{EstVor}
I_{\lambda,t}( \|\omega_j(\cdot)\|_{L^2})
\leq C( \|\omega_j(0)\|_{L^2} + S_{\lambda,t}^*( \|F_j(\cdot)\|_{L^2} ).
\end{equation}

Here, we see that $\|\omega_j(0)\|_{L^2} \leq C(K_0)\|u(0)\|_{H^1}$ and that 
\begin{align*}
S_{\lambda,t}^*( \|F_j(\cdot)\|_{L^2} )
&\leq C(K)\bigl( S_{\lambda,t}^*( \|\nabla f(\cdot)\|_{L^2} ) + S_{\lambda,t}^*( \|\nabla u(\cdot)\|_{L^2} ) \bigr) \\
&\leq C(K)S_{\lambda,t}^*( \opnorm{ f(\cdot) }_1 ) \\
&\quad\;
 + \lambda^{-1}C(K) \bigl( I_{\lambda,t}( \|\dtan u(\cdot)\|_{L^2} ) + I_{\lambda,t}( \|\dnor u(\cdot)\|_{L^2} ) \bigr).
\end{align*}
Moreover, by \eqref{eqdtS} we see also that 
\begin{align*}
I_{\lambda,t}( \|f(\cdot)\|_{L^2})
&\leq C\bigl( \|f(0)\|_{L^2} + S_{\lambda,t}^*( \|\dt f(\cdot)\|_{L^2} ) \bigr) \\
&\leq C\bigl( \opnorm{ u(0) }_1 + S_{\lambda,t}( \opnorm{ f(\cdot) }_1 ) \bigr).
\end{align*}
These estimates together with \eqref{EstNoru} and \eqref{EstVor} give the desired estimate. 
\end{proof}

\begin{proposition}\label{propfarorderm2}
Let $m$ be a non-negative integer. 
Suppose that Assumptions \ref{assFriedrichs}, \ref{ass:regAB}, \ref{asschar}, and \ref{ass:regLQW} are satisfied and that 
the constants $K_0$ and $K$ in Assumptions \ref{ass:regAB} and \ref{ass:regLQW} are taken such that $\frac{\beta_0}{\alpha_0} \leq K_0$ and 
$\frac{\beta_1}{\alpha_0}\leq K$, where the constants $\alpha_0, \beta_0$, and $\beta_1$ are those in Assumption \ref{assFriedrichs}. 
Then, any regular solution $u$ to \eqref{eqlinfixed} supported in $\overline{\cE}\cap U_\mathit{\Gamma}$ and with a boundary term 
that is weakly dissipative of order $m$ in the sense of Definition \ref{propweakdissipm} satisfies 
\[
I_{\lambda,t}(\opnorm{u(t)}_m)^2 + \tfrac{1}{\alpha_0}E_{\rm bdry}^m(t)
\leq C(K_0) \bigl( \opnorm{u(0)}_m^2 + S^*_{\lambda,t}\big(\opnorm{f(\cdot)}_m\big)^2 + \tfrac{1}{\alpha_0}S^m_{\rm data}(t) \bigr)
\]
for any $\lambda\geq\lambda_0(K)$ and $t\in[0,T]$, where $\lambda_0(K)$ depends also on $\nu$. 
\end{proposition}

\begin{proof}
Thanks to Lemma \ref{lemmvor}, the proof is almost the same as that of Proposition \ref{propfarorderm} 
so that we point out only the place where we need modifications. 
Applying ${\mathfrak d}^{(j,k,l)}=\dt^j\dtan^k\dnor^l $ with $j+k+l\leq m-1$ to the equations, we get 
\[
\dt u^{(j,k,l)}+A_1\partial_1 u^{(j,k,l)}+A_2\partial_2 u^{(j,k,l)} = f_{j,k,l},
\]
where $f_{j,k,l}$ is given by \eqref{fjkl}. 
By Lemma \ref{lemmvor}, we obtain 
\begin{align*}
I_{\lambda,t}( \|u^{(j,k,l+1)}(\cdot)\|_{L^2} )
&\leq C(K_0)\bigl( \opnorm{ u^{(j,k,l)}(0) }_1 + I_{\lambda,t}( \|\dpar u^{(j,k,l)}(\cdot)\|_{L^2} \bigr) \\
&\quad\;
 + C(K) S_{\lambda,t}^*( \opnorm{ f_{j,k,l}(\cdot) }_1 ).
\end{align*}
Here, we have $\opnorm{ f_{j,k,l}(t) }_1 \leq C(K)( \opnorm{ f(t) }_m + \opnorm{ u(t) }_m )$, 
so that by using the above estimate inductively, we obtain 
\begin{align*}
I_{\lambda,t}^m(u)
&\leq C(K_0)( \opnorm{ u(0) }_m + I_{\lambda,t}^{m,\parallel}(u) ) 
 + C(K)S_{\lambda,t}^*( \opnorm{ f(\cdot) }_m ) + \lambda^{-1}C(K)I_{\lambda,t}^m(u).
\end{align*}
Here, the last term can be absorbed in the left-hand side if we take $\lambda$ sufficiently large. 
Once we obtain such an estimate, the rest of the proof is the same as that of Proposition \ref{propfarorderm}. 
\end{proof}

\subsection{A priori energy estimates in the general case}\label{sectAEGC}
Combining the results of Propositions \ref{propNRHhighfar} and \ref{propfarorderm}, 
we show that one can remove the assumption that $u$ is compactly supported in $\overline{\cE}\cap U_\mathit{\Gamma}$ in Proposition \ref{propfarorderm}.

\begin{theorem}\label{theogenorderm}
Let $m$ be a non-negative integer. 
Suppose that Assumptions \ref{assFriedrichs}--\ref{assnonchar} are satisfied and that the constants $K_0$ and $K$ 
in Assumptions \ref{ass:regAB} and \ref{assnonchar} are taken such that $\frac{\beta_0}{\alpha_0}$ and $\frac{\beta_1}{\alpha_0}\leq K$, 
where the constants $\alpha_0, \beta_0$, and $\beta_1$ are those in Assumption \ref{assFriedrichs}. 
Then, any regular solution $u$ to \eqref{eqlinfixed} with a boundary term that is weakly dissipative of order $m$ in the sense of 
Definition \ref{propweakdissipm} satisfies 
\[
I_{\lambda,t}(\opnorm{u(t)}_m)^2 + \tfrac{1}{\alpha_0}E_{\rm bdry}^m(t)
\leq C(K_0) \bigl( \opnorm{u(0)}_m^2 + S^*_{\lambda,t}\big(\opnorm{f(\cdot)}_m\big)^2 + \tfrac{1}{\alpha_0}S^m_{\rm data}(t) \bigr)
\]
for any $\lambda\geq\lambda_0(K)$ and $t\in[0,T]$, where $\lambda_0(K)$ depends also on $\nu$ in Definition \ref{propweakdissipm}. 
\end{theorem}

\begin{proof}
Let $\chi_{\rm b}$ be the smooth cutoff function supported in $U_\mathit{\Gamma}$ defined in \eqref{defchib}. 
We can decompose $u$ as $u=u_1+u_2$ with $u_1=\chi_{\rm b} u$ and $u_2=(1-\chi_{\rm b} )u$. 
Then, we have 
\[
\dt u_\ell + A_j\partial_j u_\ell + Bu_\ell = f_\ell \qquad (\ell=1,2),
\]%
where 
\[
f_1 = \chi_{\rm b} f + (\partial_j \chi_{\rm b})A_j u \quad\mbox{and}\quad
f_2 = (1-\chi_{\rm b}) f - (\partial_j \chi_{\rm b})A_j u.
\]
We can therefore apply Proposition \ref{propfarorderm} to $u_1$ and Proposition \ref{propNRHhighfar} to $u_2$, 
which of course has exactly the same weak dissipativity properties as $u$, 
and deduce by adding the two energy estimates and using the triangular inequality that 
\begin{align*}
& I_{\lambda,t}(\opnorm{u(t)}_m)^2 + \tfrac{1}{\alpha_0}E_{\rm bdry}^m(t) \\
&\leq C(K_0)\bigl( \opnorm{u(0)}_m^2 + S^*_{\lambda,t}(\opnorm{f(\cdot)}_m)^2 + \tfrac{1}{\alpha_0}S^m_{\rm data}(t)
 + S^*_{\lambda,t}(\opnorm{u(\cdot)}_m)^2 \bigr).
\end{align*}
The term $S^*_{\lambda,t}\big(\opnorm{u(\cdot)}_m\big)^2$ in the right-hand side, 
which comes from the terms involving derivatives of $\chi$ in $f_1$ and $f_2$, 
can be absorbed into the left-hand side for $\lambda$ large enough, which yields the desired result. 
\end{proof}

Similarly, combining the results of Propositions \ref{propNRHhighfar} and \ref{propfarorderm2}, 
we show that one can remove the assumption that $u$ is compactly supported in $\overline{\cE}\cap U_\mathit{\Gamma}$ in Proposition \ref{propfarorderm2}. 
Since the proof is the same as that of Theorem \ref{theogenorderm}, we omit it.

\begin{theorem}\label{theogenorderm2}
Let $m$ be a non-negative integer. 
Assumptions \ref{assFriedrichs}, \ref{ass:regAB}, \ref{asschar}, and \ref{ass:regLQW} are satisfied and that 
the constants $K_0$ and $K$ in Assumptions \ref{ass:regAB} and \ref{ass:regLQW} are taken such that $\frac{\beta_0}{\alpha_0} \leq K_0$ and 
$\frac{\beta_1}{\alpha_0}\leq K$, where the constants $\alpha_0, \beta_0$, and $\beta_1$ are those in Assumption \ref{assFriedrichs}. 
Then, any regular solution $u$ to \eqref{eqlinfixed} with a boundary term that is weakly dissipative of order $m$ in the sense of 
Definition \ref{propweakdissipm} satisfies 
\[
I_{\lambda,t}(\opnorm{u(t)}_m)^2 + \tfrac{1}{\alpha_0}E_{\rm bdry}^m(t)
\leq C(K_0) \bigl( \opnorm{u(0)}_m^2 + S^*_{\lambda,t}\big(\opnorm{f(\cdot)}_m\big)^2 + \tfrac{1}{\alpha_0}S^m_{\rm data}(t) \bigr)
\]
for any $\lambda\geq\lambda_0(K)$ and $t\in[0,T]$, where $\lambda_0(K)$ depends also on $\nu$ in Definition \ref{propweakdissipm}. 
\end{theorem}

\section{A priori estimates for the nonlinear wave-structure interaction problem}\label{sectAPNLWS}
The goal of this section is to derive a priori estimate for solutions $u=(\zeta,v^{\rm T})^{\rm T}$ 
to the nonlinear wave-structure interaction problem \eqref{PB1}--\eqref{PB3}. 
We remark that \eqref{PB1} forms a hyperbolic quasilinear system of the form considered in the previous section, that is, $u$ solves 
\[
\dt u + A_j(u)\partial_j u = 0 \quad\mbox{in}\quad (0,T)\times\cE,
\]
where the matrices $A_j(u)$ $(j=1,2)$ are defined by 
\[
A_j(u)=\begin{pmatrix}
 v_j & h{\bf e}_{j}^{\rm T} \\
 \gr {\bf e}_j & v_j\mbox{Id}_{2\times 2}
\end{pmatrix}
\]
with $h=H_0+\zeta$. 
There is a conservation of energy for \eqref{PB1}--\eqref{PB3} from which a control of the $L^2$-norm of $u$ can be deduced. 
In order to derive higher order energy estimates for solutions to \eqref{PB1}--\eqref{PB3}, 
we consider the system satisfied by derivatives of $u$, which is a linearized system of the form 
\begin{equation}\label{pblin1}
\dt\dot{u}+A_j(u)\partial_j\dot{u} = f \quad\mbox{in}\quad (0,T)\times\cE,
\end{equation}
for which an obvious Friedrichs symmetrizer in the sense of Assumption \ref{assFriedrichs} is given by $S(u)=\mbox{diag}(\gr,h,h)$. 
Of course, $\dot{u}$ should also satisfy a linearization of the boundary conditions \eqref{PB2}--\eqref{PB3}. 
The natural strategy is to prove that solutions to \eqref{pblin1} subject to these linearized boundary conditions satisfy 
the weak dissipativity properties introduced in Section \ref{sectapriorilin}, 
so that a control of $\dot{u}$ is granted by Theorem \ref{theogenorderm}. 
There are however at least two obstructions.

The first obstruction is related to the boundary matrix $A_{\rm nor}$ given by \eqref{BM}. 
As explained in Example \ref{ExVor}, the eigenvalues of this matrix are $N\cdot v$, $N\cdot v + \sqrt{\gr h}$, and $N\cdot v - \sqrt{\gr h}$. 
While the last two eigenvalues are positive and negative, respectively, under the subcriticality condition \eqref{subcritical}, 
the first one may vanish and even change the sign. 
This is known to be a very delicate situation for the study of initial boundary value problem to the hyperbolic system. 
Classical results based on the construction of Kreiss symmetrizers cannot be applied and there is no general theory covering such situations; 
see \cite{Rauch, Metivier} as well as Chapter 9 of \cite{BenzoniSerre}. 
Here, the boundary term \eqref{defB} associated with the Friedrichs symmetrizer $S(u)=\mbox{diag}(\gr,h,h)$ is given by 
\[
{\mathfrak B}[u] = (N\cdot v) \bigl(\gr \zeta^2 + h \abs{v}^2 + 2\gr h \zeta \bigr),
\]
and it does not seem to satisfy any dissipative property, even in the weak sense of Definition \ref{propweakdissip}.

The second obstruction to the use of the results of Section \ref{sectapriorilin} is that 
the non-characteristic property of Assumption \ref{assnonchar} is obviously not satisfied.

We proceed in several steps to bypass these obstructions. 
Firstly, we use in Section \ref{sectequivform} the irrotationality condition \eqref{PBirrot} to transform the problem \eqref{PB1}--\eqref{PB3}, 
for which the boundary matrix has an eigenvalue of indefinite sign, into another problem in which the eigenvalue of indefinite sign is replaced with $0$. 
For this new formulation, the second boundary condition in \eqref{PB2} is also removed. 
In Section \ref{sectNRJcons}, we prove that the nonlinear system conserves an energy that controls the $L^2$-norm of the solution. 
As explained above, it is necessary, for higher-order estimates, to linearize the equations. 
To this end we use a non-standard linearization process and a change of unknowns that exploit the structure of the system 
and greatly simplify the boundary conditions. 
Such a linearization is explained first in Section \ref{sectlintoy} on a toy problem, 
and then in Section \ref{sectlinWS} on the original equations to the nonlinear wave-structure interaction problem. 
We then derive in Section \ref{sectWSL2} $L^2$-energy estimates for this linear system; 
the main step is to prove that the boundary conditions are weakly dissipative in the sense of Definition \ref{propweakdissip}. 
Higher order estimates are then provided in Section \ref{sectHOest}. 
The main difficulty there is that the problem does not possess the non-characteristic property in Assumption \ref{assnonchar}. 
However, we check that it satisfies properties in Assumption \ref{asschar} so that one can introduce a generalized vorticity 
to obtain higher order energy estimates. 
The nonlinear estimate can then be established in Section \ref{sectNLAP}.

\subsection{An equivalent formulation}\label{sectequivform}
Remarking that 
\begin{equation}\label{irrotid}
\nabla \bigl( \tfrac{1}{2}\abs{v}^2 \bigr) = (v\cdot \nabla)v - (\nabla^\perp \cdot v)v^\perp,
\end{equation}
we see that solutions $u=(\zeta,v^{\rm T})^{\rm T}$ to \eqref{PB1}--\eqref{PB3} satisfying the irrotational condition $\nabla^\perp \cdot v=0$ 
also solve the following system of equations 
\begin{equation}\label{PB1bis}
\begin{cases}
 \dt \zeta + \nabla\cdot (hv) = 0 &\mbox{in}\quad (0,T)\times\cE, \\
 \dt v + \nabla \big( \gr\zeta + \tfrac{1}{2}\abs{v}^2 \bigr) = 0 &\mbox{in}\quad (0,T)\times\cE,
\end{cases}
\end{equation}
and boundary conditions 
\begin{equation}\label{PB2bis}
N\cdot (hv) = \Lambda \psi_{\rm i} \quad\mbox{on}\quad (0,T)\times\mathit{\Gamma},
\end{equation}
where $\psi_{\rm i}$ is found by solving the forced ODE 
\begin{equation}\label{PB3bis}
\dt\psi_{\rm i} = - \gr \zeta - \tfrac{1}{2}\abs{v}^2 \quad\mbox{on}\quad (0,T)\times\mathit{\Gamma}. 
\end{equation}
The main difference between \eqref{PB1}--\eqref{PB3} and \eqref{PB1bis}--\eqref{PB3bis} is that the boundary condition 
on $N^\perp\cdot v$ imposed in \eqref{PB2} is no longer present in \eqref{PB2bis}. 
The following proposition shows that, for initial data with an irrotational velocity, the initial boundary value problem associated with 
\eqref{PB1}--\eqref{PB3} is equivalent to the one associated with \eqref{PB1bis}--\eqref{PB3bis}. 
From now on, we will therefore work with the latter. 
We impose the initial conditions of the form 
\begin{equation}\label{IC}
\begin{cases}
 (\zeta,v)_{\vert_{t=0}} = (\zeta^{\rm in}, v^{\rm in}) &\mbox{in}\quad \cE, \\
 {\psi_{\rm i}}_{\vert_{t=0}} = \psi_{\rm i}^{\rm in} &\mbox{on}\quad \mathit{\Gamma}. 
\end{cases}
\end{equation}

\begin{proposition}\label{prop:equiv}
Let $(\zeta^{\rm in},\phi^{\rm in})$ be smooth functions defined on $\cE$ and put $v^{\rm in}=\nabla\phi^{\rm in}$, $h^{\rm in}=H_0+\zeta^{\rm in}$, 
and $\psi_{\rm i}^{\rm in} = {\phi^{\rm in}}_{\vert_\mathit{\Gamma}}$. 
If the compatibility condition 
\[
N\cdot(h^{\rm in}v^{\rm in}) = \Lambda \psi_{\rm i}^{\rm in} \quad\mbox{on}\quad \mathit{\Gamma}
\]
holds, then $(\zeta,v,\psi_{\rm i})$ solves \eqref{PB1}--\eqref{PB3} and \eqref{IC} 
if and only if it solves \eqref{PB1bis}--\eqref{PB3bis} and \eqref{IC}. 
Moreover, we have $v=\nabla\phi$ with $\phi = \phi^{\rm in}-\int_0^t(\gr \zeta + \frac{1}{2}\abs{v}^2){\rm d}t'$ and 
$\psi_{\rm i}=\phi_{\vert_\mathit{\Gamma}}$. 
\end{proposition}

\begin{proof}
Let us prove first that solutions to \eqref{PB1bis}--\eqref{PB3bis} and \eqref{IC} also solve \eqref{PB1}--\eqref{PB3} 
for initial data as in the statement of the proposition; in particular, $\nabla^\perp\cdot v^{\rm in}=0$. 
The irrotationality is obviously conserved by the flow so that \eqref{PBirrot} is satisfied. 
Moreover, \eqref{irrotid} shows that \eqref{PB1} is also satisfied. 
It remains to show that the second boundary condition in \eqref{PB2} is satisfied, namely, that $T\cdot v = \dtan \psi_{\rm i}$ on $\mathit{\Gamma}$. 
Multiplying the second equation in \eqref{PB1bis} by $T$ and using \eqref{PB3bis}, we get that $\dt (T\cdot v-\dtan \psi_{\rm i})=0$ on $\mathit{\Gamma}$. 
Since $T\cdot v^{\rm in}-\dtan \psi_{\rm i}^{\rm in}=0$, this implies that $T\cdot v=\dtan \psi_{\rm i}$ on $\mathit{\Gamma}$ for all times. 
The reverse implication follows directly from \eqref{irrotid}. 
Finally, the last assertion of the proposition follows from the observation that writing the second equation in \eqref{PB1bis} in an integral form 
and using the fact that $v^{\rm in}=\nabla \phi^{\rm in}$ yields 
\[
v = \nabla\left( \phi^{\rm in} - \int_0^t (\gr \zeta+\tfrac{1}{2}\abs{v}^2){\rm d}t' \right). 
\]
Moreover, ${\phi^{\rm in}}_{\vert_\mathit{\Gamma}} = \psi_{\rm i}^{\rm in}$ and $\dt (\phi_{\vert_\mathit{\Gamma}}) = \dt \psi_{\rm i}$, 
so that $\phi_{\vert_\mathit{\Gamma}}=\psi_{\rm i}$ for all times. 
\end{proof}

\subsection{Energy conservation}\label{sectNRJcons}
We show here that the total energy of the fluid 
\[
E_{\rm fluid} = \frac{1}{2}\int_{\cE}\gr\zeta^2 + \frac{1}{2}\int_{\cE} h\abs{v}^2 + \frac{1}{2}\int_\mathit{\Gamma} \psi_{\rm i}\Lambda \psi_{\rm i}
\]
is preserved by the flow. 
This is not surprising since the above expression is a reformulation of the fluid energy considered in Section \ref{constotNRJ}; 
we have actually $E_{\rm fluid}=\frac{1}{\rho}{\mathfrak E}_{\rm fluid}$ so that $E_{\rm fluid}$ is not an energy in physical units. 
It is however instructive to derive this property directly from \eqref{PB1bis}--\eqref{PB3bis} since the last term can be understood as a boundary energy, 
closely related to the notion of weak dissipativity introduced in Definition \ref{propweakdissip}.

\begin{proposition}\label{propNRJcons}
Assume that $u=(\zeta,v^{\rm T})^{\rm T}$ and $\psi_{\rm i}$ form a regular solution to \eqref{PB1bis}--\eqref{PB3bis} for some $T>0$. 
Then for any $t\in [0,T]$, we have 
\[
E_{\rm fluid}(t)=E_{\rm fluid}(0).
\]
\end{proposition}

\begin{proof}
We can check that the modified nonlinear shallow water equations \eqref{PB1bis} satisfy the same conservation of the local energy 
as the standard nonlinear shallow water equations, that is, 
\[
\dt {\mathfrak e} + \nabla\cdot {\mathfrak F}=0 \quad\mbox{in}\quad (0,T)\times\cE
\]
with ${\mathfrak e}=\frac{1}{2}\gr\zeta^2+\frac{1}{2}h\abs{v}^2$ and ${\mathfrak F}=\bigl( \gr\zeta+\frac{1}{2}\abs{v}^2 \bigr)hv$. 
We therefore have 
\begin{align*}
\frac{\rm d}{{\rm d}t}\left( \int_{\cE}{\mathfrak e} + \frac{1}{2}\int_\mathit{\Gamma}\psi_{\rm i}\Lambda \psi_{\rm i} \right)
&= -\int_{\cE}\nabla\cdot {\mathfrak F} + \int_\mathit{\Gamma}(\dt \psi_{\rm i})\Lambda \psi_{\rm i} \\
&=\int_\mathit{\Gamma}(\gr\zeta+\tfrac{1}{2}\abs{v}^2) (N\cdot(hv))-\int_\mathit{\Gamma}({\mathtt g}\zeta+\tfrac{1}{2}\abs{v}^2)\Lambda\psi_{\rm i},
\end{align*}
where we used \eqref{PB3bis} to derive the second identity. 
Using \eqref{PB2bis}, the right-hand side vanishes, which proves the result. 
\end{proof}

If the water height is bounded from below by a positive constant, 
Proposition \ref{propNRJcons} furnishes a control of the $L^2(\cE)$-norm of $\zeta$ and $v$, and also, with the help of Proposition \ref{propDN}, 
of the $H^{-1/2}(\mathit{\Gamma})$-norm of $\dtan\psi_{\rm i}$. 
In order to get some control on the derivatives of these quantities, it is necessary to study the linearization of \eqref{PB1bis}--\eqref{PB3bis}. 
Since this linearization process is not the standard one, we discuss it first on a simpler toy problem.

\subsection{Transformation of the linearized equations for a toy problem}\label{sectlintoy}
Let us consider a system of equations of the form 
\begin{equation}\label{toy1}
\partial_t u + \partial_j {\mathcal F}_j(u) = 0 \quad\mbox{in}\quad (0,T)\times\cE,
\end{equation}
where $u$ and the smooth mappings ${\mathcal F}_j$ ($j=1,2$) take their values in $\R^n$. 
Denoting ${\mathcal F}_{\rm nor}(u) = N_j{\mathcal F}_j(u)\in \R^n$, we consider nonlinear boundary conditions of the form 
\begin{equation}\label{toy2}
M{\mathcal F}_{\rm nor}(u) = 0 \quad\mbox{on}\quad (0,T)\times\mathit{\Gamma},
\end{equation}
where $M$ is a $p\times n$ matrix with constant entries. 
The system \eqref{toy1} can be put under quasilinear form 
\begin{equation}\label{toy1bis}
\partial_t u + A_j(u)\partial_j u = 0 \quad\mbox{in}\quad (0,T)\times\cE,
\end{equation}
where $A_j(u)={\rm d}_u{\mathcal F}_j(u)$. 
Therefore, a standard linearized system around the state $u$ is given by 
\begin{equation}\label{toy1bislin}
\partial_t \dot{u} +A_j(u)\partial_j\dot{u} = f \quad\mbox{in}\quad (0,T)\times\cE,
\end{equation}
with boundary conditions 
\begin{equation}\label{toy2lin}
MA_{\rm nor}\dot{u}=g \quad\mbox{on}\quad (0,T)\times\mathit{\Gamma},
\end{equation}
where $A_{\rm nor}=N_jA_j(u)$ and $f$ and $g$ are source terms typically accounting for the presence of commutators. 
If solutions to \eqref{toy1bislin} subject to the boundary condition \eqref{toy2lin} satisfy, for some Friedrichs symmetrizer $S(u)$, 
the weak dissipativity property introduced in Definition \ref{propweakdissip}, 
then the energy estimates in Proposition \ref{propNRJnear} hold and furnish some control on $\dot{u}$ 
and therefore on the time and tangential derivatives of the solution $u$ to the nonlinear problem.

Alternatively, we could apply $\chi_{\rm b}\dpar$ to the conservative form \eqref{toy1} of the equations, leading to a linear system of the form 
\begin{equation}\label{toy1linbis}
\partial_t \dot{u} + \partial_j (A_j(u)\dot{u}) = f \quad\mbox{in}\quad (0,T)\times\cE, 
\end{equation}
where $f$ is different from that in \eqref{toy1bislin}. 
This formulation can be written 
\[
\partial_t \dot{u} +A_j(u) \partial_j\dot{u} +B(u,\partial_xu)\dot{u} = f \quad\mbox{in}\quad (0,T)\times\cE, 
\]
where $B(u,\partial_xu) = \partial_j (A_j(u))$; the linearized boundary condition is still given by \eqref{toy2lin}. 
This problem falls into the category \eqref{eqlinfixed} considered in Section \ref{sectapriorilin}, 
but the form of the matrix $B$ induces a loss of one space derivative in the dependence on $u$ in the energy estimates. 
This is the reason why the linearization \eqref{toy1bislin} is generally preferred to \eqref{toy1linbis}. 
However, when the matrices $A_j$ satisfy the following assumption, the linearization \eqref{toy1linbis} would be convenient.

\begin{assumption}\label{assstructure}
Let $\Omega\subset \R^n$ an open set, which represents a phase space of $u$. 
Suppose that for all $u\in \Omega$, there exists a positive definite matrix $\Sigma(u)$ and two constant symmetric matrices $G_1$ and $G_2$ 
such that $A_j(u)=G_j \Sigma(u)$ holds for $j=1,2$. 
\end{assumption}

If the assumption is satisfied, then we can rewrite \eqref{toy1linbis} under the form 
\[
\partial_t\dot{u} + \partial_j (G_j \Sigma (u) \dot{u}) = f \quad\mbox{in}\quad (0,T)\times\cE
\]
with boundary conditions 
\[
MN_jG_j\Sigma(u)\dot{u} = g \quad\mbox{on}\quad (0,T)\times\mathit{\Gamma}.
\]
Introducing $\widecheck{u}=\Sigma(u)\dot{u}$, the system \eqref{toy1}--\eqref{toy2} is therefore equivalent to 
\begin{equation}\label{eqtu1}
\partial_t \widecheck{u} + \Sigma(u) G_j \partial_j \widecheck{u} + B(u,\dt u)\widecheck{u} = \Sigma(u)f \quad\mbox{in}\quad (0,T)\times\cE, 
\end{equation}
where $B(u,\dt u) = \Sigma(u)\dt (\Sigma(u)^{-1})$, and linear boundary conditions 
\begin{equation}\label{eqtu2}
MG_{\rm nor} \widecheck{u} = g \quad\mbox{on}\quad (0,T)\times\mathit{\Gamma},
\end{equation}
where $G_{\rm nor}=N_jG_j$. 
The system \eqref{eqtu1} is obviously symmetrizable in the sense that Assumption \ref{assFriedrichs} is satisfied by choosing 
$S(u)=\Sigma(u)^{-1}$ as symmetrizer; 
in order to use Proposition \ref{propNRJnear} to derive an energy estimate on $\widecheck{u}$, we just have to check the necessary assumptions. 
In particular, we must check whether the boundary term ${\mathfrak B}[\widecheck{u}]=G_{\rm nor}\widecheck{u}\cdot\widecheck{u}$ 
associated with this symmetrizer is weakly dissipative in the sense of Definition \ref{propweakdissipm} for solutions 
$\widecheck{u}$ to \eqref{eqtu1} that satisfy the boundary conditions \eqref{eqtu2}.

In the next section, we adapt this strategy to \eqref{PB1bis}--\eqref{PB3bis} whose boundary condition is more complicated than \eqref{toy2lin}; 
the fact that such a weak dissipativity property is satisfied will then be addressed in Section \ref{sectWSL2}.

\subsection{Transformation of the linearized equations for the wave-structure interaction problem}\label{sectlinWS}
The nonlinear shallow water equations \eqref{PB1bis} are written in the conservative form \eqref{toy1}, that is, 
\begin{equation}\label{PB1F}
\dt u + \partial_j{\mathcal F}_j(u)=0 \quad\mbox{in}\quad (0,T)\times\cE,
\end{equation}
where ${\mathcal F}_j(u)^{\rm T}=\big(hv_j, (\gr\zeta+\frac{1}{2}\abs{v}^2 ) {\bf e}_j^{\rm T}\big)$ for $j=1,2$. 
The boundary conditions \eqref{PB2bis}--\eqref{PB3bis} are not exactly under the form \eqref{toy2} 
but can be seen as a nonlocal, in time and space, generalization of it. 
They can indeed be written as 
\begin{equation}\label{PB2F}
\begin{cases}
 {\mathcal F}_{\rm nor}(u)^{\rm I}-\Lambda \psi_{\rm i} = 0 &\mbox{on}\quad (0,T)\times\mathit{\Gamma}, \\
 \partial_t \psi_{\rm i}+N\cdot{\mathcal F}_{\rm nor}(u)^{\rm II} = 0 &\mbox{on}\quad (0,T)\times\mathit{\Gamma}.
\end{cases}
\end{equation}
Here and in what follows, we use the notation $u^{\rm I}=u_1$ and $u^{\rm II}=(u_2,u_3)^{\rm T}$ for $u=(u_1,u_2,u_3)^{\rm T}\in\R^3$. 
Applying $\chi_{\rm b}\dpar$ to these equations with $\dpar=\dt$ or $\dtan$ and $\chi_{\rm b}$ as in \eqref{defchib}, 
we find that $\dot{u}=\chi_{\rm b}\dpar u$ and $\dot{\psi}_{\rm i}=\dpar \psi_{\rm i}$ solve 
\begin{equation}\label{pretransf1}
\dt\dot{u} +\partial_j (A_j(u)\dot{u}) = \dot{f} \quad\mbox{in}\quad (0,T)\times\cE
\end{equation}
with boundary conditions 
\begin{equation}\label{pretransf2}
\begin{cases}
 (A_{\rm nor}\dot{u})^{\rm I} - \Lambda\dot{\psi}_{\rm i} = g_1 &\mbox{on}\quad (0,T)\times\mathit{\Gamma}, \\
 \dt \dot{\psi}_{\rm i} + N\cdot(A_{\rm nor}\dot{u})^{\rm II} = g_2 &\mbox{on}\quad (0,T)\times\mathit{\Gamma}, 
\end{cases}
\end{equation}
where $A_j(u)={\rm d}_u {\mathcal F}_j(u)$, $A_{\rm nor}=N_jA_j(u)$, and 
\begin{equation}\label{deffg1g2}
\dot{f}=-[\chi_{\rm b}\dpar,\partial_j]{\mathcal F}_j(u), \quad
g_1=[\dpar,\Lambda]\psi_{\rm i}, \quad\mbox{and}\quad
g_2=-(\dpar N)\cdot{\mathcal F}_{\rm nor}(u)^{\rm II}.
\end{equation}
In the following argument, we will always assume that the flow is subcritical so that we impose the following assumption on 
$u=(\zeta,v^{\rm T})^{\rm T}$.

\begin{assumption}\label{ass:subcritical}
There exists a positive constant $c_0$ such that for any $(t,x)\in(0,T)\times\cE$ we have 
\[
\gr h(t,x) - \abs{v(t,x)}^2 \geq c_0,
\]
where $h=H_0+\zeta$ with a positive constant $H_0$. 
\end{assumption}
The following proposition provides a convenient reformulation of the problem \eqref{pretransf1}--\eqref{pretransf2}.

\begin{proposition}\label{proplintransf}
Let $T>0$ and suppose that Assumptions \ref{ass:hi} and \ref{ass:subcritical} are satisfied. 
Let us also define the matrices $G_j$ ($j=1,2$) and $\Sigma(u)$ as 
\begin{equation}\label{defGj}
G_j=
\begin{pmatrix}
 0 & {\bf e}_j^{\rm T} \\
 {\bf e}_j & 0_{2\times 2}
\end{pmatrix}
 \quad\mbox{and}\quad
\Sigma(u)=
\begin{pmatrix}
 \gr & v^{\rm T}\\
 v & h \mbox{\rm Id}_{2\times 2}
\end{pmatrix}.
\end{equation}
Then, $\Sigma(u)$ is invertible on $(0,T)\times \cE$ and that $\dot{u}$ is a regular solution of \eqref{pretransf1}--\eqref{pretransf2} 
if and only $\widecheck{u}:=\Sigma(u)\dot{u}$ and $\widecheck{\psi}_{\rm i}:=\dot{\psi}_{\rm i}$ solve 
\begin{equation}\label{PB1bislin0}
\partial_t\widecheck{u} + \Sigma(u)G_j \partial_j \widecheck{u} + B(u,\dt u)\widecheck{u} = \Sigma(u)\dot{f} \quad\mbox{in}\quad (0,T)\times\cE,
\end{equation}
where  $B(u,\dt u)=\Sigma(u)\dt (\Sigma(u)^{-1})$, and boundary conditions 
\begin{equation}\label{PB2bislin0}
\begin{cases}
 N\cdot\widecheck{u}^{\rm II}-\Lambda \widecheck{\psi}_{\rm i} = g_1 &\mbox{on}\quad (0,T)\times\mathit{\Gamma}, \\
 \partial_t \widecheck{\psi}_{\rm i}+\widecheck{u}^{\rm I}=g_2 &\mbox{on}\quad (0,T)\times\mathit{\Gamma}. 
\end{cases}
\end{equation}
\end{proposition}

\begin{proof}
The assumption made on $u$ ensures that $\Sigma(u)$ is invertible. 
The result follows from a straightforward adaptation of the procedure described in the previous section after remarking that 
$A_j(u)={\rm d}_u{\mathcal F}_j(u)$ $(j=1,2)$ satisfy Assumption \ref{assstructure}, that is,
\begin{equation}\label{defAj}
A_j(u)=
\begin{pmatrix}
 v_j & h{\bf e}_j^{\rm T} \\
 \gr {\bf e}_j & {\bf e_j}\otimes v
\end{pmatrix}
=G_j \Sigma(u).
\end{equation}
\end{proof}

\subsection{$L^2$-estimates for the transformed linearized equations}\label{sectWSL2}
We are interested here in the analysis of the transformed linearized equations \eqref{PB1bislin0} under the boundary conditions \eqref{PB2bislin0}, 
as provided by Proposition \ref{proplintransf}. 
Since the lower order term $B(u, \partial_tu)\widecheck{u}$ in the transformed equations does not contribute to the main structure of the system in 
our analysis, we absorb the term in the right-hand side and consider in this and the following subsections the equations 
\begin{equation}\label{PB1bislin}
\dt \widecheck{u} +\Sigma(u)G_j \partial_j\widecheck{u} = \Sigma(u)\dot{f} \quad\mbox{in}\quad (0,T)\times\cE,
\end{equation}
under the boundary conditions 
\begin{equation}\label{PB2bislin}
\begin{cases}
 N\cdot\widecheck{u}^{\rm II}-\Lambda \widecheck{\psi}_{\rm i} = g_1 &\mbox{on}\quad (0,T)\times\mathit{\Gamma}, \\
 \dt\psi_{\rm i} + \widecheck{u}^{\rm I} = g_2 &\mbox{on}\quad (0,T)\times\mathit{\Gamma}. 
\end{cases}
\end{equation}
If $u\in L^\infty((0,T)\times\cE)$ satisfies Assumption \ref{ass:subcritical}, the matrix 
\begin{equation}\label{defS}
S(u) = \Sigma(u)^{-1} = \frac{1}{\gr h -\abs{v}^2}
\left(
\begin{array}{cc}
 h & -v^{\rm T} \\
 -v & \gr\mbox{Id}_{2\times 2}-\frac{1}{h}v^\perp\otimes v^\perp
\end{array}\right)
\end{equation}
furnishes a Friedrichs symmetrizer in the sense of Assumption \ref{assFriedrichs} for \eqref{PB1bislin}. 
The corresponding boundary term \eqref{defB} is given by 
\begin{equation}\label{defBtransf}
{\mathfrak B}[\widecheck{u}] = G_{\rm nor}\widecheck{u}\cdot\widecheck{u} = 2\widecheck{u}^{\rm I}( N\cdot\widecheck{u}^{\rm II}).
\end{equation}
We could obtain an $L^2$ energy estimate for $\widecheck{u}$ from Proposition \ref{propNRJnear} 
if we could control this boundary term by proving that solutions satisfying the boundary conditions \eqref{PB2bislin} 
are weakly dissipative in the sense of Definition \ref{propweakdissip} and for the choice \eqref{defS} of symmetrizer. 
This is done in the following proposition. 
We recall that the functional spaces we use are defined in Section \ref{sectfunctionspace}.

\begin{proposition}\label{propWSweakdiss}
Let $u,\dt u\in L^\infty((0,T)\times\cE)$ and suppose that Assumptions \ref{ass:hi} and \ref{ass:subcritical} are satisfied. 
Let also $0<\underline{\mathfrak c}<\underline{\mathfrak C}$ be the constants in Proposition \ref{propDN} and $\tilde{g}_1$ an extension of $g_1$ 
to $(0,T)\times\cE$ so that $\tilde{g}_{1\vert_\mathit{\Gamma}}=g_1$. 
Then, for any regular solution $(\widecheck{u},\widecheck{\psi}_{\rm i})$ of \eqref{PB1bislin}--\eqref{PB2bislin}, 
the boundary term \eqref{defBtransf} is weakly dissipative in the sense of Definition \ref{propweakdissip}, that is, 
there exists $\lambda_0=\lambda_0\bigl(\frac{1}{c_0},\Abs{(u,\dt u)}_{L^\infty((0,T)\times\cE)}\bigr)$ such that for any $\lambda\geq\lambda_0$ 
and any $t\in[0,T]$ we have 
\[
2 \int_0^t e^{-2\lambda t'} \left( \int_\mathit{\Gamma} \widecheck{u}^{\rm I}( N\cdot\widecheck{u}^{\rm II} ) \right){\rm d}t'
\leq -E_{\rm bdry}(t)+S_{\rm data}(t)+\frac{\alpha_0}{8} I_{\lambda,t}(\Abs{\widecheck{u}(\cdot)}_{L^2})^2
\]
with 
\begin{align*}
E_{\rm bdry}(t) 
&= \underline{\mathfrak c}\bigl( e^{-2\lambda t}\abs{\dtan\widecheck{\psi}_{\rm i}(t)}_{H^{-1/2}(\mathit{\Gamma})}^2
 + \lambda\abs{\dtan\widecheck{\psi}_{\rm i}}^2_{L^2_{\lambda,t}H^{-1/2}(\mathit{\Gamma})} \bigr), \\
S_{\rm data}(t) 
&= C_0\bigl( \abs{\dtan\widecheck{\psi}_{\rm i}(0)}_{H^{-1/2}(\mathit{\Gamma})}^2 + \Abs{ \tilde{g}_1(0) }_{L^2(\cE)}^2 \\
&\qquad
 + S_{\lambda,t}^*(\Abs{ \dot{f}(\cdot)}_{L^2(\cE)})^2 + S_{\lambda,t}^*(\opnorm{\tilde{g}_1(\cdot)}_{1})^2
 + \lambda^{-1}\abs{\dtan g_2}_{L^2_{\lambda,t}H^{-1/2}(\mathit{\Gamma})}^2 \bigr),
\end{align*}
where $C_0=C_0\bigl(\frac{1}{c_0},\frac{1}{\underline{\mathfrak{c}}},\underline{\mathfrak{C}},\Abs{u}_{L^\infty((0,T)\times\cE)}\bigr)$.
\end{proposition}

\begin{proof}
Throughout the proof, we use the fact that under the choice of the symmetrizer $S(u)$ given in \eqref{defS}, 
the positive constants $\alpha_0$ and $\beta_0$ in Assumption \ref{assFriedrichs} 
can be taken so that $\tfrac{1}{\alpha_0},\beta_0 \leq C\bigl(\tfrac{1}{c_0},\Abs{u}_{L^\infty((0,T)\times\cE)}\bigr)$. 
We also use the notation 
\[
E_{\rm int}(\psi_{\rm i}) = \frac{1}{2}\int_\mathit{\Gamma} \psi_{\rm i}\Lambda\psi_{\rm i}.
\]
Recalling that $\widecheck{u}^{\rm I}=-\dt\widecheck{\psi}_{\rm i}+g_2$ and that $N\cdot\widecheck{u}^{\rm II}=\Lambda\widecheck{\psi}_{\rm i}+g_1$, 
we have 
\begin{align*}
2\int_0^t e^{-2\lambda t'}\left(\int_\mathit{\Gamma} \widecheck{u}^{\rm I}( N\cdot\widecheck{u}^{\rm II} )\right){\rm d}t'
&= -2\int_0^t e^{-2\lambda t'} ( \dt\widecheck{\psi}_{\rm i},\Lambda \widecheck{\psi}_{\rm i} )_{L^2(\mathit{\Gamma})}{\rm d}t' \\
&\quad\;
 + 2\int_0^t e^{-2\lambda t'} ( \Lambda\widecheck{\psi}_{\rm i},g_2 )_{L^2(\mathit{\Gamma})}{\rm d}t'
 + 2\int_0^t e^{-2\lambda t'} ( g_1, \widecheck{u}^{\rm I} )_{L^2(\mathit{\Gamma})}{\rm d}t' \\
&=: B_1+B_2+B_3.
\end{align*}
Let us evaluate separately each integrals $B_1$, $B_2$, and $B_3$.

-- Analysis of $B_1$. 
From the symmetry of the Dirichlet-to-Neumann map $\Lambda$, we get directly 
\begin{align*}
B_1
&= -2\int_0^t e^{-2\lambda t'} \frac{\rm d}{{\rm d}t'}E_{\rm int}(\widecheck{\psi}_{\rm i}(t')){\rm d}t' \\
&= -2\left( e^{-2\lambda t} E_{\rm int}(\widecheck{\psi}_{\rm i}(t))
 + 2\lambda \int_0^t e^{-2\lambda t'}E_{\rm int}(\widecheck{\psi}_{\rm i}(t')){\rm d}t' \right) + 2E_{\rm int}(\widecheck{\psi}_{\rm i}(0)). 
\end{align*}
Here, by Proposition \ref{propDN} we have 
\begin{equation}\label{DNequi}
\underline{\mathfrak c}\abs{\dtan \psi_{\rm i}}_{H^{-1/2}(\mathit{\Gamma})}^2
 \leq 2E_{\rm int}(\psi_{\rm i})
 \leq \underline{\mathfrak C}\abs{\dtan \psi_{\rm i}}_{H^{-1/2}(\mathit{\Gamma})}^2.
\end{equation}
Therefore, we obtain 
\begin{equation}\label{estB1}
B_1 \leq -\underline{\mathfrak c}\bigl( e^{-2\lambda t} \abs{\dtan \widecheck{\psi}_{\rm i}(t)}_{H^{-1/2}(\mathit{\Gamma})}^2
 + 2\lambda\abs{\dtan\widecheck{\psi}_{\rm i}}^2_{L^2_{\lambda,t}H^{-1/2}(\mathit{\Gamma})} \bigr)
 + \underline{\mathfrak C}\abs{\dtan \widecheck{\psi}_{\rm i}(0)}_{H^{-1/2}(\mathit{\Gamma})}^2.
\end{equation}

-- Analysis of $B_2$. 
Since the Dirichlet-to-Neumann map $\Lambda$ is symmetric and positive in $L^2(\mathit{\Gamma})$, 
we can apply the Cauchy--Schwarz inequality and \eqref{DNequi} to obtain 
\begin{align*}
\abs{B_2}
&\leq 2 \int_0^t e^{-2\lambda t'}E_{\rm int}(\widecheck{\psi}_{\rm i}(t'))^{1/2}E_{\rm int}(g_2(t'))^{1/2}{\rm d}t' \\
&\leq 2\underline{\mathfrak C}\abs{\dtan\widecheck{\psi}_{\rm i}}_{L^2_{\lambda,t}H^{-1/2}(\mathit{\Gamma})}
 \abs{\dtan g_2}_{L^2_{\lambda,t}H^{-1/2}(\mathit{\Gamma})}, 
\end{align*}
so that, using Young's inequality we get 
\begin{equation}\label{estB2}
\abs{B_2} \leq \underline{\mathfrak c}\lambda\abs{\dtan\widecheck{\psi}_{\rm i}}^2_{L^2_{\lambda,t}H^{-1/2}(\mathit{\Gamma})}
 + (\underline{\mathfrak c}\lambda)^{-1}\underline{\mathfrak C}^2\abs{\dtan g_2}^2_{L^2_{\lambda,t}H^{-1/2}(\mathit{\Gamma})}. 
\end{equation}

-- Analysis of $B_3$. 
We note the identity $\widecheck{u}^{\rm I}=(0,N^{\rm T})^{\rm T}\cdot G_{\rm nor}\widecheck{u}$, so that we have 
\[
g_1\widecheck{u}^{\rm I}=\varphi\cdot G_{\rm nor}\widecheck{u}=N_j(\varphi\cdot G_j\widecheck{u})
\]
with $\varphi=(0,\chi_{\rm b}\tilde{g}_1N^{\rm T})^{\rm T}$. 
We note also that the transformed linearized equations \eqref{PB1bislin} can be written as 
\[
\dt(S(u)\widecheck{u}) + G_j\partial_j\widecheck{u} = \widecheck{f} \quad\mbox{in}\quad (0,T)\times\cE,
\]
where $\widecheck{f}=\dot{f}-(\dt S(u))\widecheck{u}$. 
Therefore, we see that 
\begin{align*}
(g_1,\widecheck{u}^{\rm I})_{L^2(\mathit{\Gamma})}
&= -\int_{\cE}\partial_j(\varphi\cdot G_j\widecheck{u}) \\
&= \frac{\rm d}{{\rm d}t}\int_\cE S(u)\varphi\cdot\widecheck{u}
 - \int_\cE \{ (\dt\varphi+G_j\partial_j\varphi)\cdot\widecheck{u}+\varphi\cdot \widecheck{f}\},
\end{align*}
so that 
\begin{align*}
B_3
&= 2e^{-2\lambda t}(S(u(t))\varphi(t),\widecheck{u}(t))_{L^2(\cE)}
 - 2(S(u(0))\varphi(0),\widecheck{u}(0))_{L^2(\cE)} \\
&\quad\;
 +4\lambda\int_0^t e^{\-2\lambda t'}(S(u)\varphi,\widecheck{u})_{L^2(\cE)} {\rm d}t' \\
&\quad\;
 -2 \int_0^t e^{-2\lambda t'}\{ (S(u)\dt\varphi+G_j\partial_j\varphi,\widecheck{u})_{L^2(\cE)} + (\varphi, \widecheck{f})_{L^2(\cE)} \}{\rm d}t'. 
\end{align*}
This implies that 
\begin{align*}
\abs{B_3}
&\leq 4\beta_0I_{\lambda,t}( \Abs{ \varphi(\cdot) }_{L^2}) I_{\lambda,t}(\Abs{ \widecheck{u}(\cdot) }_{L^2}) \\
&\quad\;
 + 2(\beta_0+1)I_{\lambda,t}(\Abs{ \widecheck{u}(\cdot) }_{L^2}) S_{\lambda,t}^*(\opnorm{ \varphi(\cdot) }_1)
 + 2I_{\lambda,t}(\Abs{ \varphi(\cdot) }_{L^2}) S_{\lambda,t}^*(\Abs{ \widecheck{f}(\cdot) }_{L^2}).
\end{align*}
Here, it follows from \eqref{eqdtS} that 
$I_{\lambda,t}( \Abs{ \varphi(\cdot) }_{L^2}) \lesssim \Abs{ \varphi(0) }_{L^2} + S_{\lambda,t}^*( \opnorm{ \varphi(\cdot) }_1 )$, 
so that by the definition of $\varphi$ we easily have $\Abs{ \varphi(0) }_{L^2} \lesssim \Abs{ \tilde{g}_1(0) }_{L^2}$ and 
$\opnorm{\varphi(t)}_1 \lesssim \opnorm{ \tilde{g}_1(t) }_1$. 
Moreover, we have $\Abs{ \widecheck{f}(t) }_{L^2} \leq \Abs{ \dot{f}(t) }_{L^2} + \beta_1\Abs{ \widecheck{u}(t) }_{L^2}$ 
with $\beta_1=\Abs{ \dt S(u) }_{L^\infty((0,T)\times\cE)}$, so that 
\begin{equation}\label{estB3}
\abs{B_3} \leq \left(\frac{\alpha_0}{16}+\frac{\beta_1^2}{\lambda^2}\right) I_{\lambda,t}(\Abs{ \widecheck{u}(\cdot) }_{L^2})^2
 +C_0\bigl( \Abs{ \tilde{g}_1(0) }_{L^2}^2 + S_{\lambda,t}^*(\opnorm{\tilde{g}_1(\cdot)}_1)^2 + S_{\lambda,t}^*(\Abs{ \dot{f}(\cdot) }_{L^2})^2 \bigr),
\end{equation}
where $C_0=C_0\bigl(\frac{1}{\alpha_0},\beta_0\bigr)$. 
Here, we note that $\beta_1 \leq C\bigl(\frac{1}{c_0},\Abs{ (u,\dt u) }_{L^\infty((0,T)\times\cE)}\bigr)$.

Gathering the estimates provided by \eqref{estB1}, \eqref{estB2}, and \eqref{estB3}, and taking $\lambda_0>0$ so large that 
$\beta_1^2 \leq \frac{\alpha_0}{16}\lambda_0^2$ holds, we get the result stated in the proposition. 
\end{proof}

From Proposition \ref{propNRJnear}, we deduce directly the following corollary that provides an energy estimate for the solution 
of the linearized equations \eqref{PB1bislin}--\eqref{PB2bislin}.

\begin{corollary}\label{coroestL2}
Under the assumptions of Proposition \ref{propWSweakdiss}, there exists a positive constant 
$\lambda_0=\lambda_0\bigl(\frac{1}{c_0},\Abs{ (u,\dt u) }_{L^\infty((0,T)\times\cE)}\bigr)$ such that 
for any regular solution $(\widecheck{u},\widecheck{\psi}_{\rm i})$ of \eqref{PB1bislin}--\eqref{PB2bislin} we have the energy estimate 
\begin{align*}
& I_{\lambda,t}(\Abs{\widecheck{u}(\cdot)}_{L^2})^2 + I_{\lambda,t}(\abs{\dtan\widecheck{\psi}_{\rm i}(\cdot)}_{H^{-1/2}})^2 \\
&\leq C_0\bigl( \Abs{\widecheck{u}(0)}_{L^2}^2 + \abs{\dtan\widecheck{\psi}_{\rm i}(0)}_{H^{-1/2}}^2 + \Abs{\tilde{g}_1(0)}_{L^2}^2 \\
&\qquad
 + S_{\lambda,t}^*(\Abs{ \dot{f}(\cdot)}_{L^2})^2
 + S_{\lambda,t}^*(\opnorm{\tilde{g}_1(\cdot)}_{1})^2 + \lambda^{-1}\abs{\dtan g_2}_{L^2_{\lambda,t}H^{-1/2}}^2 \bigr).
\end{align*}
for any $\lambda\geq\lambda_0$ and any $t\in[0,T]$, 
where $C_0=C_0\bigl(\frac{1}{c_0},\frac{1}{\underline{\mathfrak{c}}},\underline{\mathfrak{C}},\Abs{u}_{L^\infty((0,T)\times\cE)}\bigr)$. 
\end{corollary}

\begin{proof}
We note that the transformed linearized equations \eqref{PB1bislin} have the form \eqref{eqlinfixed} with $A_j=\Sigma(u)G_j$ for $j=1,2$ and 
$B=0_{3\times3}$, so that in view of the choice of the symmetrizer $S(u)$ given by \eqref{defS}, we have $\dt S+\partial_j(SA_j)-2SB=\dt S$. 
Therefore, the constant $\beta_1$ in Assumption \ref{assFriedrichs} can be taken so that 
$\beta_1\leq \Abs{ \dt S(u) }_{L^\infty((0,T)\times\cE)} \leq C\bigl(\frac{1}{c_0},\Abs{(u,\dt u)}_{L^\infty((0,T)\times\cE)}\bigr)$. 
From Propositions \ref{propNRJnear} and \ref{propWSweakdiss}, we get the estimate with $I_{\lambda,t}(\abs{\dtan\widecheck{\psi}_{\rm i}(\cdot)}_{H^{-1/2}})^2$ 
in the left-hand side replaced by $\widetilde{I}_{\lambda,t}(\abs{\dtan\widecheck{\psi}_{\rm i}(\cdot)}_{H^{-1/2}})^2$, 
where $\widetilde{I}_{\lambda,t}(\cdot)$ was defined in \eqref{defItilde}. 
Proceeding as in the proof of Proposition \ref{propNRJL2far}, we obtain the desired estimate. 
\end{proof}

\subsection{Higher order estimates for the transformed linearized equations}\label{sectHOest}
We showed in the previous section that solutions to \eqref{PB1bislin}--\eqref{PB2bislin} satisfy the weak dissipativity property of 
Definition \ref{propweakdissip}, and deduced in Corollary \ref{coroestL2} an $L^2$-based a priori energy estimate for such solutions. 
Following the general approach developed in Section \ref{sectapriorilin}, the derivation of higher order energy estimates requires that 
solutions to \eqref{PB1bislin}--\eqref{PB2bislin} satisfy a higher order dissipativity property, in the sense of Definition \ref{propweakdissipm}. 
Such a property is established in the following proposition. 
We refer to \eqref{defHm0}, \eqref{defHm}, and \eqref{defHm2} for the definition of $L_{\lambda,t}^pH^{s}(\mathit{\Gamma})$, 
$L_{\lambda,t}^pH^{s}_{(m)}(\mathit{\Gamma})$, and $\abs{g(t)}_{H^{s}_{(m)}}$. 
We also recall that $\opnorm{\widecheck{u}(t)}_{m,\parallel}$ was defined in \eqref{defopnormpar} 
and that the constants $\underline{\mathfrak c}$ and $\underline{\mathfrak C}$ were introduced in Proposition \ref{propDN}.

\begin{proposition}\label{propWSweakdissm}
Let $m$ be a non-negative integer and assume that Assumptions \ref{ass:hi} and \ref{ass:subcritical} are satisfied. 
Assume also that $u\in\mathbb{W}_T^m \cap \mathbb{W}_T^{2,p}$ for some $p\in(2,\infty)$ and take two constants $0<K_0\leq K$ such that 
\[
\begin{cases}
 \frac{1}{c_0},\frac{1}{\underline{\mathfrak{c}}},\underline{\mathfrak{C}},\Abs{u}_{L^\infty((0,T)\times\cE)} \leq K_0, \\
 K_0, \Abs{u}_{\mathbb{W}_T^m \cap \mathbb{W}_T^{2,p}} \leq K.
\end{cases}
\]
Then, for any regular solution $(\widecheck{u},\widecheck{\psi}_{\rm i})$ of \eqref{PB1bislin}--\eqref{PB2bislin} supported in $\overline{\cE}\cap U_\mathit{\Gamma}$, 
the boundary term \eqref{defBtransf} is weakly dissipative of order $m$ in the sense of Definition \ref{propweakdissipm}, that is, 
for any $\alpha=(\alpha_0,\alpha_1)\in\N^2$ satisfying $\abs{\alpha} \leq m$ and any $\lambda\geq\lambda_0(K)$ it holds that 
\begin{align*}
& 2\int_0^t e^{-2\lambda t'} \left( \int_\mathit{\Gamma} (\dpar^\alpha\widecheck{u})^{\rm I}( N\cdot(\dpar^\alpha\widecheck{u})^{\rm II} ) \right){\rm d}t' \\
&\leq -E_{{\rm bdry},\alpha}(t)+S_{{\rm data},\alpha}(t)+\frac{\alpha_0}{8} I_{\lambda,t}(\Abs{\dpar^\alpha\widecheck{u}(\cdot)}_{L^2})^2
 + C(K)\lambda^{-2} I_{\lambda,t}(\opnorm{\widecheck{u}(\cdot)}_m)^2
\end{align*}
with 
\begin{align*}
E_{{\rm bdry},\alpha}(t) 
&= \tfrac12\underline{\mathfrak c}\bigl( e^{-2\lambda t}\abs{\dpar^\alpha\dtan\widecheck{\psi}_{\rm i}(t)}_{H^{-1/2}}^2
 + \lambda\abs{\dpar^\alpha\dtan\widecheck{\psi}_{\rm i}}^2_{L^2_{\lambda,t}H^{-1/2}} \bigr), \\
S_{{\rm data},\alpha}(t) 
&= C(K_0)\bigl( \Abs{ [\dpar^\alpha,N]\cdot\widecheck{u}^{\rm II}(0) }_{L^2}^2
 + \abs{\dpar^\alpha\dtan\widecheck{\psi}_{\rm i}(0)}_{H^{-1/2}}^2 + \abs{\dpar^\alpha g_1(0)}_{H^{-1/2}}^2 \\
&\qquad
 + S_{\lambda,t}^*(\Abs{\dpar^\alpha \dot{f}(\cdot)}_{L^2})^2 + S_{\lambda,t}^*(\abs{\dpar^\alpha g_1(\cdot)}_{H_{(1)}^{-1/2}})^2
 + \lambda^{-1}\abs{\dpar^\alpha\dtan g_2}_{L^2_{\lambda,t}H^{-1/2}}^2 \bigr).
\end{align*}
\end{proposition}

\begin{proof}
Applying $\dpar^\alpha$ with $\abs{\alpha}\leq m$ to the transformed linearized problem \eqref{PB1bislin}--\eqref{PB2bislin} 
we see that $\widecheck{u}^{(\alpha)}:=\dpar^\alpha\widecheck{u}$ and $\widecheck{\psi}_{\rm i}^{(\alpha)}:=\dpar^\alpha\widecheck{\psi}_{\rm i}$ solves 
\begin{equation}\label{PB1bislinalpha}
\partial_t \widecheck{u}^{(\alpha)} + \Sigma(u)G_j \partial_j \widecheck{u}^{(\alpha)} = \Sigma(u) \dot{f}^{(\alpha)}
 \quad\mbox{in}\quad (0,T)\times(\cE\cap U_\mathit{\Gamma})
\end{equation}
with boundary conditions 
\begin{equation}\label{PB2bislinalpha}
\begin{cases}
 N\cdot(\widecheck{u}^{(\alpha)})^{\rm II} - \Lambda \widecheck{\psi}^{(\alpha)}_{\rm i} = g^{(\alpha)}_1 &\mbox{on}\quad (0,T)\times\mathit{\Gamma}, \\
 \partial_t\widecheck{\psi}_{\rm i}^{(\alpha)} + (\widecheck{u}^{(\alpha)})^{\rm I} = \dpar^\alpha g_2 &\mbox{on}\quad (0,T)\times\mathit{\Gamma}, 
\end{cases}
\end{equation}
where 
\begin{align*}
\dot{f}^{(\alpha)}
&= \dpar^\alpha \dot{f} - [\dpar^\alpha,S(u)]\dt\widecheck{u}, \\
g_1^{(\alpha)}
&= (\dpar^\alpha g_1 + [\dpar^\alpha,\Lambda]\psi_{\rm i}) + (-[\dpar^\alpha, N ]\cdot\widecheck{u}^{\rm II}) \\
&=: g^{(\alpha)}_{1,1}+g^{(\alpha)}_{1,2}. 
\end{align*}
The problem \eqref{PB1bislinalpha}--\eqref{PB2bislinalpha} has the same structure as \eqref{PB1bislin}--\eqref{PB2bislin}, and 
we can therefore apply Proposition \ref{propWSweakdiss} with $\widecheck{u}$ replaced by 
$\widecheck{u}^{(\alpha)}$, $\dot{f}$ by $\dot{f}^{(\alpha)}$, etc. 
The result follows therefore from the proposition and the following estimates.

-- Estimate of $\opnorm{\tilde{g}_1^{(\alpha)}(t)}_1$, where $\tilde{g}_1^{(\alpha)}$ is an extension of 
$g_1^{(\alpha)}$ to $(0,T)\times\cE$ so that $\tilde{g}_1^{(\alpha)}$$_{\vert_\mathit{\Gamma}} = g_1^{(\alpha)}$. 
This extension is taken as $\tilde{g}_1^{(\alpha)}=\tilde{g}_{1,1}^{(\alpha)}+\tilde{g}_{1,2}^{(\alpha)}$ with 
$\tilde{g}_{1,1}^{(\alpha)}=(g_{1,1}^{(\alpha)})^{\rm ext}$ and $\tilde{g}_{1,2}^{(\alpha)}=-[\dpar^\alpha,N]\cdot\widecheck{u}^{\rm II}$, 
where $(\cdot)^{\rm ext}$ is the extension constructed in Proposition \ref{propextext}. 
Therefore, it follows from Proposition \ref{propDN} that 
\begin{align*}
\opnorm{ \tilde{g}_{1,1}^{(\alpha)} }_1
&\lesssim \abs{ g_{1,1}^{(\alpha)} }_{H^{1/2}} + \abs{ \dt g_{1,1}^{(\alpha)} }_{H^{-1/2}} \\
&\leq \abs{ \dpar^\alpha g_1 }_{H^{1/2}} + \abs{ \dt\dpar^\alpha g_1 }_{H^{-1/2}}
 + \abs{ [\dpar^\alpha,\Lambda]\widecheck{\psi}_{\rm i} }_{H^{1/2}} + \abs{ [\dpar^\alpha,\Lambda]\dt\widecheck{\psi}_{\rm i} }_{H^{-1/2}} \\
&\lesssim \abs{ \dpar^\alpha g_1 }_{H_{(1)}^{-1/2}} + 
\begin{cases}
 0 &\mbox{if}\quad \alpha_1=0, \\
 \abs{ \dtan\dt^{\alpha_0}\widecheck{\psi}_{\rm i} }_{H^{\alpha_1-1/2}}
 + \abs{ \dtan\dt^{\alpha_0+1}\widecheck{\psi}_{\rm i} }_{H^{\alpha_1-3/2}} &\mbox{if}\quad \alpha_1\geq1.
\end{cases}
\end{align*}
Here, we have easily 
$\abs{ \dtan\dt^{\alpha_0}\widecheck{\psi}_{\rm i} }_{H^{\alpha_1-1/2}} \lesssim \abs{ \dtan\dpar^\alpha\widecheck{\psi}_{\rm i} }_{H^{-1/2}}$. 
On the other hand, in the case $\alpha_1\geq1$ by using $\dt\widecheck{\psi}_{\rm i}=g_2-\widecheck{u}^{\rm I}$ we see that 
\begin{align*}
\abs{ \dtan\dt^{\alpha_0+1}\widecheck{\psi}_{\rm i} }_{H^{\alpha_1-3/2}}
&\leq \abs{ \dtan\dt^{\alpha_0}g_2 }_{H^{\alpha_1-3/2}} + \abs{ \dt^{\alpha_0}\widecheck{u} }_{H^{\alpha_1-1/2}} \\
&\lesssim \abs{\dtan\dpar^\alpha g_2 }_{H^{-3/2}} + \opnorm{ \widecheck{u} }_m.
\end{align*}
As for $\tilde{g}_{1,2}^{(\alpha)}$, we evaluate easily it as $\opnorm{ \tilde{g}_{1,2}^{(\alpha)} }_1 \lesssim \opnorm{ \widecheck{u} }_{m-1}$. 
Therefore, we obtain 
\begin{align*}
S_{\lambda,t}^*(\opnorm{ \tilde{g}_1^{(\alpha)}(\cdot) }_1 )^2
&\lesssim S_{\lambda,t}( \abs{ \dpar^\alpha g_1(\cdot) }_{H_{(1)}^{-1/2}} )^2
 + \lambda^{-1}\abs{ \dtan\dpar^\alpha g_2 }_{L_{\lambda,t}^2H^{-3/2}}^2 \\
&\quad\;
 + \lambda^{-1}\abs{ \dtan\dpar^\alpha \widecheck{\psi}_{\rm i} }_{L_{\lambda,t}^2H^{-1/2}}^2
 + \lambda^{-2}I_{\lambda,t}( \opnorm{ \widecheck{u}(\cdot) }_m )^2,
\end{align*}
where we used $S_{\lambda,t}(f)\leq\lambda^{-1/2}\abs{f}_{L_{\lambda,t}^2}\leq\lambda^{-1}I_{\lambda,t}(f)$. 
Here, we note that the third term in the right-hand side can be absorbed into $E_{{\rm bdry},\alpha}$ for sufficiently large $\lambda$.

-- Estimate of $\Abs{\tilde{g}_{1}^{(\alpha)}(0)}_{L^2}$. 
Using again Proposition \ref{propDN} to control the commutator term in $g_{1,1}^{(\alpha)}$, we get easily 
\[
\Abs{ \tilde{g}_{1}^{(\alpha)}(0) }_{L^2} \lesssim \Abs{ [\dpar^\alpha,N]\cdot\widecheck{u}^{\rm II}(0) }_{L^2}
 + \abs{\dpar^\alpha\dtan\widecheck{\psi}_{\rm i}(0)}_{H^{-3/2}} + \abs{\dpar^\alpha g_1(0)}_{H^{-1/2}}. 
\]

-- Estimate of $\Abs{\dot{f}^{(\alpha)}}_{L^2}$. 
If follows from Lemma \ref{lemmcommut} that the commutators can be evaluated as 
\[
\Abs{ [\dpar^\alpha,S(u)]\dt\widecheck{u} }_{L^2}
\lesssim
\begin{cases}
 \opnorm{ \boldsymbol{\partial} S(u) }_{1,p} \opnorm{ \widecheck{u} }_m &\mbox{for}\quad m=1,2, \\
 \opnorm{ \boldsymbol{\partial} S(u) }_{m-1} \opnorm{ \widecheck{u} }_m &\mbox{for}\quad m\geq3.
\end{cases}
\]
The following lemma gives a non-sharp variant of the classical Moser type inequality, which can be easily shown by using Sobolev embedding 
$H^1(\cE) \hookrightarrow L^q(\cE)$ for any $q\in[2,\infty)$.

\begin{lemma}\label{lem:Moser}
Let $\mathcal{U}$ be an open set in $\R^N$, $F\in C^\infty(\mathcal{U})$, and $m\in\N$. 
If $u\in \mathbb{W}_T^{m+1}$ takes its value in a compact set $\mathcal{K}\subset\mathcal{U}$, then for any $t\in[0,T]$ we have 
\[
\opnorm{ (\boldsymbol{\partial} F(u))(t) }_m 
\leq C(\mathcal{K}) \bigl( 1+\opnorm{ (\boldsymbol{\partial} u)(t) }_{\max\{m-2,1\}} \bigr)^m \opnorm{ (\boldsymbol{\partial} u)(t) }_m.
\]
\end{lemma}

Using this lemma, we obtain 
\[
\Abs{ \dot{f}^{(\alpha)} }_{L^2} \leq \Abs{ \dpar^\alpha \dot{f} }_{L^2} + C(K)\opnorm{ \widecheck{u} }_m.
\]

Finally, applying Proposition \ref{propWSweakdiss} to the problem \eqref{PB1bislinalpha}--\eqref{PB2bislinalpha} and using the above estimates, 
we get the result stated in the proposition. 
\end{proof}

In the general theory developed in Section \ref{sectapriorilin}, dissipativity of order $m$ of the boundary term allows us 
to derive higher order energy estimates, see Theorem \ref{theogenorderm}. 
However, this result relies on the non-characteristic property of Assumption \ref{assnonchar}, which is not satisfied by the system \eqref{PB1bislin}. 
Indeed, the associated boundary matrix is $\Sigma(u)G_{\rm nor}$, which is given by 
\[
\Sigma(u)G_{\rm nor} = 
\begin{pmatrix}
 N\cdot v & \gr N^{\rm T}\\
 hN & v\otimes N
\end{pmatrix}
\]
and has eigenvalues $\lambda_1=0$, $\lambda_2=N\cdot v+\sqrt{\gr h}$, and $\lambda_3=N\cdot v-\sqrt{\gr h}$. 
The matrix is therefore not invertible and Assumption \ref{assnonchar} is not satisfied. 
The key step in the proof of the following corollary, which provides a priori energy estimates of order $m$ on $\widecheck{u}$, 
is to prove that it satisfies the properties in Assumption \ref{asschar} so that one can apply Theorem \ref{theogenorderm2} 
in place of Theorem \ref{theogenorderm}.

\begin{corollary}\label{corhigher}
Under the assumptions of Proposition \ref{propWSweakdissm} and with the same notations, for any regular solution $(\widecheck{u},\widecheck{\psi}_{\rm i})$ 
of \eqref{PB1bislin}--\eqref{PB2bislin}, we have the energy estimate 
\[
I_{\lambda,t}( \opnorm{\widecheck{u}(\cdot)}_{m} )^2 + I_{\lambda,t}( \abs{\dtan\widecheck{\psi}_{\rm i}(\cdot)}_{H_{(m)}^{-1/2}} )^2
 \leq C(K_0) S^m_{\rm data}(t)
\]
with 
\begin{align*}
S^m_{\rm data}(t)&= \opnorm{\widecheck{u}(0)}_{m}^2 + \abs{\dtan\widecheck{\psi}_{\rm i}(0)}_{H_{(m)}^{-1/2}}^2 + \abs{g_1(0)}_{H^{-1/2}_{(m)}}^2 \\
&\quad\;
 + S_{\lambda,t}^*( \opnorm{ \dot{f}(\cdot) }_{m} )^2
 + S_{\lambda,t}^*( \abs{g_1(\cdot)}_{H^{-1/2}_{(m+1)}} )^2 + \lambda^{-1}\abs{\dtan g_2}_{L^2_{\lambda,t}H_{(m)}^{-1/2}}^2
\end{align*}
for any $\lambda\geq\lambda_0(K)$ and $t\in[0,T]$. 
\end{corollary}

\begin{proof}
It is sufficient to check the conditions in Assumptions \ref{asschar} and \ref{ass:regLQW}. 
We choose $\widetilde{A}_0=S(u)=\Sigma(u)^{-1}$. 
Then, the corresponding boundary matrix is given by $\widetilde{A}_{\rm nor}=G_{\rm nor}$, whose eigenvalues are given by 
$\lambda_1=0$, $\lambda_2=1$, and $\lambda_3=-1$, so that we can choose $n_1=1$ and $n_2=2$. 
The corresponding eigenvectors are given by $\bm{l}_1=(0,(N^\perp)^{\rm T})^{\rm T}$, $\bm{l}_2=(1,N^{\rm T})^{\rm T}$, and 
$\bm{l}_3=(-1,N^{\rm T})^{\rm T}$. 
The condition (ii) in Assumption \ref{asschar} is satisfied with $\bm{q}_{1,1}=(0,{\bf e}_2^{\rm T})^{\rm T}$, 
$\bm{q}_{1,2}=-(0,{\bf e}_1^{\rm T})^{\rm T}$, and $w_1(t,x)\equiv0$. 
Moreover, as for the condition (iii) we see that 
$\det L = 2S(u)\begin{psmallmatrix} 0 \\ N^\perp \end{psmallmatrix} \cdot \begin{psmallmatrix} 0 \\ N^\perp \end{psmallmatrix}$, 
which is strictly positive. 
Therefore, we can apply Theorem \ref{theogenorderm2} to obtain the desired estimate. 
\end{proof}

\subsection{Nonlinear a priori estimates}\label{sectNLAP}
We are now ready to derive a priori estimates for the nonlinear problem \eqref{PB1bis}--\eqref{PB3bis} which imply the following stability result.

\begin{theorem}\label{th:APE}
Let $m\geq 3$ be an integer and assume that Assumption \ref{ass:hi} is satisfied. 
Let also $(u,\psi_{\rm i})$ be a regular solution of \eqref{PB1bis}--\eqref{PB3bis} such that the initial data satisfy 
$(u(0,\cdot),\dtan\psi(0,\cdot))\in H^m(\cE)\times H^{m-1/2}(\mathit{\Gamma})$ and the subcriticality condition 
\[
\gr h(0,x)-\abs{v(0,x)}^2 \geq 2 c_0 \quad\mbox{for}\quad x\in\cE
\]
with a positive constant $c_0$. 
Then there exist a positive time $T_0$ and a constant $M_0$ depending only on $\Abs{u(0,\cdot)}_{H^m(\cE)}$, 
$\abs{\dtan\psi_{\rm i}(0,\cdot)}_{H^{m-1/2}(\mathit{\Gamma})}$, and $c_0^{-1}$ such that we have 
\[
\begin{cases}
 \opnorm{u(t)}_m^2 + \abs{\dtan\psi_{\rm i}(t)}_{H^{-1/2}_{(m)}}^2 \leq M_0, \\
 \gr h(t,x) - \abs{v(t,x)}^2 \geq c_0 \quad\mbox{for}\quad (t,x)\in(0,T)\times\cE.
\end{cases}
\]
\end{theorem}

\begin{proof}
Let us introduce the notation 
\[
J^m_{\lambda,t}(u,\psi_{\rm i}):=I_{\lambda,t}(\opnorm{u(\cdot)}_{m})^2+{I}_{\lambda,t}(\abs{\dtan\psi_{\rm i}(\cdot)}_{H_{(m)}^{-1/2}})^2,
\]
which is the quantity we are going to control. 
Decomposing $u=\chi_{\rm b}u+(1-\chi_{\rm b})u$, we get 
\begin{equation}\label{decompJ1}
J^m_{\lambda,t}(u,\psi_{\rm i}) \lesssim J^m_{\lambda,t}(\chi_{\rm b}u,\psi_{\rm i})+I_{\lambda,t}(\opnorm{(1-\chi_{\rm b})u(\cdot)}_{m})^2,
\end{equation}
and we further decompose the first term of the right-hand side as 
\begin{align}\label{decompJ2}
J^m_{\lambda,t}(\chi_{\rm b}u,\psi_{\rm i})
&\lesssim J^0_{\lambda,t}(\chi_{\rm b}u,\psi_{\rm i}) + J^{m-1}_{\lambda,t}(\chi_{\rm b}\dt u,\dt\psi_{\rm i})
 + J^{m-1}_{\lambda,t}(\chi_{\rm b}\dtan u,\dtan\psi_{\rm i}) \\
&\quad\;
 + I_{\lambda,t}(\opnorm{\dnor(\chi_{\rm b}u)}_{m-1})^2, \nonumber
\end{align}
where we use the fact that $\chi_{\rm b}$ commutes with $\dt$ and $\dtan$. 
There are therefore five terms to control: 
$I_{\lambda,t}(\opnorm{(1-\chi_{\rm b})u(\cdot)}_{m})$, $J^0_{\lambda,t}(\chi_{\rm b}u,\psi_{\rm i})$, 
$J^{m-1}_{\lambda,t}(\chi_{\rm b}\dpar u,\dpar\psi_{\rm i})$ ($\dpar=\dt$ or $\dtan$), 
and $I_{\lambda,t}(\opnorm{\dnor(\chi_{\rm b}u)}_{m-1})$. 
We evaluate these terms in four separate lemmas. 
In the statements of these lemmas, the constants $K_0$, $K$, and $K^{\rm in}$ are such that
\[
\begin{cases}
\frac{1}{c_0}, \Abs{u}_{L^{\infty}((0,T)\times\cE)} \leq K_0, \\
K_0, \Abs{u}_{\mathbb{W}^m_T} \leq K, \\
\Abs{u(0)}_{H^{m-1}(\cE)} \leq K^{\rm in}. 
\end{cases}
\]

\begin{lemma}\label{lemmapf1}
Let $u=(\zeta,v^{\rm T})^{\rm T}$ be a regular solution of \eqref{PB1bis} such that Assumption \ref{ass:subcritical} is satisfied. 
Then, 
\[
I_{\lambda,t}(\opnorm{(1-\chi_{\rm b})u(\cdot)}_{m})^2 \leq C(K_0)\opnorm{u(0)}_m^2 + C(K)\lambda^{-2}I_{\lambda,t}(\opnorm{u(\cdot)}_m)^2
\]
holds for any $\lambda\geq\lambda_0(K)$ and $t\in[0,T]$. 
\end{lemma}

\begin{proof}[Proof of lemma \ref{lemmapf1}]
We write \eqref{PB1bis} in the form 
\[
\dt u+A_j(u)\partial_j u = 0 \quad\mbox{in}\quad (0,T)\times\cE,
\]
where $A_j(u)$ is given by \eqref{defAj} in Proposition \ref{proplintransf}, 
so that we checks readily that $\dot{u}=(1-\chi_{\rm b})u$ solves 
\[
\dt\dot{u} + A_j(u)\partial_j\dot{u} = f \quad\mbox{in}\quad (0,T)\times\cE
\]
with $f=-(\partial_j\chi_{\rm b}) A_j(u)u$. 
Here, we see easily that $\Sigma(u)$ is a symmetrizer of this system and that we have 
$C(K_0)^{-1}\mathrm{Id}_{3\times3} \leq \Sigma(u) \leq C(K_0)\mathrm{Id}_{3\times3}$ in $(0,T)\times\cE$ and 
\[
\Abs{ \dt\Sigma(u) + \partial_j(\Sigma(u)A_j(u)) }_{L^\infty((0,T)\times\cE)} \leq C(K_0)K. 
\]
Moreover, by the embedding $\mathbb{W}_T^2 \hookrightarrow \mathbb{W}_T^{1,p}$ for any $p\in[2,\infty)$ we have also 
$\Abs{ \partial A_j(u) }_{\mathbb{W}_T^{m-1} \cap \mathbb{W}_T^{1,p}} \leq C(p,K)$ for $j=1,2$. 
Since $\opnorm{f(t)}_m\leq C(K)\opnorm{u(t)}_m$, the result therefore follows directly from Proposition \ref{propNRHhighfar} and \eqref{propSstar}. 
\end{proof}

\begin{lemma}\label{lemmapf2}
Under Assumption \ref{ass:hi}, let $(u,\psi_{\rm i})$ be a regular solution of \eqref{PB1bis}--\eqref{PB3bis} such that 
Assumption \ref{ass:subcritical} is satisfied. 
Then, 
\[
J^0_{\lambda,t}(u,\psi_{\rm i})\leq C(K_0)\bigl( \Abs{u(0)}_{L^2}^2 + \abs{\dtan \psi_{\rm i}(0)}_{H^{-1/2}}^2 \bigr).
\]
holds for any $\lambda>0$ and $t\in[0,T]$. 
\end{lemma}

\begin{proof}[Proof of lemma \ref{lemmapf2}]
From the energy conservation in Proposition \ref{propNRJcons} and the coercivity of the Dirichlet-to-Neumann map established in Proposition \ref{propDN}, 
we have 
\[
\Abs{u(t)}_{L^2}^2 + \abs{\dtan \psi_{\rm i}(t)}_{H^{-1/2}}^2 \leq C(K_0)\bigl( \Abs{u(0)}_{L^2}^2 +\abs{\dtan \psi_{\rm i}(0)}_{H^{-1/2}}^2 \bigr),
\]
from which we infer the result.
\end{proof}

\begin{lemma}\label{lemmapf3}
Under Assumption \ref{ass:hi}, let $(u,\psi_{\rm i})$ be a regular solution of \eqref{PB1bis}--\eqref{PB3bis} such that 
Assumption \ref{ass:subcritical} is satisfied. 
Then, 
\[
J_{\lambda,t}^{m-1}(\chi_{\rm b}\dpar \uu,\dpar\psi_{\rm i}) \leq 
C(K_0,K^{\rm in})\bigl( \opnorm{u(0)}^2_{m} + \abs{\dtan\psi_{\rm i}(0)}_{H^{m-1/2}}^2 \bigr) + C(K)\lambda^{-2}J^m_{\lambda,t}(u,\psi)
\]
holds for any $\lambda\geq\lambda_0(K)$ and $t\in[0,T]$. 
\end{lemma}

\begin{proof}[Proof of lemma \ref{lemmapf3}]
Let us write $(\dot{u},\dot{\psi}_{\rm i})=(\chi_{\rm b}\dpar u,\chi_{\rm b}\dpar \psi_{\rm i})$, which solves \eqref{pretransf1}--\eqref{deffg1g2}. 
Defining, as in Proposition \ref{proplintransf}, $\widecheck{u}=\Sigma(u)\dot{u}$ and $\widecheck{\psi}_{\rm i}=\dot{\psi}$ which solves 
\eqref{PB1bislin0}--\eqref{PB2bislin0} and denoting $S(u)=\Sigma(u)^{-1}$, we have 
\begin{align*}
\opnorm{\dot{u}(t)}_{m-1}
&= \opnorm{ S(u)\widecheck{u}(t) }_{m-1} \\
&\leq C(K_0)\opnorm{ \widecheck{u}(t) }_{m-1} + \sum_{\abs{\alpha}\leq m-1}\Abs{ [\partial^\alpha, S(u)]\widecheck{u}(t) }_{L^2}. 
\end{align*}
Here, by using \eqref{eqdtS} and Lemmas \ref{lemmcommut} and \ref{lem:Moser} we have 
\begin{align*}
\Abs{ [\partial^\alpha, S(u)]\widecheck{u}(t) }_{L^2}
&\leq C\bigl( \Abs{ [\partial^\alpha, S(u)]\widecheck{u}(0) }_{L^2} + S_{\lambda,t}^*( \Abs{ \dt[\partial^\alpha, S(u)]\widecheck{u}(\cdot) }_{L^2} ) \\
&\leq C\opnorm{ \dot{u}(0) }_{m-1} + C(K_0)\opnorm{ \widecheck{u}(0) }_{m-1} + C(K)S_{\lambda,t}^*( \opnorm{ \widecheck{u}(\cdot) }_{m-1}) \\
&\leq C(K_0,K^{\rm in})\opnorm{ u(0) }_m + C(K)S_{\lambda,t}^*( \opnorm{ \widecheck{u}(\cdot) }_{m-1}),
\end{align*}
and therefore, 
\[
I_{\lambda,t}(\opnorm{ \dot{u}(\cdot) }_{m-1})
\leq C(K_0) I_{\lambda,t}(\opnorm{\widecheck{u}(\cdot)}_{m-1}) + C(K_0,K^{\rm in})\opnorm{ u(0) }_{m}
 + C(K)S_{\lambda,t}^*(\opnorm{\widecheck{u}(\cdot)}_{m-1}).
\]
The last term can be absorbed into the first term in the right-hand side for sufficiently large $\lambda$, so that 
\[
I_{\lambda,t}(\opnorm{ \dot{u}(\cdot) }_{m-1} \leq C(K_0)I_{\lambda,t}(\opnorm{ \widecheck{u}(\cdot) }_{m-1}) + C(K_0,K^{\rm in})\opnorm{ u(0) }_{m},
\]
and then 
\begin{equation}\label{eqpreclem}
J_{\lambda,t}^{m-1}(\dot{u},\dot{\psi}_{\rm i}) 
\leq C(K_0)J_{\lambda,t}^{m-1}(\widecheck{u},\widecheck{\psi}_{\rm i}) + C(K_0,K^{\rm in})\opnorm{u(0)}_{m}. 
\end{equation}
The result stated in the lemma is a direct consequence of \eqref{eqpreclem} and of the following upper bound on 
$J_{\lambda,t}^{m-1}(\widecheck{u},\widecheck{\psi}_{\rm i})$ 
\begin{equation}\label{eqclaim}
J_{\lambda,t}^{m-1}(\widecheck{u},\widecheck{\psi}_{\rm i}) \leq 
C(K_0,K^{\rm in})\bigl( \opnorm{u(0)}^2_{m} + \abs{\dtan\psi_{\rm i}(0)}_{H^{m-1/2}}^2 \bigr) + C(K)\lambda^{-2}J^m_{\lambda,t}(u,\psi_{\rm i}).
\end{equation}
that we now turn to prove. 
We note that by the embedding $\mathbb{W}_T^3 \hookrightarrow \mathbb{W}_T^{2,p}$ for any $p\in[2,\infty)$ we have 
$\Abs{ u }_{\mathbb{W}_T^{m-1} \cap \mathbb{W}_T^{2,p}} \leq CK$, 
so that we can use Corollary \ref{corhigher} with $m$ replaced by $m-1$ to obtain 
\begin{align*}
J_{\lambda,t}^{m-1}(\widecheck{u},\widecheck{\psi}_{\rm i})
&\leq C(K_0)\bigl( \opnorm{\widecheck{u}(0)}^2_{m-1} + \abs{\dtan\widecheck{\psi}_{\rm i}(0)}_{H_{(m-1)}^{-1/2}}^2
 + \abs{g_1(0)}^2_{H^{-1/2}_{(m-1)}} \bigr) \\
&\quad\;
 + C(K_0)\bigl( S_{\lambda,t}^*(\opnorm{f(\cdot)}_{m-1})^2
  + S^*_{\lambda,t}(\abs{g_1(\cdot)}_{H^{-1/2}_{(m)}})^2 + \lambda^{-1}\abs{\dtan g_2}_{L^2_{\lambda,t}H_{(m-1)}^{-1/2}}^2 \bigr) \\
&=: J_1+J_2,
\end{align*}
where $f=\Sigma(u)\dot{f}-B(u,\dt u)\widecheck{u}$ with $\dot{f}$, $g_1$, and $g_2$ given by \eqref{deffg1g2}, that is, 
\[
\dot{f}=-[\chi_{\rm b}\dpar,\partial_j]{\mathcal F}_j(u), \quad
g_1=[\dpar,\Lambda]\psi_{\rm i}, \quad\mbox{and}\quad
g_2=-(\dpar N)\cdot{\mathcal F}_{\rm nor}(u)^{\rm II}.
\]
Here, we recall that $B(u,\dt)=\Sigma(u)\dt(\Sigma(u)^{-1})$, 
${\mathcal F}_j(u)^{\rm T}=\bigl(hv_j,(\gr\zeta+\frac{1}{2}\abs{v}^2){\bf e}_j^{\rm T}\bigr)$, and ${\mathcal F}_{\rm nor}(u)=N_j{\mathcal F}_j(u)$. 
Let us control separately $J_1$ and $J_2$, which deal with the initial data and source terms, respectively.

-- Control of $J_1$. 
From the definition of $\widecheck{u}$, we have easily $ \opnorm{\widecheck{u}(0)}_{m-1}\leq C(K^{\rm in}) \opnorm{u(0)}_{m}$ 
and using Proposition \ref{propDN}, we get $\abs{g_1(0)}_{H^{-1/2}_{(m-1)}}\lesssim \abs{\dtan\psi_{\rm i}(0)}_{H^{-1/2}_{(m-1)}}$, so that 
\[
J_1\leq C(K_0,K^{\rm in})\bigl( \opnorm{u(0)}^2_{m} + \abs{\dtan\psi_{\rm i}(0)}_{H_{(m)}^{-1/2}}^2 \bigr).
\]
In order to get an upper bound on $\abs{\dtan\psi_{\rm i}(0)}_{H_{(m)}^{-1/2}}$, 
we need to give an upper bound on $\abs{\dpar^\alpha\dtan\psi_{\rm i}(0)}_{{H}^{-1/2}}$ for all $\alpha=(\alpha_0,\alpha_1)\in \N^2$ 
such that $\abs{\alpha}\leq m$. 
In the case $\alpha_0=0$, $\dpar^\alpha$ contains only tangential derivatives so that 
$\abs{\dpar^\alpha\dtan\psi_{\rm i}(0)}_{H^{-1/2}} \leq \abs{\dtan\psi_{\rm i}(0)}_{H^{m-1/2}}$. 
In the other cases, we have $\abs{\dpar^\alpha\dtan\psi_{\rm i}(0)}_{H^{-1/2}}\leq \abs{\dt\psi_{\rm i}(0)}_{H^{1/2}_{(m-1)}}$. 
Replacing $\dt\psi_{\rm i} = -\gr\zeta-\frac{1}{2}\abs{v}^2$ and using the trace theorem, 
we deduce that $\abs{\dpar^\alpha\dtan\psi_{\rm i}(0)}_{H^{-1/2}}\leq C(K^{\rm in})\opnorm{u(0)}_m$. 
Therefore, 
\[
J_1 \leq C(K_0,K^{\rm in})\bigl( \opnorm{u(0)}^2_{m} + \abs{\dtan\psi_{\rm i}(0)}_{{H}^{m-1/2}}^2 \Bigr).
\]

-- Control of $J_2$. 
Using Proposition \ref{propDN}, Lemma \ref{lem:Moser}, and the trace theorem, we get 
\[
\begin{cases}
 \abs{g_1(t)}_{H^{-1/2}_{(m)}} \leq C\abs{\dtan \psi_{\rm i}(t)}_{H^{-1/2}_{(m)}}, \\
 \opnorm{ f(t) }_{m-1} + \abs{\dtan g_2(t)}_{H_{(m-1)}^{-1/2}} \leq C(K)\opnorm{u(t)}_m.
\end{cases}
\]
From \eqref{propSstar}, we then obtain 
\[
J_2 \leq C(K)\lambda^{-2} J^m_{\lambda,t}(u,\psi_{\rm i}).
\]

These upper bounds on $J_1$ and $J_2$ prove \eqref{eqclaim} and therefore complete the proof of the lemma.
\end{proof}

\begin{lemma}\label{lemmapf4}
Let $u=(\zeta,v^{\rm T})^{\rm T}$ be a regular solution of \eqref{PB1bis} such that Assumption \ref{ass:subcritical} is satisfied. 
Then, 
\begin{align*}
I_{\lambda,t}(\opnorm{\dnor(\chi_{\rm b}u)}_{m-1})^2
&\leq C(K_0)\bigl( \opnorm{u(0)}^2_{m} + I_{\lambda,t}(\opnorm{(\chi_{\rm b}\dt u,\chi_{\rm b}\dtan u)}_{m-1})^2 \bigr) \\
&\quad\;
 + C(K)\lambda^{-2}I_{\lambda,t}(\opnorm{u(\cdot)}_m)^2
\end{align*}
holds for any $\lambda>0$ and $t\in[0,T]$. 
\end{lemma}

\begin{proof}[Proof of the lemma \ref{lemmapf4}]
The proof can be carried out in the way as the proof of Proposition \ref{propfarorderm2}. 
Let us write $\dot{u}=\chi_{\rm b} u$. 
Then, we have 
\[
\dt\dot{u} + A_j(u)\partial_j\dot{u} = f \quad\mbox{in}\quad (0,T)\times\cE
\]
with $f=(\partial_j \chi_{\rm b})A_j(u)u$, where $A_j(u)$ is given by \eqref{defAj} in Proposition \ref{proplintransf}. 
Denoting $\dot{u}^{(j,k,l)}=\dt^j\dtan^k\dnor^l\dot{u}$ as before, for $j+k+l\leq m-1$ we have 
\[
\dt\dot{u}^{(j,k,l)} + A_1\dot{u}^{(j,k,l)} + A_2\dot{u}^{(j,k,l)} = f_{j,k,l} \quad\mbox{in}\quad (0,T)\times\cE. 
\]
We proceed to verify that the coefficient matrices $A_1$ and $A_2$ satisfy Assumptions \ref{asschar} and \ref{ass:regLQW}. 
We choose $\widetilde{A}_0=S(u)=\Sigma(u)^{-1}$. 
Then, the corresponding boundary matrix is given by $\widetilde{A}_{\rm nor}=\Sigma(u)^{-1}G_{\rm nor}\Sigma(u)$, 
whose eigenvalues are given by $\lambda_1=0$, $\lambda_2=1$, and $\lambda_3=-1$, so that we choose $n_1=1$ and $n_2=2$. 
The corresponding left eigenvectors are given by $\bm{l}_1=\Sigma(u)(0,(N^\perp)^{\rm T})^{\rm T}$, 
$\bm{l}_2=\Sigma(u)(1,N^{\rm T})^{\rm T}$, and $\bm{l}_3=\Sigma(u)(-1,N^{\rm T})^{\rm T}$. 
The condition (ii) in Assumption \ref{asschar} is satisfied with $\bm{q}_{1,1}=\Sigma(u)(0,{\bf e}_2^{\rm T})^{\rm T}$, 
$\bm{q}_{1,2}=-\Sigma(u)(0,{\bf e}_1^{\rm T})^{\rm T}$, and $w_1(t,x)\equiv0$. 
Moreover, as for the condition (iii) we see that 
$\det L = 2\det(\Sigma(u))S(u)\begin{psmallmatrix} 0 \\ N^\perp \end{psmallmatrix} \cdot \begin{psmallmatrix} 0 \\ N^\perp \end{psmallmatrix}$, 
which is strictly positive. 
Therefore, we can apply Lemma \ref{lemmvor} to obtain 
\begin{align*}
I_{\lambda,t}( \|u^{(j,k,l+1)}(\cdot)\|_{L^2} )
&\leq C(K_0)\bigl( \opnorm{ u^{(j,k,l)}(0) }_1 + I_{\lambda,t}( \|\dpar u^{(j,k,l)}(\cdot)\|_{L^2} \bigr) \\
&\quad\;
 + C(K) S_{\lambda,t}^*( \opnorm{ f_{j,k,l}(\cdot) }_1 ).
\end{align*}
Using this inductively, we get the desired estimate. 
\end{proof}

By \eqref{decompJ1}, \eqref{decompJ2}, and Lemmas \ref{lemmapf1}--\ref{lemmapf4}, we get 
\[
J_{\lambda,t}^{m}( u,\psi_{\rm i}) \leq 
C(K_0,K^{\rm in})\bigl( \opnorm{u(0)}_{m}^2 + \abs{\dtan\psi_{\rm i}(0)}_{H^{m-1/2}}^2 \bigr) + C(K)\lambda^{-2}J^m_{\lambda,t}(u,\psi_{\rm i}).
\]
The last term can be absorbed by the left-hand side for sufficiently large $\lambda$. 
We can also use the equation to get $\opnorm{u(0)}_{m} \leq C(K^{\rm in})\Abs{u(0)}_{H^m(\cE)}$, leading to the upper bound 
\begin{equation}\label{eqNRJNL}
J_{\lambda,t}^{m}( u,\psi_{\rm i}) \leq C_1(K_0,K^{\rm in})\bigl( \Abs{u(0)}_{H^m(\cE)}^2 + \abs{\dtan\psi_{\rm i}(0)}_{H^{m-1/2}}^2 \bigr)
\end{equation}
for any $\lambda\geq\lambda_0(K)$ and any $t\in[0,T]$. 
We note also that 
\begin{align}\label{estsubcritical}
\gr h(t,x)-\abs{v(t,x)}^2 
&= \gr h(0,x)-\abs{v(0,x)}^2 +\int_0^t \big( \gr\dt\zeta - 2v\cdot\dt v\big)(t',x){\rm d}t' \\
&\geq 2 c_0 - C_2(K)t \nonumber
\end{align}
hold for any $(t,x)\in(0,T)\times\cE$, where we used the Sobolev embedding $\mathbb{W}_T^m \hookrightarrow W^{1,\infty}((0,T)\times\cE)$. 
Similarly, for any $(t,x)\in(0,T)\times\cE$ we have 
\begin{equation}\label{estLinf}
\abs{u(t,x)} \leq \Abs{ u(0,\cdot) }_{L^\infty(\cE)} + C_2(K)t.
\end{equation}
We can now conclude the proof of the theorem. 
We recall that we want to find $M_0$ and $T_0$ that depend only on the initial data, more precisely, 
on $c_0$ and an upper bound of $\Abs{ u(0) }_{H^m(\cE)}$ and $\abs{\dtan\psi_{\rm i}(0)}_{H^{m-1/2}}$, 
such that $\mathscr{E}_m(t):=\opnorm{u(t)}_m^2+\abs{\dtan\psi_{\rm i}(t)}_{H^{-1/2}_{(m)}}^2 \leq M_0$ holds over the time interval $[0,T_0]$. 
We will do this in two steps.

\medskip
\noindent
{\bf Step 1.} Choice of $M_0$ and $T_0$. 
We first choose positive constants $K_0$ and $K^{\rm in}$ such that 
\[
\begin{cases}
 \frac{1}{c_0}, 2\Abs{ u(0,\cdot) }_{L^\infty(\cE)} \leq K_0, \\
 \Abs{ u(0,\cdot) }_{H^{m-1}(\cE)} \leq K^{\rm in},
\end{cases}
\]
where $c_0$ is the positive constant in the subcriticality condition on the initial data. 
Then, we choose the positive constant $M_0$ such that 
\[
2C_1(K_0,K^{\rm in})\bigl( \Abs{u(0)}_{H^m(\cE)}^2 + \abs{\dtan\psi_{\rm i}(0)}_{H^{m-1/2}}^2 \bigr) \leq M_0,
\]
where $C_1(K_0,K^{\rm in})$ is the constant in \eqref{eqNRJNL}. 
In view of $\opnorm{ u(t) }_m ^2 \leq e^{2\lambda t}J_{\lambda,t}^m(u,\psi_{\rm i})$, 
we choose the positive constant $K$ such that $\frac{1}{c_0}, \sqrt{M_0} \leq K$. 
Finally, we choose a positive time $T_0$ such that 
\[
\begin{cases}
 2\lambda_0(K)T_0 \leq \log2, \\
 C_2(K)T_0 \leq \min\{ c_0,\frac12K_0 \},
\end{cases}
\]
where $\lambda_0(K)$ is the constant so that \eqref{eqNRJNL} holds for any $\lambda\geq\lambda_0(K)$ and 
$C_2(K)$ is the constants in \eqref{estsubcritical} and \eqref{estLinf}. 
Here, we note that by the choice of the constant $M_0$ we have $\mathscr{E}_m(0) \leq \frac12M_0$.

\medskip
\noindent
{\bf Step 2.}
We will show that under the choice of the constants $M_0$ and $T_0$ in the previous step 
$\mathscr{E}_m(t) \leq M_0$ holds for any $t\in[0,T_0]$ by contradiction. 
In fact, if it does not hold, then there exists a time $t_0\in(0,T_0]$ such that $\mathscr{E}_m(t_0)>M_0$. 
Then, we can define a time $T_*$ by 
\[
T_* = \inf\{t>0 \,|\, \mathscr{E}_m(t) > M_0 \}.
\]
Since $\mathscr{E}_m(0)<M_0$ and $\mathscr{E}_m(t_0)>M_0$, we have $0<T_*<t_0\leq T_0$. 
We have also $\mathscr{E}_m(t)\leq M_0$ for $t\in[0,T_*]$. 
Therefore, by our choice of the constants in Step 1 we see that the subcriticality condition in Assumption \ref{ass:subcritical}, 
$\Abs{ u(t) }_{L^\infty(\cE)} \leq K_0$, and $\opnorm{ u(t) }_m\leq K$ hold for any $t\in[0,T_*]$, 
so that \eqref{eqNRJNL} also holds for any $t\in[0,T_*]$. 
Particularly, we have $\mathscr{E}_m(t) \leq \frac12 M_0 e^{2\lambda_0(K)T_*}<\frac12 M_0 e^{2\lambda_0(K)T_0} \leq M_0$ for any $t\in[0,T_*]$. 
Therefore, by the continuity in time, there exists a small $\varepsilon>0$ such that $\mathscr{E}_m(t)<M_0$ holds for any $t\in[0,T_*+\varepsilon]$, 
so that by the definition of $T_*$ we get $T_*+\varepsilon\leq T_*$, which is a contradiction. 
This completes the proof. 
\end{proof}

\section{Well-posedness of the linearized hyperbolic system}\label{sectexistlin}
The aim of this section is to establish the well-posedness of the initial boundary value problem to the transformed linearized equations 
\eqref{PB1bislin}--\eqref{PB2bislin} for the wave-structure interaction problem, 
so that we will consider in this section the following linear symmetric hyperbolic system 
\begin{equation}\label{PB1lin}
S(t,x)\dt u + G_j\partial_ju = f(t,x) \quad\mbox{in}\quad (0,T)\times\cE
\end{equation}
under the boundary conditions 
\begin{equation}\label{PB2lin}
\begin{cases}
 N\cdot u^\mathrm{II} - \Lambda\psi_{\rm i} = g_1(t,x) &\mbox{on}\quad (0,T)\times\mathit{\Gamma}, \\
 \dt\psi_{\rm i} + u^\mathrm{I} = g_2(t,x) &\mbox{on}\quad (0,T)\times\mathit{\Gamma},
\end{cases}
\end{equation}
where $f$ is an $\R^3$-valued function defined in $(0,T)\times \cE$ for some $T>0$, 
$S$ takes its values in the space of $3\times 3$ real symmetric and positive definite matrices, 
$G_j$ $(j=1,2)$ are constant symmetric matrices given by 
\[
G_j = 
\begin{pmatrix}
 0 & \mathbf{e}_j^{\rm T} \\
 \mathbf{e}_j & 0_{2\times2}
\end{pmatrix},
\]
$N$ is the extension of the normal vector to $\mathit{\Gamma}$ constructed in Section \ref{sectnoprmtang}, $\Lambda$ is the 
Dirichlet-to-Neumann map defined in Section \ref{sectDN}, $g_1$ and $g_2$ are real-valued functions defined in $(0,T)\times\mathit{\Gamma}$. 
Comparing \eqref{PB1lin} to \eqref{PB1bislin}, the equations have been multiplied by $S=\Sigma^{-1}$; both formulations are of course equivalent. 
We also recall that for $u=(u_1,u_2,u_3)^{\rm T}\in\R^3$ we write $u^{\rm I}=u_1$ and $u^{\rm II}=(u_2,u_3)^{\rm T}$. 
We impose the initial conditions 
\begin{equation}\label{IC2}
\begin{cases}
 u_{\vert_t=0}=u^{\rm in} &\mbox{in}\quad \cE, \\
 {\psi_{\rm i}}_{\vert_t=0}=\psi_{\rm i}^{\rm in} &\mbox{on}\quad \mathit{\Gamma}.
\end{cases}
\end{equation}
Since $S(t,x)$ is invertible by assumption, the symmetric hyperbolic system \eqref{PB1lin} can be written in the form \eqref{eqlinfixed} 
and $S(t,x)$ plays the role of a Friedrichs symmetrizer. 
The corresponding boundary quadratic form defined as \eqref{defB} is now given by $\mathfrak{B}[u]=G_{\rm nor}u\cdot u$ with $G_{\rm nor}=N_jG_j$. 
The eigenvalues of the boundary matrix $G_{\rm nor}$ are $\pm1$ and $0$, so that the boundary is characteristic.

In order to prove the well-posedness of the problem \eqref{PB1lin}--\eqref{IC2}, 
we first propose in Section \ref{sectregul} a regularization of the problem for which the well-posedness is known. 
We then establish in Section \ref{sectunifEE} uniform energy estimates with respect to the regularization parameter. 
As usual with hyperbolic initial boundary value problems, we have to impose compatibility conditions on the data to allow regular solutions. 
Such compatibility conditions are studied in Section \ref{sect:CC} and several approximation results proved; 
the task is made delicate by the fact that the problem is characteristic. 
It is in particular shown that data compatible up to order $m-1$ can be approximated by more regular data satisfying higher order compatibility conditions. 
The well-posedness of the linearized equations is then proved in Section \ref{subsectexistlin}. 
It is important to note that we are able to work with regularity requirements on the coefficients of the operator 
that are lower than those usually required for linear IBVPs. 
This will be used to get the critical regularity in Theorem \ref{th:exist}.

\subsection{Regularized problem}\label{sectregul}
As we have seen in Propositions \ref{propWSweakdiss} and \ref{propWSweakdissm}, for any regular solutions of \eqref{PB1lin}--\eqref{PB2lin} 
the boundary term is weakly dissipative in the sense of Definitions \ref{propweakdissip} and \ref{propweakdissipm} so that a priori energy 
estimates of the solutions can be obtained. 
Nevertheless, a standard theory of the initial boundary problem for hyperbolic systems cannot be applicable directly to show the existence of regular solutions. 
To show the existence, we first regularize the equations with a regularizing parameter $\varepsilon>0$ and construct a solution 
$(u^\varepsilon,\psi_{\rm i}^\varepsilon)$ to the regularized problem. 
Then, by passing to the limit $\varepsilon\to+0$ we shall show that the series of approximate solutions 
$\{(u^\varepsilon,\psi_{\rm i}^\varepsilon)\}_{\varepsilon>0}$ converges to a solution of the original problem \eqref{PB1lin}--\eqref{PB2lin}. 
In this article, we adopt the following regularized problem 
\begin{equation}\label{PB1reg}
S(t,x)(\dt-\varepsilon\chi_{\rm b}N\cdot\nabla)u + G_j\partial_ju = f(t,x) - \varepsilon S(t,x)(\chi_{\rm b}N\cdot\nabla)u^{(0),\varepsilon} \big) 
 \quad\mbox{in}\quad (0,T)\times\cE
\end{equation}
under the regularized boundary conditions 
\begin{equation}\label{PB2reg}
\begin{cases}
 N\cdot u^\mathrm{II} = J_\varepsilon\Lambda J_\varepsilon\psi_{\rm i} + g_1(t,x) &\mbox{on}\quad (0,T)\times\mathit{\Gamma}, \\
 \dt\psi_{\rm i} + u^\mathrm{I} = g_2(t,x) &\mbox{on}\quad (0,T)\times\mathit{\Gamma},
\end{cases}
\end{equation}
where $\chi_{\rm b}$ is the smooth cutoff function supported in $U_\mathit{\Gamma}$ defined in \eqref{defchib}, 
$u^{(0),\varepsilon}$ will be constructed a priori from the initial data so that 
$(\dt^ju^{(0),\varepsilon})_{\vert_{t=0}} = (\dt^ju)_{\vert_{t=0}}$ holds for $j=0,1,2,\ldots$, 
and $J_\varepsilon=(1+\varepsilon\abs{D})^{-1}$ is a smoothing operator. 
Here, we used a standard notation of Fourier multipliers acting on functions defined on $\mathit{\Gamma}$, so that $\abs{D}=(-\dtan^2)^{1/2}$. 
We note that this type of regularization in \eqref{PB1reg} was already used in the analysis of the local well-posedness for the compressible 
Euler equations in a fixed domain \cite{Schochet1986}. 
We impose the initial conditions 
\begin{equation}\label{ICreg}
\begin{cases}
 u_{\vert_t=0}=u^{{\rm in},\varepsilon} &\mbox{in}\quad \cE, \\
 {\psi_{\rm i}}_{\vert_t=0}=\psi_{\rm i}^{\rm in} &\mbox{on}\quad \mathit{\Gamma},
\end{cases}
\end{equation}
where the initial data $u^{{\rm in},\varepsilon}$ is in general not the same data as $u^{\rm in}$ to the original problem \eqref{PB1lin}--\eqref{PB2lin}; 
even if the data $(u^{\rm in},\psi_{\rm i}^{\rm in})$ satisfy the compatibility conditions to the original problem \eqref{PB1lin}--\eqref{PB2lin}, 
which will be introduced in Section \ref{sect:CC} below, they do not necessarily satisfy the compatibility conditions to the regularized problem 
\eqref{PB1reg}--\eqref{PB2reg}. 
Therefore, we need to modify the initial data $u^{\rm in}$ so that the modified data $(u^{{\rm in},\varepsilon},\psi_{\rm i}^{\rm in})$ satisfy 
the compatibility conditions to the regularized problem and that $\{u^{{\rm in},\varepsilon}\}_{\varepsilon>0}$ converges to $u^{\rm in}$ as 
$\varepsilon\to+0$.

The boundary matrix to the regularized hyperbolic system \eqref{PB1reg} is given by $G_\mathrm{nor}-\varepsilon S$. 
Concerning the eigenvalues of this boundary matrix, we have the following lemma.

\begin{lemma}\label{lem:EVofRBM}
Let $S$ be a $3\times3$ real symmetric matrix. 
There exists a sufficiently small $\varepsilon_0>0$ such that for any unit vector $N\in\R^2$ and any 
$\varepsilon \in[-\varepsilon_0,\varepsilon_0]$ the matrix $G_\mathrm{nor}-\varepsilon S$ has three different eigenvalues 
$\lambda_0^\varepsilon(N)$ and $\pm\lambda_\pm^\varepsilon(N)$. 
Moreover, the eigenvalues satisfy 
\begin{equation}\label{EVexpansion}
\lambda_0^\varepsilon(N) = \lambda_0(N)\varepsilon + O(\varepsilon^2), \quad
\lambda_\pm^\varepsilon(N) = 1 + O(\varepsilon)
\end{equation}
as $\varepsilon\to0$, where 
$\lambda_0(N) = -\begin{psmallmatrix} 0 \\ N^\perp \end{psmallmatrix} \cdot S \begin{psmallmatrix} 0 \\ N^\perp \end{psmallmatrix}$. 
Particularly, if $S$ is positive definite, then for any $\varepsilon\in(0,\varepsilon_0]$ the matrix $G_\mathrm{nor}-\varepsilon S$ 
has one positive eigenvalue and two negative eigenvalues. 
\end{lemma}

\begin{proof}
In the case $\varepsilon=0$, the matrix $G_\mathrm{nor}$ has three different eigenvalues $0$ and $\pm1$, 
so that the perturbation theory from simple eigenvalues implies the first part of the lemma. 
Therefore, it is sufficient to show that 
$\lambda_0(N) = -\begin{psmallmatrix} 0 \\ N^\perp \end{psmallmatrix} \cdot S \begin{psmallmatrix} 0 \\ N^\perp \end{psmallmatrix}$. 
The characteristic equation of $G_\mathrm{nor}-\varepsilon S$ is given by $\det(-G_{\rm nor}+\lambda{\rm Id}_{3\times3}+\eps S)=0$. 
Here, we recall Jacobi's formula 
$\frac{\mathrm{d}}{\mathrm{d}\theta}\det(A(\theta)) = \operatorname{tr}(\operatorname{adj}(A(\theta))\frac{\mathrm{d}}{\mathrm{d}\theta}A(t))$, 
where $\operatorname{adj}(A)$ is the adjugate of the matrix $A$. 
Particularly, we have $\det(A+B)=\det A + \operatorname{tr}(\operatorname{adj}(A)B)+O(|B|^2)$ as $|B|\to0$. 
Therefore, as $\varepsilon\to0$ and $\lambda\to0$, the characteristic equation can be expanded as 
\begin{align*}
0&= \operatorname{tr}\bigl( \operatorname{adj}(-G_{\rm nor})(\lambda{\rm Id}_{3\times3}+\eps S) \bigr) + O(\lambda^2+\eps^2) \\
&= -\lambda - \begin{pmatrix} 0 \\ N^\perp \end{pmatrix}\cdot S \begin{pmatrix} 0 \\ N^\perp \end{pmatrix}\varepsilon
 + O(\lambda^2+\varepsilon^2).
\end{align*}
Plugging $\lambda=\lambda_0(N)\varepsilon + O(\varepsilon^2)$ into the above equation and looking at the coefficient of 
$\varepsilon$, we obtain the expression of $\lambda_0(N)$. 
\end{proof}

\begin{remark}\label{rem:exist}
In view of this lemma, the number of the boundary conditions should be one for the regularized system \eqref{PB1reg}, 
which is consistent to the boundary conditions \eqref{PB2reg}, because the second condition in \eqref{PB2reg} can be integrated as 
$\psi_{\rm i}=\psi_{\rm i}^{\rm in}+\int_0^t(g_2-u^{\rm I}){\rm d}t'$ so that the unknown $\psi_{\rm i}$ could be eliminated from 
the boundary conditions, leading to one boundary condition. 
We also note that the term $J_\varepsilon\Lambda J_\varepsilon\psi_{\rm i}$ can be regarded as a lower order term 
thanks to the smoothing operator $J_\varepsilon$. 
Moreover, if we assume that $S(t,x)\geq \alpha_0\mathrm{Id}_{3\times3}$ holds for any $(t,x)\in(0,T)\times\cE$, then we see that 
\begin{align*}
(G_\mathrm{nor}-\varepsilon S)u\cdot u
&= 2u^{\rm I}(N\cdot u^{\rm II}) - \varepsilon Su\cdot u \\
&\leq 2\abs{u}\abs{ N\cdot u^{\rm II} } - \varepsilon \alpha_0 \abs{u}^2 \\
&\leq -\frac{\varepsilon \alpha_0}{2} \abs{u}^2 + \frac{2}{\varepsilon \alpha_0} \abs{ N\cdot u^{\rm II} }^2.
\end{align*}
Therefore, \eqref{PB2reg} could be regarded as strictly dissipative boundary conditions for the regularized system \eqref{PB1reg}, 
so that we can apply a standard theory to construct a solution to the regularized problem \eqref{PB1reg}--\eqref{ICreg}. 
\end{remark}

\subsection{Uniform energy estimates}\label{sectunifEE}
We will derive uniform energy estimates of regular solutions to the regularized problem \eqref{PB1reg}--\eqref{PB2reg} 
with respect to the small regularized parameter $\varepsilon$ along with the same lines as in Sections \ref{sectWSL2} and \ref{sectHOest} 
with slight modifications. 
For simplify the notation we put $f^\varepsilon=f-\varepsilon S(\chi_{\rm b}N\cdot\nabla)u^{(0),\varepsilon}$, so that the regularized 
hyperbolic system \eqref{PB1reg} can be written simply as 
\[
S(t,x)(\dt-\varepsilon\chi_{\rm b}N\cdot\nabla)u + G_j\partial_ju = f^\varepsilon(t,x) \quad\mbox{in}\quad (0,T)\times\cE.
\]
Corresponding to Assumption \ref{assFriedrichs}, we impose the following conditions.

\begin{assumption}\label{ass:S}
There exist a positive constant $c_0$ and a constant $\beta_1$ such that for any $(t,x)\in(0,T)\times\cE$ the following conditions hold. 
\begin{enumerate}
\item[{\rm (i)}]
$\alpha_0{\rm Id}_{3\times3} \leq S(t,x) \leq \beta_0 {\rm Id}_{3\times3}$. 
\item[{\rm (ii)}]
$\dt S(t,x) - \varepsilon\partial_j(\chi_{\rm b}(x)N_j(x)S(t,x)) \leq \beta_1{\rm Id}_{3\times3}$. 
\end{enumerate}
\end{assumption}

In the statement of the following proposition, we use the notational convention 
$\Abs{ \cdot }_{\mathbb{W}_T^{-1} \cap \mathbb{W}_T^{1,p}}=0$ in the case $m=0$. 
Note also that the quantity $S_{{\rm data},\eps}^m(t)$ is the same as the quantity $S_{{\rm data}}^m(t)$ that appears in Corollary \ref{corhigher}, 
except that $\dot{f}$ has been replaced by $f^\eps$ and that it contains an additional term 
$\varepsilon \abs{ g_1 }_{H_\lambda^m((0,t)\times\mathit{\Gamma})}^2$.

\begin{proposition}\label{prop:UEE}
Let $m$ be a non-negative integer and assume that Assumptions \ref{ass:hi} and \ref{ass:S} are satisfied. 
Assume also that $\partial S\in\mathbb{W}_T^{m-1} \cap \mathbb{W}_T^{1,p}$ for some $p\in(2,\infty)$ in the case $m\geq1$ 
and take two constants $0<K_0\leq K$ such that 
\[
\begin{cases}
 \frac{1}{\alpha_0},\beta_0,\frac{1}{\underline{\mathfrak{c}}},\underline{\mathfrak{C}} \leq K_0, \\
 K_0, \beta_1, \Abs{ \partial S }_{\mathbb{W}_T^{m-1} \cap \mathbb{W}_T^{1,p}} \leq K.
\end{cases}
\]
Then, there exist positive constants $\varepsilon_0=\varepsilon_0(K_0)$ and $\lambda_0=\lambda_0(K)$ such that for any 
$\varepsilon\in[0,\varepsilon_0]$ and any regular solution $(u,\psi_{\rm i})$ of \eqref{PB1reg}--\eqref{PB2reg} we have the energy estimate 
\[
I_{\lambda,t}( \opnorm{ u(\cdot) }_{m} )^2 + I_{\lambda,t}( \abs{\dtan J_\varepsilon\psi_{\rm i}(\cdot) }_{H_{(m)}^{-1/2}} )^2
 + \varepsilon\abs{u}_{H_\lambda^m((0,t)\times\mathit{\Gamma})}^2 \leq C(K_0)S_{{\rm data},\eps}^m(t)
\]
for any $\lambda\geq\lambda_0(K)$ and $t\in[0,T]$, where
\begin{align*}
S_{{\rm data},\eps}^m(t)
&= \opnorm{ u(0) }_{m}^2 + \abs{\dtan \psi_{\rm i}(0) }_{H_{(m)}^{-1/2}}^2 + \abs{ g_1(0) }_{H_{(m)}^{-1/2}}^2 \\
&\quad\;
 + S_{\lambda,t}^*( \opnorm{ f^\varepsilon(\cdot) }_{m} )^2 + S_{\lambda,t}^*( \abs{ g_1(\cdot) }_{H_{(m+1)}^{-1/2}} )^2
 + \lambda^{-1}\abs{ \dtan g_2 }_{L_{\lambda,t}^2H_{(m)}^{-1/2}}^2 + \varepsilon \abs{ g_1 }_{H_\lambda^m((0,t)\times\mathit{\Gamma})}^2.
\end{align*}
\end{proposition}

\begin{proof}
We can prove the proposition along with the same lines as in Sections \ref{sectWSL2} and \ref{sectHOest} with slight modifications, 
so that we focus the points where we need the modifications, which will be given by the following lemmas.

\begin{lemma}\label{lem:UEE1}
Under the assumptions of Proposition \ref{prop:UEE} and with the same notations, there exists a positive constant $\lambda_0=\lambda_0(K)$ 
such that for any $\varepsilon\in[0,1]$ and any regular solution $(u,\psi_{\rm i})$ of \eqref{PB1reg}--\eqref{PB2reg} with $u$ 
supported in $\overline{\cE}\cap U_\mathit{\Gamma}$ we have 
\begin{align*}
I_{\lambda,t}( \opnorm{ u(\cdot) }_{m,\parallel} )^2 
&+ I_{\lambda,t}( \abs{\dtan J_\varepsilon\psi_{\rm i}(\cdot) }_{H_{(m)}^{-1/2}} )^2
 + \varepsilon\abs{u}_{H_\lambda^m((0,t)\times\mathit{\Gamma})}^2 \\
&\leq C(K_0)S_{\rm data}^m(t) + C(K)\lambda^{-2}I_{\lambda,t}( \opnorm{ u(\cdot) }_m )^2
\end{align*}
for any $\lambda\geq\lambda_0(K)$ and $t\in[0,T]$. 
\end{lemma}

\begin{proof}[Proof of Lemma \ref{lem:UEE1}]
We prove the lemma only in the case $m=0$, 
because the case $m\geq1$ can be proved along with the same line as in the proof of Proposition \ref{propWSweakdissm}. 
For the regularized system \eqref{PB1reg}, the boundary quadratic form $\mathfrak{B}[u]$ is given by 
$\mathfrak{B}[u] = 2u^{\rm I}(N\cdot u^{\rm II}) -\varepsilon Su\cdot u$. 
Therefore, by the boundary conditions \eqref{PB2reg} and noting that the smoothing operator $J_\varepsilon$ is symmetric in $L^2(\mathit{\Gamma})$ we have 
\begin{align*}
&\int_0^t e^{-2\lambda t'} \left( \int_\mathit{\Gamma}\mathfrak{B}[u] \right){\rm d}t' \\
&= -2\int_0^t e^{-2\lambda t'} ( \dt J_\varepsilon\psi_{\rm i},\Lambda J_\varepsilon \psi_{\rm i} )_{L^2(\mathit{\Gamma})}{\rm d}t' 
 + 2\int_0^t e^{-2\lambda t'} ( \Lambda J_\varepsilon\psi_{\rm i},J_\varepsilon g_2 )_{L^2(\mathit{\Gamma})}{\rm d}t'\\
&\quad\;
 + 2\int_0^t e^{-2\lambda t'} ( g_1, u^{\rm I} )_{L^2(\mathit{\Gamma})}{\rm d}t'
 - 2\varepsilon\int_0^t e^{-2\lambda t'} (Su, u)_{L^2(\mathit{\Gamma})}{\rm d}t' \\
&=: B_1+B_2+B_3+B_4.
\end{align*}
Here, as in the proof of Proposition \ref{propWSweakdiss} we obtain 
\begin{align*}
B_1 &\leq -\underline{\mathfrak c}\bigl( e^{-2\lambda t} \abs{\dtan J_\varepsilon\psi_{\rm i}(t)}_{H^{-1/2}(\mathit{\Gamma})}^2
 + 2\lambda\abs{\dtan J_\varepsilon\psi_{\rm i}}^2_{L^2_{\lambda,t}H^{-1/2}(\mathit{\Gamma})} \bigr)
 + \underline{\mathfrak C}\abs{\dtan \psi_{\rm i}(0)}_{H^{-1/2}(\mathit{\Gamma})}^2, \\
\abs{ B_2 } &\leq \underline{\mathfrak c}\lambda\abs{\dtan J_\varepsilon\psi_{\rm i}}^2_{L^2_{\lambda,t}H^{-1/2}(\mathit{\Gamma})}
 + (\underline{\mathfrak c}\lambda)^{-1}\underline{\mathfrak C}^2\abs{\dtan g_2}^2_{L^2_{\lambda,t}H^{-1/2}(\mathit{\Gamma})},
\end{align*}
where we used the inequality $\abs{ J_\varepsilon \psi }_{H^s(\mathit{\Gamma})} \leq \abs{ \psi }_{H^s(\mathit{\Gamma})}$. 
As for $B_3$, we put $\varphi=(0,\chi_{\rm b}\tilde{g}_1N^{\rm T})^{\rm T}$ with $\tilde{g}_1$ an extension of $g_1$ 
to $(0,T)\times\cE$ so that $\tilde{g}_{1\vert_\mathit{\Gamma}}=g_1$ as before. 
Then, we obtain 
\[g_1u^{\rm I} = N_j(\varphi\cdot (G_j-\varepsilon \chi_{\rm b}N_jS)u) + \varepsilon\varphi\cdot Su,\] 
which together with the hyperbolic system \eqref{PB1reg} implies 
\begin{align*}
(g_1,u^{\rm I})_{L^2(\mathit{\Gamma})}
&= \frac{\rm d}{{\rm d}t}\int_\cE \varphi\cdot Su
 + \varepsilon\int_\mathit{\Gamma}\varphi\cdot Su \\
&\quad\;
 - \int_\cE \{ (S(\dt-\varepsilon\chi_{\rm b}N_j\partial_j)\varphi+G_j\partial_j\varphi )\cdot u
  + \varphi\cdot( f^\varepsilon + (\dt S - \varepsilon\partial_j(\chi_{\rm b}N_jS) )u )\}.
\end{align*}
Therefore, proceeding as in the proof of Proposition \ref{propWSweakdiss} we get 
\begin{align*}
\abs{B_3} 
&\leq \left(\frac{\alpha_0}{8}+\frac{K}{\lambda^2}\right) I_{\lambda,t}(\Abs{ u(\cdot) }_{L^2})^2
 + \frac{\alpha_0}{2}\varepsilon \abs{u}_{L_\lambda^2((0,t)\times\mathit{\Gamma})}^2 \\
&\quad\;
 + C(K_0)\bigl( \Abs{ \tilde{g}_1(0) }_{L^2}^2 + S_{\lambda,t}^*(\opnorm{\tilde{g}_1(\cdot)}_1)^2 + S_{\lambda,t}^*(\Abs{ f^\varepsilon(\cdot) }_{L^2})^2
  + \varepsilon \abs{g_1}_{L_\lambda^2((0,t)\times\mathit{\Gamma})}^2 \bigr),
\end{align*}
where we used the Sobolev embedding $W^{1,p}(\cE) \hookrightarrow L^\infty(\cE)$ to obtain 
$\Abs{ \partial S }_{L^\infty((0,T)\times\cE)} \lesssim \Abs{ \partial S }_{\mathbb{W}_T^{1,p}}$. 
Finally, $B_4$ is easily evaluated as $B_4 \leq -\alpha_0\varepsilon\abs{u}_{L_\lambda^2((0,t)\times\mathit{\Gamma})}^2$. 
Gathering the above estimates, we can deduce the desired estimate. 
\end{proof}

\begin{lemma}\label{lem:UEE2}
Under the assumptions of Proposition \ref{prop:UEE} and with the same notations, there exist positive constants $\varepsilon_0=\varepsilon_0(K_0)$ 
and $\lambda_0=\lambda_0(K)$ such that for any $\varepsilon\in[0,\varepsilon_0]$ and any regular solution $(u,\psi_{\rm i})$ of 
\eqref{PB1reg}--\eqref{PB2reg} with $u$ supported in $\overline{\cE}\cap U_\mathit{\Gamma}$ we have 
\[
I_{\lambda,t}( \opnorm{ u(\cdot) }_{m} )^2 
\leq C(K_0)\bigl( I_{\lambda,t}( \opnorm{ u(\cdot) }_{m,\parallel} )^2 + \opnorm{ u(0) }_m + S_{\lambda,t}^*( \opnorm{ f^\varepsilon(\cdot) }_m )^2 \bigr)
\]
for any $\lambda\geq\lambda_0(K)$ and $t\in[0,T]$. 
\end{lemma}

\begin{proof}[Proof of Lemma \ref{lem:UEE2}]
The proof can be carried out in the same way as the proof of Proposition \ref{propfarorderm2} and Lemma \ref{lemmapf4}, 
that is, we apply Lemma \ref{lemmvor} to $u^{(j,k,l)}=\dt^j\dtan^k\dnor^l u$ for $j+k+l\leq m-1$. 
To this end, we have to check that the coefficient matrices satisfy Assumptions \ref{asschar} and \ref{ass:regLQW}. 
In this case, we choose $\widetilde{A}_0={\rm Id}_{3\times3}$ so that $A_j=\Sigma G_j-\varepsilon\chi_{\rm b}N_j$ for $j=1,2$, 
where $\Sigma(t,x)=S(t,x)^{-1}$. 
The corresponding boundary matrix is given by $A_{\rm nor}=\Sigma G_{\rm nor}-\varepsilon{\rm Id}_{3\times3}$. 
To calculate eigenvalues of this matrix, we write 
\[
\Sigma = 
\begin{pmatrix}
 \sigma_0 & \bm{\sigma}^{\rm T} \\
 \bm{\sigma} & \Sigma'
\end{pmatrix}.
\]
Then, the eigenvalues of $A_{\rm nor}$ are given by $\lambda_1=-\varepsilon$, 
$\lambda_2=N\cdot\bm{\sigma}+\sqrt{\sigma_0N\cdot\Sigma'N}-\varepsilon$, 
and $\lambda_3=N\cdot\bm{\sigma}-\sqrt{\sigma_0N\cdot\Sigma'N}-\varepsilon$. 
By (i) in Assumption \ref{ass:S}, we have $\beta_0^{-1}{\rm Id}_{3\times3} \leq \Sigma(t,x) \leq \alpha_0^{-1}{\rm Id}_{3\times3}$, so that 
\[
\beta_0^{-1}{\rm Id}_{2\times2} \leq 
\begin{pmatrix}
 \sigma_0 & N\cdot\bm{\sigma} \\
 N\cdot\bm{\sigma} & N\cdot\Sigma'N
\end{pmatrix}
 \leq \alpha_0^{-1}{\rm Id}_{2\times2}.
\]
Therefore, in the case $\varepsilon=0$ we have $|\lambda_2\lambda_3|\geq\beta_0^{-2}$ so that 
$|\lambda_2^{-1}|,|\lambda_3^{-1}| \leq 2\beta_0^2\alpha_0^{-1}$. 
Hence, $\lambda_2$ and $\lambda_3$ are positive and negative definite, respectively, for $|\varepsilon|\leq\varepsilon_0(K_0)$. 
Although these eigenvalues have definite signs for $0<\varepsilon\leq\varepsilon_0(K_0)$, 
we have to choose $n_1=1$ and $n_2=2$ to obtain uniform estimates with respect to the regularizing parameter $\varepsilon$. 
The corresponding left eigenvectors are given by 
\[
\bm{l}_1 = S\begin{pmatrix} 0 \\ N^\perp \end{pmatrix}, \quad
\bm{l}_2 = S\begin{pmatrix} \sqrt{\sigma_0} \\ \sqrt{N\cdot\Sigma'N}N \end{pmatrix}, \quad
\bm{l}_2 = S\begin{pmatrix} -\sqrt{\sigma_0} \\ \sqrt{N\cdot\Sigma'N}N \end{pmatrix}. 
\]
It is easy to see that the condition (ii) in Assumption \ref{asschar} is satisfied with 
$\bm{q}_{1,1}=S(0,{\bf e}_2^{\rm T})^{\rm T}$, $\bm{q}_{1,2}=-S(0,{\bf e}_1^{\rm T})^{\rm T}$, and $w_1=-\varepsilon\chi_{\rm b}N$. 
Moreover, we have $\det L=2(\det S)\sqrt{\sigma_0N\cdot\Sigma'N} \geq 2\alpha_0^3\beta_0^{-1}$. 
Therefore, we can apply Lemma 3 to obtain the desired estimate. 
\end{proof}

Lemmas \ref{lem:UEE1} and \ref{lem:UEE2} are sufficient to conclude the proof of the proposition. 
\end{proof}

\subsection{Compatibility conditions}\label{sect:CC}
It is classical for hyperbolic initial boundary value problems that smooth solutions can exist only if 
the initial and boundary data satisfy an appropriate number of so-called compatibility conditions. 
In this section, we derive these compatibility conditions for the initial boundary value problem \eqref{PB1lin}--\eqref{IC2} 
as well as for its regularized version \eqref{PB1reg}--\eqref{ICreg}. 
It turns out that these compatibility conditions are different, 
so that data that are compatible with the original problem are not necessarily compatible with the regularized one; 
we show however that it is possible to approximate them by data that are compatible with the regularized problem. 
We also show that compatible data for \eqref{PB1lin}--\eqref{IC2} can be approximated by more regular data 
satisfying higher order compatibility conditions.

\subsubsection{Compatibility conditions for the  original problem}
We first consider the initial boundary value problem \eqref{PB1lin}--\eqref{IC2}. 
Let $(u,\psi_{\rm i})$ be a smooth solution to \eqref{PB1lin}--\eqref{IC2} and put 
$(u_j^{\rm in},\psi_{{\rm i},j}^{\rm in})=(\dt^ju,\dt^j\psi_{\rm i})_{\vert_{t=0}}$ for $j=0,1,2,\ldots$. 
By applying $\dt^j$ to \eqref{PB1lin} and to the second condition in \eqref{PB2lin} 
and putting $t=0$ we see that the $\{u_j^{\rm in}\}$ are calculated inductively by 
\[
S_{\vert_{t=0}}u_{j+1}^{\rm in} + \sum_{k=0}^{j-1}\binom{j}{k}(\dt^{j-k}S)_{\vert_{t=0}} u_{k+1}^{\rm in}
 + G_p\partial_p u_j^{\rm in} = (\dt^j f)_{\vert_{t=0}}
\]
for $j=0,1,2,\ldots$, and that $\{\psi_{{\rm i},j}^{\rm in}\}$ are given by 
$\psi_{{\rm i},j+1}^{\rm in} = (\dt^jg_2)_{\vert_{t=0}}-{(u_j^{\rm in})^{\rm I}}_{\vert_\mathit{\Gamma}}$ for $j=0,1,2,\ldots$. 
Then, by applying $\dt^j$ to the first condition in \eqref{PB2lin} and putting $t=0$ we obtain 
\begin{equation}\label{CC1}
N\cdot (u_j^{\rm in})^{\rm II} = \Lambda\psi_{{\rm i},j}^{\rm in} + (\dt^jg_1)_{\vert_{t=0}}
 \quad\mbox{on}\quad \mathit{\Gamma}
\end{equation}
for $j=0,1,2,\ldots$. 
These are necessary conditions that the data $(u^{\rm in},\psi_{\rm i}^{\rm in},f,g_1,g_2)$ should satisfy 
for the existence of a regular solution to the problem \eqref{PB1lin}--\eqref{IC2}. 
We check easily that under assumptions 
\begin{equation}\label{Data}
u^{\rm in} \in H^m(\cE), \quad \psi_{\rm i}^{\rm in} \in H^{m+1/2}(\mathit{\Gamma}), \quad
 f\in\mathbb{W}_T^{m-1}, \quad g_1,g_2\in\mathbb{W}_{{\rm b},T}^{m-1+1/2},
\end{equation}
together with Assumption \ref{ass:S} (i) and $\partial S\in \mathbb{W}_T^{m-2} \cap \mathbb{W}_T^{0,p}$ for some $p\in(2,\infty)$ in the case $m\geq2$, 
we have $u_j^{\rm in} \in H^{m-j}(\cE)$ and $\psi_{\rm i}^{\rm in} \in H^{m-j+1/2}(\mathit{\Gamma})$ for 
$j=0,1,\ldots,m$, so that each term in \eqref{CC1} belongs to $H^{m-1-j+1/2}(\mathit{\Gamma})$ for $j=0,1,\ldots,m-1$. 
Therefore, the condition \eqref{CC1} make sense for $j=0,1,\ldots,m-1$.

\begin{definition}\label{def:CC}
Let $m\geq1$ be an integer. 
For any integer $j\in\{0,1,\ldots,m-1\}$ we say that the data $(u^{\rm in},\psi_{\rm i}^{\rm in},f,g_1,g_2)$ satisfying \eqref{Data} 
for the initial boundary value problem \eqref{PB1lin}--\eqref{IC2} satisfy the compatibility condition at order $j$ if \eqref{CC1} holds. 
\end{definition}

\subsubsection{Compatibility conditions for the regularized problem}
We proceed to consider the initial boundary value problem \eqref{PB1reg}--\eqref{ICreg} for the regularized system with the regularized parameter 
$\varepsilon\in(0,1]$. 
Let $(u^\varepsilon,\psi_{\rm i}^\varepsilon)$ be a smooth solution to \eqref{PB1reg}--\eqref{ICreg} and put 
$(u_j^{{\rm in},\varepsilon},\psi_{{\rm i},j}^{{\rm in},\varepsilon})
 =(\dt^ju^\varepsilon,\dt^j\psi_{\rm i}^\varepsilon)_{\vert_{t=0}}$ for $j=0,1,2,\ldots$. 
We recall that $u^{(0),\varepsilon}$ will be constructed a priori from the initial data so that 
$(\dt^ju^{(0),\varepsilon})_{\vert_{t=0}} = (\dt^ju^\varepsilon)_{\vert_{t=0}}$ holds for $j=0,1,2,\ldots$. 
Therefore, by applying $\dt^j$ to \eqref{PB1reg} and to the second condition in \eqref{PB2reg} 
and putting $t=0$ we see that $\{u_j^{{\rm in},\varepsilon}\}$ are calculated inductively by 
\[
S_{\vert_{t=0}}u_{j+1}^{{\rm in},\varepsilon} + \sum_{k=0}^{j-1}\binom{j}{k}(\dt^{j-k}S)_{\vert_{t=0}} u_{k+1}^{{\rm in},\varepsilon}
 + G_p\partial_pu_j^{{\rm in},\varepsilon} = (\dt^j f)_{\vert_{t=0}}
\]
for $j=0,1,2,\ldots$, and that $\{\psi_{{\rm i},j}^{{\rm in},\varepsilon}\}$ are given by 
$\psi_{{\rm i},j+1}^{{\rm in},\varepsilon} = (\dt^jg_2)_{\vert_{t=0}}-{(u_j^{{\rm in},\varepsilon})^{\rm I}}_{\vert_\mathit{\Gamma}}$ 
for $j=0,1,2,\ldots$. 
Here, we note that this recursion formula is exactly the same as that of $\{u_j^{\rm in}\}$; 
the only difference is the initial data $u_0^{\rm in}=u^{\rm in}$ and $u_0^{{\rm in},\varepsilon}=u^{{\rm in},\varepsilon}$. 
Then, by applying $\dt^j$ to the first condition in \eqref{PB2reg} and putting $t=0$ we obtain 
\begin{equation}\label{CC2}
N\cdot (u_j^{{\rm in},\varepsilon})^{\rm II} = J_\varepsilon\Lambda J_\varepsilon\psi_{{\rm i},j}^{{\rm in},\varepsilon} + (\dt^jg_1)_{\vert_{t=0}}
 \quad\mbox{on}\quad \mathit{\Gamma}
\end{equation}
for $j=0,1,2,\ldots$. 
As in the case $\varepsilon=0$, this condition makes sense for $j=0,1,\ldots,m-1$ under the assumptions
\begin{equation}\label{Data2}
u^{{\rm in},\varepsilon} \in H^m(\cE), \quad \psi_{\rm i}^{\rm in} \in H^{m+1/2}(\mathit{\Gamma}), \quad
 f\in\mathbb{W}_T^{m-1}, \quad g_1,g_2\in\mathbb{W}_{{\rm b},T}^{m-1+1/2},
\end{equation}
together with the same condition on $S$ as before.

\begin{definition}\label{def:CC2}
Let $m\geq1$ be an integer. 
For any integer $j\in\{0,1,\ldots,m-1\}$ we say that the data $(u^{{\rm in},\varepsilon},\psi_{\rm i}^{\rm in},f,g_1,g_2)$ satisfying \eqref{Data2} 
for the initial boundary value problem \eqref{PB1reg}--\eqref{ICreg} satisfy the compatibility condition at order $j$ if \eqref{CC2} holds. 
\end{definition}

\subsubsection{Approximation of the data by compatible data to the regularized problem}
Even if the data $(u^{\rm in},\psi_{\rm i}^{\rm in},f,g_1,g_2)$ satisfy the compatibility conditions for the initial boundary value 
problem \eqref{PB1lin}--\eqref{IC2}, the same data does not necessarily satisfy the compatibility conditions for the initial boundary value 
problem \eqref{PB1reg}--\eqref{ICreg} to the regularized system. 
However, the following proposition ensures that we can approximate the initial data $u^{\rm in}$ by 
$u^{{\rm in},\varepsilon}$ so that the data $(u^{{\rm in},\varepsilon},\psi_{\rm i}^{\rm in},f,g_1,g_2)$ satisfy 
the compatibility conditions for the initial boundary value problem \eqref{PB1reg}--\eqref{ICreg} to the regularized system, and that 
$u^{{\rm in},\varepsilon} \to u^{\rm in}$ as $\varepsilon\to+0$.

\begin{proposition}\label{prop:appData}
Let $m\geq1$ be an integer and suppose that the coefficient matrix $S$ satisfies Assumption \ref{ass:S} (i) and 
$\partial S\in \mathbb{W}_T^{m-2} \cap \mathbb{W}_T^{0,p}$ for some $p\in(2,\infty)$ in the case $m\geq2$, and that Assumption \ref{ass:hi} is satisfied. 
Assume also that the data $(u^{\rm in},\psi_{\rm i}^{\rm in},f,g_1,g_2)$ satisfy \eqref{Data} and the compatibility conditions 
for the problem \eqref{PB1lin}--\eqref{IC2} up to order $m-1$. 
Then, for any $\varepsilon\in(0,1]$ there exists an initial data $u^{{\rm in},\varepsilon}\in H^m(\cE)$ such that 
the data $(u^{{\rm in},\varepsilon},\psi_{\rm i}^{\rm in},f,g_1,g_2)$ satisfy the compatibility conditions for the regularized problem 
\eqref{PB1reg}--\eqref{ICreg} up to order $m-1$. 
Moreover, we have 
\begin{equation}\label{limAppData}
\lim_{\varepsilon\to+0}\sum_{j=0}^m \Abs{ u_j^{{\rm in},\varepsilon} - u_j^{\rm in} }_{H^{m-j}(\cE)} = 0.
\end{equation}
\end{proposition}

Since the boundary is characteristic for the symmetric hyperbolic system \eqref{PB1lin}, a standard technique cannot directly be applicable 
to prove this type of proposition and we need a more delicate analysis than the standard one. 
Before giving a proof of this proposition, we state the following algebraic lemma whose proof is straightforward and thus omitted.

\begin{lemma}\label{lemmacompalt}
Let $W,F\in  {\mathbb R^3}$ and $G_{\rm nor}=N_1 G_1+N_2G_2$, 
where $G_1$ and $g_2$ are matrices defined as in \eqref{defGj} and $N=(N_1,N_2)^{\rm T}\in\R^2$ is a unit vector. 
Then, the following two assertions are equivalent:
\begin{enumerate}
\item[\rm (i)]
The vectors $W$ and $F$ solve the algebraic equation $G_{\rm nor}W=F$;
\item[\rm (ii)]
We have $F\cdot\begin{psmallmatrix} 0 \\ N^\perp \end{psmallmatrix}=0$ and there exists $\alpha\in {\mathbb R}$ such that 
$W=G_{\rm nor} F+\alpha \begin{psmallmatrix} 0 \\ N^\perp \end{psmallmatrix}$.
\end{enumerate}
If, in addition, $S_0$ is a definite positive matrix and $\widetilde{f}\in {\mathbb R}$, then the following two assertions are equivalent: 
\begin{enumerate}
\item[\rm (iii)]
The triplet $(W,F,\widetilde{f})$ solves the algebraic equations 
$G_{\rm nor}W=F$ and $S_0 \begin{psmallmatrix} 0 \\ N^\perp \end{psmallmatrix}\cdot W=\widetilde{f}$;
\item[\rm (iv)]
We have $F\cdot\begin{psmallmatrix} 0 \\ N^\perp \end{psmallmatrix}=0$ and 
$W=G_{\rm nor} F+\alpha \begin{psmallmatrix} 0 \\ N^\perp \end{psmallmatrix}$, with $\alpha$ given by 
\[
\alpha=\frac{1}{S_0 \begin{psmallmatrix} 0 \\ N^\perp \end{psmallmatrix} \cdot \begin{psmallmatrix} 0 \\ N^\perp \end{psmallmatrix}}
\bigl( \widetilde{f} -S_0\begin{psmallmatrix} 0 \\ N^\perp \end{psmallmatrix}\cdot G_{\rm nor} F \bigr).
\]
\end{enumerate}
\end{lemma}

\begin{proof}[Proof of Proposition \ref{prop:appData}]
We construct such an approximate data $u^{{\rm in},\varepsilon}$ in the form 
$u^{{\rm in},\varepsilon}=u^{\rm in}+w^{{\rm in},\varepsilon}$ and put 
$w_j^{{\rm in},\varepsilon}:=u_j^{{\rm in},\varepsilon}-u_j^{\rm in}$ and 
$\phi_{{\rm i},j}^{{\rm in},\varepsilon}:=\psi_{{\rm i},j}^{{\rm in},\varepsilon}-\psi_{{\rm i},j}^{\rm in}$ for $j=0,1,2,\ldots,m$. 
Then, we see that $\{w_j^{{\rm in},\varepsilon}\}_{j=0}^m$ are determined inductively from $w_0^{{\rm in},\varepsilon}=w^{{\rm in},\varepsilon}$ by 
\begin{equation}\label{wjin}
S_{\vert_{t=0}}w_{j+1}^{{\rm in},\varepsilon} + \sum_{k=0}^{j-1}\binom{j}{k}(\dt^{j-k}S)_{\vert_{t=0}} w_{k+1}^{{\rm in},\varepsilon}
 + G_p\partial_pw_j^{{\rm in},\varepsilon} = 0
\end{equation}
and that $\{\phi_{{\rm i},j}^{{\rm in},\varepsilon}\}_{j=0}^m$ are given by $\phi_{{\rm i},0}^{{\rm in},\varepsilon}=0$ and 
$\phi_{{\rm i},j+1}^{{\rm in},\varepsilon}=-{(w_j^{{\rm in},\varepsilon})^{\rm I}}_{\vert_\mathit{\Gamma}}$ for $j=0,1,\ldots,m-1$. 
Moreover, the compatibility condition \eqref{CC2} at order $j$ to the regularized system is reduced to 
\begin{equation}\label{rRjCC}
N\cdot (w_j^{{\rm in},\varepsilon})^{\rm II}
= J_\varepsilon\Lambda J_\varepsilon\phi_{{\rm i},j}^{{\rm in},\varepsilon}
 + (J_\varepsilon\Lambda J_\varepsilon - \Lambda)\psi_{{\rm i},j}^{\rm in}
 \quad\mbox{on}\quad \mathit{\Gamma}.
\end{equation}
We are going to construct the initial data $w_0^{{\rm in},\varepsilon}$ such that these compatibility conditions are satisfied up to order $m-1$.

\medskip
\noindent
{\bf Step 1.} Compatibility at order $j=0$. 
The compatibility condition at order $0$ can be written as 
$N\cdot(w_0^{{\rm in},\varepsilon})^{\rm II}=g_0^\eps$ with $g_0^\eps
= (J_\varepsilon\Lambda J_\varepsilon - \Lambda)\psi_{{\rm i},0}^{\rm in}$ on $\mathit{\Gamma}$. 
There are many choices for such $w_0^{{\rm in},\varepsilon}$. 
We can for instance impose additionally 
\begin{equation}\label{addcond0}
\begin{cases}
 (w_0^{{\rm in},\varepsilon})^{\rm I}=0 &\mbox{on}\quad \mathit{\Gamma}, \\
 \begin{psmallmatrix} 0 \\ N^\perp \end{psmallmatrix}\cdot w_0^{{\rm in},\varepsilon}=0 &\mbox{on}\quad \mathit{\Gamma},
\end{cases}
\end{equation}
so that $w_0^{{\rm in},\varepsilon}$ solves the equations $G_{\rm nor}w_0^{{\rm in},\varepsilon}=F$ and 
$S_0\begin{psmallmatrix} 0 \\ N^\perp \end{psmallmatrix}\cdot w_0^{{\rm in},\eps}=\widetilde{f}$, 
with $F=(g_0^\eps,0,0)^{\rm T}$, $S_0={\rm Id}_{3\times3}$ and $\widetilde{f}=0$. 
It follows that $w_0^{{\rm in},\varepsilon}$ is given by the explicit formula in (iv) of Lemma \ref{lemmacompalt}. 
In particular, since we have ${ \psi_{{\rm i},0}^{\rm in} }_{\vert_{\mathit{\Gamma}}} \in H^{m+1/2}(\mathit{\Gamma})$, 
this expressions determines ${ w_0^{{\rm in},\varepsilon} }_{\vert_{\mathit{\Gamma}}}$ uniquely as a function in $H^{m-1+1/2}(\mathit{\Gamma})$. 
Moreover, we see that 
\begin{align*}
\abs{ w_0^{{\rm in},\varepsilon} }_{H^{m-1+1/2}(\mathit{\Gamma})}
&\lesssim \abs{ (J_\varepsilon\Lambda J_\varepsilon - \Lambda)\psi_{{\rm i},0}^{\rm in} }_{H^{m-1+1/2}(\mathit{\Gamma})} \\
&\lesssim \abs{ (J_\varepsilon-\mathrm{Id})\psi_{{\rm i},0}^{\rm in} }_{H^{m+1/2}(\mathit{\Gamma})}
 + \abs{ (J_\varepsilon-\mathrm{Id})\Lambda\psi_{{\rm i},0}^{\rm in} }_{H^{m-1+1/2}(\mathit{\Gamma})} \\
&\to 0 \quad\mbox{as}\quad \varepsilon\to+0,
\end{align*}
where we used Proposition \ref{propDN}.

\medskip
\noindent
{\bf Step 2.} Analysis of the induction relation \eqref{wjin}. 
Taking the trace of \eqref{wjin} on ${\mathit \Gamma}$, we obtain the equation 
$G_{\rm nor} \dnor w_j^{{\rm in},\varepsilon} =-S_{\vert_{t=0}}w_{j+1}^{{\rm in},\varepsilon}+f_{j,0}^\eps$ with 
\[
f_{j,0}^\eps = - \sum_{k=0}^{j-1}\binom{j}{k}(\dt^{j-k}S)_{\vert_{t=0}} w_{k+1}^{{\rm in},\varepsilon}
 - G_{\rm tan}\dtan w_j^{{\rm in},\varepsilon} \quad\mbox{ on }\quad {\mathit \Gamma},
\]
where $G_{\rm tan}=\frac{1}{\abs{T}^2}(T_1G_1+T_2G_2)$. 
More generally, if $p+q=j$, then applying $\dnor^q$ to \eqref{wjin}$_{j=p}$  and taking the trace on ${\mathit \Gamma}$ yields the equation 
$G_{\rm nor} \dnor^{q+1} w_p^{{\rm in},\varepsilon} =-S_{\vert_{t=0}}\dnor^q w_{p+1}^{{\rm in},\varepsilon} +  f_{p,q}^\eps$ with 
\begin{align*}
f_{p,q}^\varepsilon
&= - [\dnor^q,S_{\vert_{t=0}}]w_{p+1}^{{\rm in},\varepsilon} + [\dnor^q,G_\mathrm{nor}]\dnor w_p^{{\rm in},\varepsilon} \\
&\qquad
 + \dnor^q \biggl( 
  \sum_{k=0}^{p-1}\binom{p}{k}(\dt^{p-k}S)_{\vert_{t=0}} w_{k+1}^{{\rm in},\varepsilon}
  + G_\mathrm{tan}\dtan w_p^{{\rm in},\varepsilon} \biggr)   \quad\mbox{ on }\quad {\mathit \Gamma};
\end{align*}
an important observation is the following: 
if $\{(\dnor^q w_p^{{\rm in},\varepsilon})_{\vert_{\mathit{\Gamma}}} \,|\, p+q\leq j\}$ are determined such that 
$(\dnor^q w_p^{{\rm in},\varepsilon})_{\vert_{\mathit{\Gamma}}}\to0$ in $H^{m-1-(p+q)+1/2}(\mathit{\Gamma})$ as $\varepsilon\to+0$, 
then the regularity assumed on $S_{\vert_{t=0}}$ is enough to ensure that 
$f_{p,q}^\varepsilon\to0$ in $H^{m-2-(p+q)+1/2}(\mathit{\Gamma})$ as $\varepsilon\to+0$.

Now, since $G_{\rm nor}W=F$ with $W= \dnor^{q+1} w_p^{{\rm in},\varepsilon}$ and 
$F= -S_{\vert_{t=0}}\dnor^q w_{p+1}^{{\rm in},\varepsilon} +  f_{p,q}^\eps$, we can apply the first assertion of Lemma \ref{lemmacompalt} to get 
\begin{equation}\label{systequiv0}
\begin{cases}
 S_{\vert_{t=0}}\dnor^q w_{p+1}^{{\rm in},\varepsilon}\cdot
  \begin{psmallmatrix} 0 \\ N^\perp \end{psmallmatrix} = f_{p,q}^\eps\cdot \begin{psmallmatrix} 0 \\ N^\perp \end{psmallmatrix}, \\
 \dnor^{q+1} w_p^{{\rm in},\varepsilon} = G_{\rm nor}\big (- S_{\vert_{t=0}}\dnor^qw_{p+1}^{{\rm in},\varepsilon} + f_{p,q}^\eps\big)+\alpha_{p,q}
  \begin{psmallmatrix} 0 \\ N^\perp \end{psmallmatrix},
\end{cases}
\end{equation}
for some coefficient $\alpha_{p,q}\in \R$. 
The set of all these relations obtained for all possible choices of $p$ and $q$ such that $p+q=j$ can be rewritten as 
\begin{equation}\label{systequiv1}
S_{\vert_{t=0}}w_{j+1}^{{\rm in},\varepsilon}\cdot \begin{psmallmatrix} 0 \\ N^\perp \end{psmallmatrix}
 = f_{j,0}^\eps\cdot \begin{psmallmatrix} 0 \\ N^\perp \end{psmallmatrix},
\end{equation}
which corresponds to the first equation of \eqref{systequiv0}$_{j,0}$, together with 
\begin{equation}\label{systequiv2}
\begin{cases}
 \dnor^{q} w_{j-q+1}^{{\rm in},\varepsilon} = G_{\rm nor}\big (- S_{\vert_{t=0}}\dnor^{q-1}w_{j-q+2}^{{\rm in},\varepsilon}
  + f_{j-q+1,q-1}^\eps\big)+\alpha_{j-q+1,q-1} \begin{psmallmatrix} 0 \\ N^\perp \end{psmallmatrix}, \\
 S_{\vert_{t=0}}\dnor^q w_{j-q+1}^{{\rm in},\varepsilon}\cdot \begin{psmallmatrix} 0 \\ N^\perp \end{psmallmatrix}
  = f_{j-q,q}^\eps\cdot \begin{psmallmatrix} 0 \\ N^\perp \end{psmallmatrix},
\end{cases}
\end{equation}
which correspond respectively to the second equation of \eqref{systequiv0}$_{j-q+1,q-1}$ and the first equation of 
\eqref{systequiv0}$_{j-q,q}$, for all $1\leq q\leq j$, as well as 
\begin{equation}\label{systequiv3}
\dnor^{j+1} w_{0}^{{\rm in},\varepsilon} 
= G_{\rm nor}\bigl (- S_{\vert_{t=0}}\dnor^{j}w_{1}^{{\rm in},\varepsilon} + f_{0,j}^\eps \bigr)
 + \alpha_{0,j}  \begin{psmallmatrix} 0 \\ N^\perp \end{psmallmatrix},
\end{equation}
which corresponds to the second equation of \eqref{systequiv0}$_{0,j}$.

\medskip
\noindent
{\bf Step 3.} Compatibility condition at order $j+1$. 
The compatibility condition at order $j+1$ can be written as 
\begin{equation}\label{j+1CC}
N\cdot(w_{j+1}^{{\rm in},\varepsilon})^{\rm II} = g_{j+1}^\varepsilon \quad\mbox{on}\quad \mathit{\Gamma},
\end{equation}
where $g_{j+1}^\varepsilon = J_\varepsilon\Lambda J_\varepsilon\phi_{{\rm i},j+1}^{{\rm in},\varepsilon}
 + (J_\varepsilon\Lambda J_\varepsilon - \Lambda)\psi_{{\rm i},j+1}^{\rm in}$. 
As there is no condition on $(w_{j+1}^{{\rm in},\varepsilon})^{\rm I}$, we are free to set it for instance equal to zero. 
This together with \eqref{systequiv1} yields 
\begin{equation}\label{systequiv1bis}
\begin{cases}
 w_{j+1}^{{\rm in},\varepsilon}=g_{j+1}^\eps \begin{psmallmatrix} 0 \\ N \end{psmallmatrix}
  + \alpha_{j+1} \begin{psmallmatrix} 0 \\ N^\perp \end{psmallmatrix}, \\
 S_{\vert_{t=0}} w_{j+1}^{{\rm in},\varepsilon}\cdot \begin{psmallmatrix} 0 \\ N^\perp \end{psmallmatrix}
  =f_{j,0}^\eps\cdot \begin{psmallmatrix} 0 \\ N^\perp \end{psmallmatrix}
\end{cases}
\end{equation}
for some $\alpha_{j+1}\in \R$ which is fully determined by Lemma \ref{lemmacompalt} by taking this time $S_0=S_{\vert_{t=0}}$. 
In a similar way, we are free to complement \eqref{systequiv3} by an addition condition on 
$S_{\vert_{t=0}}\begin{psmallmatrix} 0 \\ N^\perp \end{psmallmatrix}\cdot\dnor^{j+1}w_0^{{\rm in},\varepsilon}$; 
for the sake of simplicity, we set this quantity equal to zero, so that we can replace \eqref{systequiv3} by 
\begin{equation}\label{systequiv3bis}
\begin{cases}
 \dnor^{j+1} w_{0}^{{\rm in},\varepsilon}=G_{\rm nor}\big (- S_{\vert_{t=0}}\dnor^{j}w_{1}^{{\rm in},\varepsilon} + f_{0,j}^\eps\big)+\alpha_{0,j}
  \begin{psmallmatrix} 0 \\ N^\perp \end{psmallmatrix}, \\
 S_{\vert_{t=0}}\begin{psmallmatrix} 0 \\ N^\perp \end{psmallmatrix}\cdot\dnor^{j+1}w_0^{{\rm in},\varepsilon}=0,
\end{cases}
\end{equation}
which, owing to Lemma \ref{lemmacompalt}, determines the coefficient $\alpha_{0,j}$.

\medskip
\noindent
{\bf Step 4.} Induction. 
We want to show that the relations formed by \eqref{systequiv1bis}, \eqref{systequiv2}, and \eqref{systequiv3} 
allow us to construct the desired approximation. 
More precisely, we prove by induction that for all $0\leq j\leq m-1$, 
the $\{(\dnor^q w_p^{{\rm in},\varepsilon})_{\vert_{\mathit{\Gamma}}} \,|\, p+q\leq j\}$ are determined and such that 
$(\dnor^q w_p^{{\rm in},\varepsilon})_{\vert_{\mathit{\Gamma}}}\to0$ in $H^{m-1-(p+q)+1/2}(\mathit{\Gamma})$ as $\varepsilon\to+0$. 
We have proved in Step 1 that this quantity is true for $j=0$. 
Let us thus assume that it is satisfied for $0\leq j\leq m-2$ and prove that it also holds for $j+1$. 
We know from the induction assumption that $g_{j+1}^\eps$ is already determined and satisfies 
$g_{j+1}^\varepsilon \to 0$ in $H^{m-2-j+1/2}(\mathit{\Gamma})$ as $\varepsilon\to+0$. 
As already observed in Step 2, the some holds for $f_{j,0}^\eps$. 
Using the explicit expression of the solution to \eqref{systequiv1bis} furnished by the second assertion of Lemma \ref{lemmacompalt} 
with $S_0=S_{\vert_{t=0}}$, and the fact that we assumed enough regularity on $S_{\vert_{t=0}}$, 
we get the desired bound on $(w_{j+1}^{{\rm in},\varepsilon})_{\vert_{\mathit{\Gamma}}}$. 
Using \eqref{systequiv2} and \eqref{systequiv3}, we readily show that the same holds true for all the 
$(\dnor^q w_p^{{\rm in},\varepsilon})_{\vert_{\mathit{\Gamma}}}$ with $p+q\leq j=1$, so that the induction is complete.

\medskip
\noindent
{\bf Step 5.} Conclusion. 
Now, we have seen that the compatibility conditions up to order $m-1$ together with the additional conditions \eqref{addcond0} 
and the second equation of \eqref{systequiv3bis} determine $\{(\dnor^j w_0^{{\rm in},\varepsilon})_{\vert_{\mathit{\Gamma}}}\}_{j=0}^{m-1}$ 
uniquely and that $(\dnor^j w_0^{{\rm in},\varepsilon})_{\vert_{\mathit{\Gamma}}}\to0$ in $H^{m-1-j+1/2}(\mathit{\Gamma})$ as $\varepsilon\to+0$. 
Therefore, by a standard extension theorem we can construct the desired correction $w_0^{{\rm in},\varepsilon}$. 
Moreover, the limit \eqref{limAppData} follows directly from the above construction of $w_0^{{\rm in},\varepsilon}$. 
\end{proof}

\subsubsection{Approximation of the data by more regular compatible data}
In order to approximate the data $(u^{\rm in},\psi_{\rm i}^{\rm in},f,g_1,g_2)$ for the initial boundary value problem \eqref{PB1lin}--\eqref{IC2} 
by more regular data satisfying higher order compatibility conditions, as in \cite{Audiard} in the non-characteristic case, 
we use the following proposition.

\begin{proposition}\label{prop:appData2}
Let $m,s\geq1$ be integers and suppose that $S\in C_{\rm b}^\infty(\overline{(0,T)\times\cE})$ satisfies Assumption \ref{ass:S} (i) and that 
Assumption \ref{ass:hi} is satisfied. 
Assume also that the data $(u^{\rm in},\psi_{\rm i}^{\rm in},f,g_1,g_2)$ for the initial boundary value problem \eqref{PB1lin}--\eqref{IC2} 
satisfy \eqref{Data} and the compatibility conditions up to order $m-1$. 
Let $\{\psi_{\rm i}^{{\rm in}(n)}\}_{n=1}^\infty \subset H^{m+s+1/2}(\mathit{\Gamma})$, 
$\{f^{(n)}\}_{n=1}^\infty \subset \mathbb{W}_T^{m+s-1}$, $\{(g_1^{(n)},g_2^{(n)})\}_{n=1}^\infty \subset \mathbb{W}_{{\rm b},T}^{m+s-1+1/2}$ 
be approximations of the data, which converge to $(u^{\rm in},\psi_{\rm i}^{\rm in},f,g_1,g_2)$ in the class indicated in \eqref{Data} as $n\to\infty$. 
Then, there exists $\{u^{{\rm in}(n)}\}_{n=1}^\infty \subset H^{m+s}(\cE)$ satisfying 
$u^{{\rm in}(n)} \to u^{\rm in}$ in $H^{m}(\cE)$ such that the data $(u^{{\rm in}(n)},\psi_{\rm i}^{{\rm in}(n)},f^{(n)},g_1^{(n)},g_2^{(n)})$ 
satisfy the compatibility conditions up to order $m+s-1$. 
\end{proposition}

\begin{proof}
By the density of the space, we first approximate the initial data $u^{\rm in}$ by a sequence of data 
$\{\tilde{u}^{{\rm in}(n)}\}_{n=1}^\infty \subset H^{m+s}(\cE)$ such that $\tilde{u}^{{\rm in}(n)} \to u^{\rm in}$ in $H^{m}(\cE)$ as $n\to\infty$. 
In general, these data $(\tilde{u}^{{\rm in}(n)},\psi_{\rm i}^{{\rm in}(n)}, f^{(n)},g_1^{(n)},g_2^{(n)})$ do not satisfy the compatibility conditions. 
Therefore, we compensate the initial data in the form $u^{{\rm in}(n)}=\tilde{u}^{{\rm in}(n)}+w^{{\rm in}(n)}$ so that the modified data satisfy the 
compatibility conditions up to order $m+s-1$. 
Let $(\tilde{u}_j^{{\rm in}(n)},\tilde{\psi}_{{\rm i},j}^{{\rm in}(n)})$ and $(u_j^{{\rm in}(n)},\psi_{{\rm i},j}^{{\rm in}(n)})$ be initial date 
$(\dt^j u,\dt^j\psi_{\rm i})_{\vert_{t=0}}$ for the initial boundary value problem \eqref{PB1lin}--\eqref{IC2} determined from the data 
$(\tilde{u}^{{\rm in}(n)},\psi_{\rm i}^{{\rm in}(n)},f^{(n)},g_1^{(n)},g_2^{(n)})$, and 
$(u^{{\rm in}(n)},\psi_{\rm i}^{{\rm in}(n)},f^{(n)},g_1^{(n)},g_2^{(n)})$, respectively, and put 
$w_j^{{\rm in}(n)}=u_j^{{\rm in}(n)}-\tilde{u}_j^{{\rm in}(n)}$ and 
$\phi_{{\rm i},j}^{{\rm in}(n)}=\tilde{\psi}_{{\rm i},j}^{{\rm in}(n)}-\psi_{{\rm i},j}^{{\rm in}(n)}$. 
Then, we see that $\{w_j^{{\rm in}(n)}\}_{j=0}^{m+s}$ are determined inductively from $w_0^{{\rm in}(n)}=w^{{\rm in}(n)}$ by 
\[
S_{\vert_{t=0}}w_{j+1}^{{\rm in}(n)} + \sum_{k=0}^{j-1}\binom{j}{k}(\dt^{j-k}S)_{\vert_{t=0}} w_{k+1}^{{\rm in}(n)}
 + G_p\partial_pw_j^{{\rm in}(n)} = 0,
\]
and that $\{\phi_{{\rm i},j}^{{\rm in}(n)}\}_{j=0}^{m+s}$ are given by $\phi_{{\rm i},0}^{{\rm in}(n)}=0$ and 
$\phi_{{\rm i},j+1}^{{\rm in}(n)}=-{(w_j^{{\rm in}(n)})^{\rm I}}_{\vert_{\mathit{\Gamma}}}$ for $j=0,1,\ldots,m+s-1$. 
Moreover, the compatibility condition at order $j$ can be written as 
\[
N\cdot (w_j^{{\rm in}(n)})^{\rm II} = \Lambda\phi_{{\rm i},j}^{{\rm in}(n)}
 - (N\cdot (\tilde{u}_j^{{\rm in}(n)})^{\rm II} - \Lambda\tilde{\psi}_{{\rm i},j}^{{\rm in}(n)})
 \quad\mbox{on}\quad \mathit{\Gamma}.
\]
Here, we see that 
\begin{align*}
& N\cdot (\tilde{u}_j^{{\rm in}(n)})^{\rm II} - \Lambda\tilde{\psi}_{{\rm i},j}^{{\rm in}(n)} \in 
 H^{m+s-1-j+1/2}(\mathit{\Gamma}) \quad\mbox{for}\quad j=0,1,\ldots,m+s-1, \\
&\to N\cdot (u_j^{\rm in})^{\rm II} - \Lambda\psi_{{\rm i},j}^{\rm in} = 0
 \quad\mbox{in}\quad H^{m-1-j+1/2}(\mathit{\Gamma}) \quad\mbox{for}\quad j=0,1,\ldots,m-1.
\end{align*}
Therefore, as in the proof of Proposition \ref{prop:appData}, if we impose additionally the corresponding conditions to \eqref{addcond0} 
and the second equation in \eqref{systequiv3bis} for $j=0,1,\ldots,m+s-2$, 
then we can determine $\{(\dnor^j w_0^{{\rm in}(n)})_{\vert_{\mathit{\Gamma}}}\}_{j=0}^{m+s-1}$ uniquely 
such that $(\dnor^j w_0^{{\rm in}(n)})_{\vert_{\mathit{\Gamma}}} \in H^{m+s-1-j+1/2}(\mathit{\Gamma})$ for $j=0,1,\ldots,m+s-1$ and 
$(\dnor^j w_0^{{\rm in}(n)})_{\vert_{\mathit{\Gamma}}}$ $\to0$ in $H^{m-1-j+1/2}(\mathit{\Gamma})$ as $n\to\infty$ for $j=0,1,\ldots,m-1$. 
Now, by using the method in \cite{RauchMassey1974} we can construct the desired correction $w_0^{{\rm in}(n)}$. 
\end{proof}

\subsection{Existence theorem}\label{subsectexistlin}
We can now state the main result of this section, 
which provides an existence result for the initial boundary value problem \eqref{PB1lin}--\eqref{IC2}. 
Let us insist on the fact that the assumption $\partial S \in \mathbb{W}_T^{1,p}$ is weaker than the assumption $\partial S \in {\mathbb W}^2_T$ 
usually found in the literature on hyperbolic initial boundary value problems. 
This will be used in Section \ref{sect:exist2} to obtain a nonlinear existence theorem with sharp regularity.

\begin{theorem}\label{th:linexist}
Let $m\geq1$ be an integer and $T>0$, and suppose that Assumption \ref{ass:hi} is satisfied and that the coefficient matrix $S$ satisfies 
$S \in C_{\rm b}^1(\overline{(0,T)\times\cE})$, $\partial S \in \mathbb{W}_T^{m-1} \cap \mathbb{W}_T^{1,p}$ for some $p\in(2,\infty)$, 
and Assumption \ref{ass:S}. 
Then, for any data $(u^{\rm in},\psi_{\rm i}^{\rm in}, f,g_1,g_2)$ satisfying 
\begin{equation}\label{dataclass}
\begin{cases}
 u^{\rm in} \in H^m(\cE), \quad \psi_{\rm i}^{\rm in} \in H^{m+1/2}(\mathit{\Gamma}), \quad f \in H^m((0,T)\times\cE), \\
 (1+\abs{D})^{-1/2}g_1 \in H^{m+1}((0,T)\times\mathit{\Gamma}), \quad (1+\abs{D})^{1/2}g_2 \in H^{m}((0,T)\times\mathit{\Gamma}),
\end{cases}
\end{equation}
and the compatibility conditions up to order $m-1$, the initial boundary value problem \eqref{PB1lin}--\eqref{IC2} has a unique solution 
$(u,\psi_{\rm i}) \in \mathbb{W}_T^m\times\mathbb{W}_{{\rm b},T}^{m+1/2}$. 
Moreover, the solution satisfies the energy estimate obtained in Proposition \ref{prop:UEE} with $\eps=0$. 
\end{theorem}

\begin{proof}
The proof consists of three steps. 
In the first step, we prove the existence for $S$ smooth and more regular data. 
In the second step, the assumption of additional regularity on the data is removed while, in the third step, 
we remove the additional regularity assumption on $S$.

\medskip
\noindent
{\bf Step 1.}
We first consider the case where the coefficient matrix $S$ is sufficiently smooth as $S\in C_{\rm b}^\infty(\overline{(0,T)\times\cE})$ and the data 
$(u^{\rm in},\psi_{\rm i}^{\rm in},f,g_1,g_2)$ satisfy \eqref{dataclass} with $m$ replaced by $m+3$ and the compatibility conditions up to order $m+1$. 
By Proposition \ref{prop:appData}, for each $\varepsilon\in(0,1]$ there exists an initial datum 
$u^{{\rm in},\varepsilon} \in H^{m+3}(\cE)$ such that the data $(u^{{\rm in},\varepsilon},\psi_{\rm i}^{\rm in}, f,g_1,g_2)$ 
satisfy the compatibility conditions up to order $m+1$ for the initial value problem \eqref{PB1reg}--\eqref{ICreg} to the regularized system. 
Moreover, we have $u^{{\rm in},\varepsilon} \to u^{\rm in}$ in $H^{m+3}(\cE)$ as $\varepsilon\to+0$. 
In this case we have $u_j^{{\rm in},\varepsilon} \in H^{m+3-j}(\cE)$ for $j=0,1,\ldots,m+3$, so that we can construct 
$u^{(0),\varepsilon} \in \mathbb{W}_T^{m+3}$ satisfying $(\dt^j u^{(0),\varepsilon})_{\vert_{t=0}} = u_j^{{\rm in},\varepsilon}$ for $j=0,1,\ldots,m+3$ 
and $\Abs{ u^{(0),\varepsilon} }_{\mathbb{W}_T^{m+3}} \lesssim \Abs{ u^{{\rm in},\varepsilon} }_{H^{m+3}(\cE)}$. 
Then, we consider the initial boundary value problem \eqref{PB1reg}--\eqref{ICreg} to the regularized system for $\varepsilon\in(0,1]$. 
We note also that $f^\varepsilon = f - \varepsilon S(\chi_{\rm b}N\cdot\nabla)u^{(0),\varepsilon} \in \mathbb{W}_T^{m+2}$. 
Therefore, in view of Lemma \ref{lem:EVofRBM} and Remark \ref{rem:exist} we see that there exists sufficiently small $\varepsilon_0>0$ such that 
for each $\varepsilon\in(0,\varepsilon_0]$ the initial value problem \eqref{PB1reg}--\eqref{ICreg} has a unique 
solution $(u^\varepsilon,\psi_{\rm i}^\varepsilon) \in \mathbb{W}_T^{m+2} \times \mathbb{W}_{{\rm b},T}^{m+2+1/2}$. 
Furthermore, by Propositions \ref{prop:UEE} with $m$ replaced by $m+2$ we have a uniform bound of the solution 
\[
\opnorm{ u^\varepsilon(t) }_{m+2}^2 + \abs{ \dtan J_\varepsilon\psi_{\rm i}^\varepsilon(t) }_{H_{(m+2)}^{-1/2}} \leq C
\]
for any $\varepsilon\in(0,\varepsilon_0]$ and $t\in[0,T]$ with a constant $C$ independent of $\varepsilon$ and $t$.

We proceed to see the convergence of these approximated solutions 
$\{(u^\varepsilon,\psi_{\rm i}^\varepsilon)\}_{0<\varepsilon\leq\varepsilon_0}$ as $\varepsilon\to+0$. 
For $\varepsilon,\delta \in (0,\varepsilon_0]$, we put $u^{\varepsilon,\delta}=u^\varepsilon-u^\delta$ and 
$\psi_{\rm i}^{\varepsilon,\delta}=\psi_{\rm i}^\varepsilon-\psi_{\rm i}^\delta$, which satisfy the equations 
\[
S\dt u^{\varepsilon,\delta} + G_j\partial_j u^{\varepsilon,\delta} = f^{\varepsilon,\delta} \quad\mbox{in}\quad (0,T)\times\cE
\]
and boundary conditions 
\[
\begin{cases}
 N\cdot(u^{\varepsilon,\delta})^{\rm II}
  = \Lambda\psi_{\rm i}^{\varepsilon,\delta} + g_1^{\varepsilon,\delta} &\mbox{on}\quad (0,T)\times\mathit{\Gamma}, \\
 \dt\psi_{\rm i}^{\varepsilon,\delta} + (u^{\varepsilon,\delta})^{\rm I} = 0 &\mbox{on}\quad (0,T)\times\mathit{\Gamma},
\end{cases}
\]
where 
\begin{align*}
f^{\varepsilon,\delta}
&= S(\chi_{\rm b}N\cdot\nabla)\{ \varepsilon(u^\varepsilon-u^{(0),\varepsilon}) - \delta(u^\delta-u^{(0),\delta})\}, \\
g_1^{\varepsilon,\delta}
&= (J_\varepsilon-{\rm Id})\Lambda J_\varepsilon\psi_{\rm i}^\varepsilon + \Lambda(J_\varepsilon-{\rm Id})\psi_{\rm i}^\varepsilon
 - (J_\delta-{\rm Id})\Lambda J_\delta\psi_{\rm i}^\delta - \Lambda(J_\delta-{\rm Id})\psi_{\rm i}^\delta.
\end{align*}
Therefore, by Proposition \ref{prop:UEE} with $\varepsilon=0$ we have 
\begin{align*}
& \sup_{t\in[0,T]} \bigl( \opnorm{ u^{\varepsilon,\delta}(t) }_m^2 + \abs{ \dtan\psi_{\rm i}^{\varepsilon,\delta}(t) }_{H_{(m)}^{-1/2}}^2 \bigr) \\
&\lesssim \opnorm{ u^{\varepsilon,\delta}(0) }_m^2 + \abs{ \dtan\psi_{\rm i}^{\varepsilon,\delta}(0) }_{H_{(m)}^{-1/2}}^2 
 + \abs{ g_1^{\varepsilon,\delta}(0) }_{H_{(m)}^{-1/2}}^2 
 + \int_0^T \bigl( \opnorm{ f^{\varepsilon,\delta}(t) }_m^2 + \abs{ g_1^{\varepsilon,\delta}(t) }_{H_{(m+1)}^{-1/2}}^2 \bigr){\rm d}t.
\end{align*}
Here, we easily see that 
\begin{align*}
\abs{ \dtan\psi_{\rm i}^{\varepsilon,\delta}(0) }_{H_{(m)}^{-1/2}}
&\lesssim \opnorm{ u^{\varepsilon,\delta}(0) }_m \\
&\lesssim \Abs{ u^{{\rm in},\varepsilon} - u^{{\rm in},\delta} }_{H^m} + \opnorm{ f^{\varepsilon,\delta}(0) }_{m-1}
\end{align*}
and that 
\begin{align*}
\opnorm{ f^{\varepsilon,\delta}(t) }_m
&\lesssim \varepsilon\opnorm{ (u^\varepsilon,u^{(0),\varepsilon)})(t) }_{m+1} + \delta\opnorm{ (u^\delta,u^{(0),\delta)})(t) }_{m+1} \\
&\lesssim \varepsilon+\delta.
\end{align*}
In view of $\abs{ (J_\varepsilon-\mathrm{Id})\phi }_{H^s} \leq \varepsilon\abs{ J_\varepsilon\phi }_{H^{s+1}}$ 
and Proposition \ref{propDN} we also get 
\begin{align*}
\abs{ g_1^{\varepsilon,\delta}(t) }_{H_{(m+1)}^{-1/2}}
&\lesssim \varepsilon \abs{ \dtan J_\varepsilon\psi_{\rm i}^\varepsilon(t) }_{H_{(m+2)}^{-1/2}}
 + \delta \abs{ \dtan J_\delta\psi_{\rm i}^\delta(t) }_{H_{(m+2)}^{-1/2}} \\
&\lesssim \varepsilon+\delta.
\end{align*}
Moreover, in view of $\dt\psi_{\rm i}^{\varepsilon,\delta}=-(u^{\varepsilon,\delta})^{\rm I}$ and ${\psi_{\rm i}^{\varepsilon,\delta}}_{\vert_{t=0}}=0$ 
and the trace theorem we have 
\[
\abs{ \dt^j \psi_{\rm i}^{\varepsilon,\delta}(t) }_{H^{m-j+1/2}} \lesssim \Abs{ u^{\varepsilon,\delta} }_{\mathbb{W}_T^m}
 + \abs{ \dtan \psi_{\rm i}^{\varepsilon,\delta}(t) }_{H^{m-1/2}}
\]
for $j=0,1,\ldots,m$ and $t\in[0,T]$. 
Therefore, we obtain 
\[
\sup_{t\in[0,T]}\sum_{j=0}^m \bigl( \Abs{ \dt^j(u^\varepsilon-u^\delta)(t) }_{H^{m-j}}
 + \abs{ \dt^j(\psi_{\rm i}^\varepsilon-\psi_{\rm i}^\delta)(t) }_{H^{m-j+1/2}} \bigr)
\lesssim \Abs{ u^{{\rm in},\varepsilon} - u^{{\rm in},\delta} }_{H^m} + \varepsilon+\delta.
\]
This shows that the approximate solutions $\{(u^\varepsilon,\psi_{\rm i}^\varepsilon)\}_{0<\varepsilon\leq\varepsilon_0}$ 
converge as $\varepsilon\to+0$ and the limit is the desired solution.

\medskip
\noindent
{\bf Step 2.} 
We still consider the case where the coefficient matrix $S$ is of $C_{\rm b}^\infty$-class 
but we do not assume additional regularity assumption on the data $(u^{\rm in},\psi_{\rm i}^{\rm in},f,g_1,g_2)$. 
By the density of function spaces, there exists a sequence of data $\{(\psi_{\rm i}^{{\rm in}(n)},f^{(n)},g_1^{(n)},g_2^{(n)})\}_{n=1}^\infty$, 
which satisfies the regularity indicated in \eqref{dataclass} with $m$ replaced by $m+3$ and converges to $(\psi_{\rm i}^{\rm in},f,g_1,g_2)$ 
in the space indicated in \eqref{dataclass}. 
Then, by Proposition \ref{prop:appData2} there exists $\{u^{{\rm in}(n)}\}_{n=1}^\infty \subset H^{m+3}(\cE)$ 
satisfying $u^{{\rm in}(n)} \to u^{\rm in}$ in $H^m(\cE)$ such that the data 
$(u^{{\rm in}(n)},\psi_{\rm i}^{{\rm in}(n)},f^{(n)},g_1^{(n)},g_2^{(n)})$ satisfy the compatibility conditions up to order $m+1$. 
By the result in Step 1, for each $n\in\N$ there exists a unique solution $(u^{(n)},\psi_{\rm i}^{(n)})$ to the initial boundary value problem 
\eqref{PB1lin}--\eqref{IC2} for the data $(u^{{\rm in}(n)},\psi_{\rm i}^{{\rm in}(n)},f^{(n)},g_1^{(n)},g_2^{(n)})$. 
Then, by the linearity of the equations and by the energy estimate given in Proposition \ref{prop:UEE} with $\varepsilon=0$ 
we see that $\{(u^{(n)},\psi_{\rm i}^{(n)})\}_{n=1}^\infty$ converges in $\mathbb{W}_T^m \times \mathbb{W}_{{\rm b},T}^{m+1/2}$ as $n\to\infty$ 
and the limit is the desired solution.

\medskip
\noindent
{\bf Step 3.} 
Finally, we consider the case where the coefficient matrix $S$ is not necessarily of $C_{\rm b}^\infty$-class 
but satisfies the assumptions in the theorem. 
By the density of function spaces, we can approximate $S$ by a sequence of symmetric and positive definite matrix valued functions 
$\{S^{(n)}\}_{n=1}^\infty \subset C_{\rm b}^\infty(\overline{(0,T)\times\cE})$ such that 
\[
\begin{cases}
 S^{(n)} \to S \quad\mbox{in}\quad C_{\rm b}^1(\overline{(0,T)\times\cE}), \\
 \partial S^{(n)} \to \partial S \quad\mbox{in}\quad \mathbb{W}_T^{m-1} \cap \mathbb{W}_T^{1,p}
\end{cases}
\]
as $n\to\infty$. 
Then, we consider the initial boundary value problem to the hyperbolic system 
\begin{equation}\label{appPB1lin}
S^{(n)}(t,x)\dt u + G_j\partial_j u = f(t,x) \quad\mbox{in}\quad (0,T)\times\cE
\end{equation}
under the boundary conditions 
\begin{equation}\label{appPB2lin}
\begin{cases}
 N\cdot u^{\rm II} - \Lambda\psi_{\rm in} = g_1(t,x) &\mbox{on}\quad (0,T)\times\mathit{\Gamma}, \\
 \dt\psi_{\rm i} + u^{\rm I} = g_2(t,x) &\mbox{on}\quad (0,T)\times\mathit{\Gamma},
\end{cases}
\end{equation}
and the initial conditions 
\begin{equation}\label{appIC}
\begin{cases}
 u_{\vert_{t=0}}=u^{{\rm in}(n)} &\mbox{in}\quad \cE, \\
 {\psi_{\rm i}}_{\vert_{t=0}}=\psi_{\rm i}^{\rm in} &\mbox{on}\quad \mathit{\Gamma}.
\end{cases}
\end{equation}
Here, the initial data $u^{{\rm in}(n)}$ should be modified from the original data $u^{\rm in}$ in order that the data 
 $(u^{{\rm in}(n)},\psi_{\rm i}^{\rm in},f,g_1,g_2)$ satisfy the compatibility conditions up to order $m-1$. 
In exactly the same way as in the proof of Proposition \ref{prop:appData}, we can construct such an initial data $u^{{\rm in}(n)}$, 
which converges to $u^{\rm in}$ in $H^m(\cE)$ as $n\to\infty$. 
By the result in Step 2, for each $n\in\N$ there exists a unique solution 
$(u,\psi_{\rm i})=(u^{(n)},\psi_{\rm i}^{(n)}) \in \mathbb{W}_T^m \times \mathbb{W}_{{\rm b},T}^{m+1/2}$ 
to the initial boundary value problem \eqref{appPB1lin}--\eqref{appIC}. 
Moreover, by Proposition \ref{prop:UEE} we have a uniform bound of the solution 
\[
\opnorm{ u^{(n)}(t) }_m^2 + \abs{ \dtan \psi_{\rm i}^{(n)}(t) }_{H_{(m)}^{-1/2}} \leq C
\]
for any $n\in\N$ and $t\in[0,T]$ with a constant $C$ independent of $n$ and $t$.

We readily get, using Proposition \ref{prop:UEE} with $\varepsilon=0$ as in Step 1, 
that $\{(u^{(n)},\psi_{\rm i}^{(n)})\}_{n=1}^\infty$ converges in $\mathbb{W}_T^{m-1}\times\mathbb{W}_{{\rm b},T}^{m-1+1/2}$; 
the limit is denoted by $(u,\psi_{\rm i})$. 
By standard compactness arguments, for each $t\in[0,T]$ we also have 
\[
\begin{cases}
 \partial^\alpha u^{(n)}(t) \rightharpoonup \partial^\alpha u(t)
  &\mbox{weakly in}\quad L^2(\cE) \quad (\abs{\alpha}=m), \\
 \dpar^\alpha\psi_{\rm i}^{(n)}(t) \rightharpoonup \dpar^\alpha\psi_{\rm i}(t)
  &\mbox{weakly in}\quad H^{1/2}(\mathit{\Gamma}) \quad (\abs{\alpha}=m)
\end{cases}
\]
with a uniform bound 
\[
\opnorm{ \bm{u}(t) }_m + \abs{ \dtan\psi_{\rm i}(t) }_{H_{(m)}^{-1/2}} \leq C,
\]
where the constant $C$ is independent of $t\in[0,T]$. 
By a standard method we can also show 
\[
\begin{cases}
 \partial^\alpha u \in C_w^0([0,T];L^2(\cE)) &(\abs{\alpha}=m), \\
 \dpar \psi_{\rm i} \in C_w^0([0,T];H^{1/2}(\mathit{\Gamma})) &(\abs{\alpha}=m).
\end{cases}
\]
Obviously, $(u,\psi_{\rm i})$ is a unique solution to the initial boundary value problem \eqref{PB1lin}--\eqref{IC2}. 
It remains to show that the above weak continuity in time can be replaced by the strong continuity. 
To this end, we use the technique by A. J. Majda \cite{Majda1984}, see also A. J. Majda and A. L. Bertozzi \cite{MajdaBertozzi2002}, 
that is, we make use of the energy estimate. 
Here, we also note that the initial boundary value problem \eqref{PB1lin}--\eqref{IC2} can be solved for negative time; 
in fact, under the change of variable $t \to -t$ the weak dissipative structure does not change. 
More precisely, for the space $X:= L^2(\cE) \times H^{1/2}(\mathit{\Gamma})$ and for each $t\in[0,T]$ we introduce an inner product 
\[
\langle((u,\psi_{\rm i}),(\tilde{u},\tilde{\psi}_{\rm i})\rangle_t
:= (u,S(t,\cdot)\tilde{u})_{L^2(\cE)} + \langle \Lambda\psi_{\rm i},\tilde{\psi}_{\rm i}\rangle_{H^{-1/2}\times H^{1/2}}
 + (\psi_{\rm i},\tilde{\psi}_{\rm i})_{L^2(\mathit{\Gamma})},
\]
and denote the corresponding norm by $\Abs{\cdot}_t$, which is equivalent to the standard $L^2(\cE)\times H^{1/2}(\mathit{\Gamma})$-norm 
by Proposition \ref{propDN}. 
Then, by an energy identity corresponding to the one obtained in the proof of Proposition \ref{prop:UEE}, 
for any $\alpha\in\N$ satisfying $\abs{\alpha}=m$ and $t_0\in[0,T]$ we can show that 
\[
\lim_{t\to t_0} \Abs{ (\chi_{\rm b}\dpar^\alpha u(t),\dpar^\alpha\psi_{\rm i}(t)) }_t
 = \Abs{ (\chi_{\rm b}\dpar^\alpha u(t_0),\dpar^\alpha\psi_{\rm i}(t_0)) }_{t_0}. 
\]
Since we already know the weak continuity, this gives the strong continuity, that is, we have 
$\chi_{\rm b}\dpar^\alpha u \in C^0([0,T];L^2(\cE))$ and $\dpar^\alpha\psi_{\rm i} \in C^0([0,T];H^{1/2}(\mathit{\Gamma}))$, 
leading to $\psi_{\rm i} \in \mathbb{W}_{{\rm b},T}^{m+1/2}$. 
Then, as in the proof of Proposition \ref{propfarorderm2}; see also Lemma \ref{lemmvor}, 
we obtain $\chi_{\rm b}\partial^\alpha u\in C^0([0,T];L^2(\cE))$ for $\abs{ \alpha }=m$. 
Similarly and more easily, we can see $(1-\chi_{\rm b})\partial^\alpha u \in C^0([0,T];L^2(\cE))$ for $\abs{\alpha}=m$. 
Therefore, we obtain $\partial^\alpha u \in C^0([0,T];L^2(\cE))$ for $\abs{\alpha}=m$, so that $u\in\mathbb{W}_T^m$. 
The proof is complete. 
\end{proof}

\section{Local well-posedness of the nonlinear wave-structure interaction problem}\label{sectmain}
In this section we will show the local well-posedness of the nonlinear wave-structure interaction problem \eqref{PB1}--\eqref{PB3}, 
which is equivalent to \eqref{PB1bis}--\eqref{PB3bis} under appropriate initial conditions as shown in Proposition \ref{prop:equiv}. 
Therefore, we consider the nonlinear shallow water equations 
\begin{equation}\label{SWEs}
\begin{cases}
 \dt \zeta + \nabla\cdot (hv) = 0 &\mbox{in}\quad (0,T)\times\cE, \\
 \dt v + \nabla \big( \gr\zeta + \tfrac{1}{2}\abs{v}^2 \bigr) = 0 &\mbox{in}\quad (0,T)\times\cE
\end{cases}
\end{equation}
under the boundary conditions 
\begin{equation}\label{BCs}
\begin{cases}
 N\cdot (hv) = \Lambda \psi_{\rm i} &\mbox{on}\quad (0,T)\times\mathit{\Gamma}, \\
 \dt\psi_{\rm i} + \gr \zeta +\tfrac{1}{2}\abs{v}^2 = 0 &\mbox{on}\quad (0,T)\times\mathit{\Gamma}, 
\end{cases}
\end{equation}
and the initial conditions 
\begin{equation}\label{ICs}
\begin{cases}
 (\zeta,v)_{\vert_t=0}=(\zeta^{\rm in},v^{\rm in}) &\mbox{in}\quad \cE, \\
 {\psi_{\rm i}}_{\vert_t=0}=\psi_{\rm i}^{\rm in} &\mbox{on}\quad \mathit{\Gamma}.
\end{cases}
\end{equation}
For simplicity, as before we denote $u^{\rm in}=(\zeta^{\rm in},(v^{\rm in})^{\rm T})^{\rm T}$ as well as $u=(\zeta,v^{\rm T})^{\rm T}$. 

The statement of the main result of this article, showing the local well-posedness of the nonlinear model, 
is provided in Section \ref{sect:state-ET}. 
In the rest of this article, we will prove this existence theorem. 
The proof consists of several steps. 
Although the nonlinear shallow water equations \eqref{SWEs} are quasilinear with respect to $u$, 
the boundary conditions \eqref{BCs} are not quasilinear but fully nonlinear, 
so that we first need to reduce the problem to a system of quasilinear equations. 
This step can be carried out through the analysis in Section \ref{sectlinWS} by introducing new unknowns 
\begin{equation}\label{newUNs}
\widecheck{u} = \Sigma(u)\dt u, \quad \widecheck{\psi}_{\rm i} = \dt\psi_{\rm i},
\end{equation}
where $\Sigma(u)$ is the matrix defined as \eqref{defGj}. 
This reduction to a quasilinear system will be given in Section \ref{sect:reduction}. 
Then, by applying the existence theorem for a linear problem given by Theorem \ref{th:linexist} and a standard method of Picard's iteration, 
we will show in Section \ref{sect:exist} the existence of a unique solution to the reduced quasilinear system 
under additional assumptions on the initial data. 
In order to prove the theorem without the additional assumptions, we approximate the initial data $(u^{\rm in},\psi_{\rm i}^{\rm in})$ by a sequence of 
more regular initial data $\{(u^{{\rm in}(n)},\psi_{\rm i}^{{\rm in}(n)})\}_{n=1}^\infty$ satisfying higher order compatibility conditions. 
By the existence result in Section \ref{sect:exist}, we can construct a sequence of smoother approximate solutions. 
Then, thanks to the a priori estimate obtained in Theorem \ref{th:APE} together with the energy estimate for the linear system obtained in Proposition 
\ref{prop:UEE} we can show in Section \ref{sect:exist2} that the approximate solutions converge to the desired solution 
and that the solution satisfies the regularity \eqref{Reg}.

\subsection{Statement of the existence theorem}\label{sect:state-ET}
In the following, we let $m\geq3$ be an integer and assume that the initial data $(u^{\rm in},\psi_{\rm i}^{\rm in})$ satisfy 
\begin{equation}\label{regID}
u^{\rm in} \in H^m(\cE), \quad \psi_{\rm i}^{\rm in} \in H^{m+1/2}(\mathit{\Gamma}).
\end{equation}
We recall that $m=3$ is the minimal integer regularity in a general existence theory for quasilinear symmetric hyperbolic systems 
in two space dimensions. 
Suppose that $(u,\psi_{\rm i})$ is a smooth solution to the initial value problem \eqref{SWEs}--\eqref{ICs} and put 
$u_j^{\rm in} = (\zeta_j^{\rm in},(v_j^{\rm in})^{\rm T})^{\rm T} = (\dt^j u)_{\vert_{t=0}}$ and 
$\psi_{{\rm i},j}^{\rm in}=(\dt^j\psi_{\rm i})_{\vert_{t=0}}$ for $j=0,1,2,\ldots$. 
Then, it follows from the nonlinear shallow water equations \eqref{SWEs} and the second boundary condition in \eqref{BCs} that 
$(u_j^{\rm in},\psi_{{\rm i},j}^{\rm in})$ can be written, at least formally, in terms of the initial data $(u^{\rm in},\psi_{\rm i}^{\rm in})$ 
inductively as $(u_0^{\rm in},\psi_{{\rm i},0}^{\rm in})=(u^{\rm in},\psi_{\rm i}^{\rm in})$ and 
\begin{equation}\label{DefID1}
\begin{cases}
 \zeta_{j+1}^{\rm in} = -\nabla\cdot\bigl( \sum_{k=0}^j\binom{j}{k} h_{j-k}^{\rm in}v_k^{\rm in} \bigr), \\
 v_{j+1}^{\rm in} = -\nabla\bigl( \gr\zeta_j^{\rm in} + \frac12\sum_{k=0}^j\binom{j}{k}v_{j-k}^{\rm in}\cdot v_k^{\rm in}\bigr), \\
 \psi_{{\rm i},j+1}^{\rm in} 
  = -\bigl( \gr\zeta_j^{\rm in} + \frac12\sum_{k=0}^j\binom{j}{k}v_{j-k}^{\rm in}\cdot v_k^{\rm in} \bigr)_{\vert_{\mathit{\Gamma}}}
\end{cases}
\end{equation}
for $j=0,1,2,\ldots$, where $h_j^{\rm in} = (\dt^j h)_{\vert_{t=0}}$ so that $h_0^{\rm in}=H_0+\zeta^{\rm in}$ and 
$h_j^{\rm in}=\zeta_j^{\rm in}$ for $j\geq1$. 
It is easy to check that under the assumptions in \eqref{regID}, $(u_j^{\rm in},\psi_{{\rm i},j}^{\rm in})$ can be, in fact, 
defined for $j=0,1,\ldots,m$ by the above recursion formula so that $u_j^{\rm in}\in H^{m-j}(\cE)$ and 
$\psi_{{\rm i},j}^{\rm in} \in H^{m-j+1/2}(\mathit{\Gamma})$. 
Now, applying $\dt^j$ to the first boundary condition in \eqref{BCs} and putting $t=0$ 
we see that $\{u_k^{\rm in}\}_{k=0}^j$ and $\psi_{{\rm i},j}^{\rm in}$ should satisfy 
\begin{equation}\label{CC}
\sum_{k=0}^j\binom{j}{k}N\cdot(h_{j-k}^{\rm in}v_k^{\rm in}) = \Lambda\psi_{\rm{i},j}^{\rm in} \quad\mbox{on}\quad \mathit{\Gamma}.
\end{equation}
This condition makes sense for $j=0,1,\ldots,m-1$ under our regularity assumptions \eqref{regID} on the initial data and Assumption \ref{ass:hi}. 
As with the linearized equations in the previous section, we need to define the notion of compatible data.

\begin{definition}\label{def:CC3}
Let $m\geq3$ be an integer and suppose that the initial data $(u^{\rm in},\psi_{\rm i}^{\rm in})$ satisfy \eqref{regID} 
and that Assumption \ref{ass:hi} is satisfied. 
For an integer $j$ satisfying $0\leq j\leq m-1$ we say that the data $(u^{\rm in},\psi_{\rm i}^{\rm in})$ for the problem 
\eqref{SWEs}--\eqref{ICs} satisfy the compatibility condition at order $j$ if the $\{u_k^{\rm in}\}_{k=0}^j$ and $\psi_{{\rm i},j}^{\rm in}$ 
defined by \eqref{DefID1} satisfy \eqref{CC}. 
\end{definition}

In Theorem \ref{th:APE}, we established a priori estimates for the nonlinear wave-structure interaction problem \eqref{SWEs}--\eqref{ICs}; 
the following theorem, which is the main result of this article, shows a stronger result, namely, 
the local well-posedness of these equations for data at the quasilinear regularity threshold $m=3$.

\begin{theorem}\label{th:exist}
Let $\gr$, $c_0$, and $M_0$ be positive constants, and $m$ an integer such that $m\geq3$. 
Under Assumption \ref{ass:hi} there exists a positive time $T$ and a positive constant $C$ such that 
for any data $(u^{\rm in},\psi_{\rm i}^{\rm in})$ satisfying 
\[
\begin{cases}
 \Abs{ u^{\rm in} }_{H^m(\cE)} + \abs{ \dtan\psi_{\rm i}^{\rm in} }_{H^{m-1/2}(\mathit{\Gamma})} \leq M_0, \\
 \gr h^{\rm in}(x)-\abs{ v^{\rm in}(x) }^2 \geq 2c_0 \quad\mbox{for}\quad x\in\cE,
\end{cases}
\]
and the compatibility conditions up to order $m-1$, the initial value problem to the nonlinear wave-structure interaction problem 
\eqref{SWEs}--\eqref{ICs} has a unique solution $(u,\psi_{\rm i})$ on the time interval $[0,T]$ satisfying 
\begin{equation}\label{Reg}
u \in \mathbb{W}_T^m, \quad \psi_{\rm i} \in \mathbb{W}_{{\rm b},T}^{m+1/2}
\end{equation}
and $\Abs{ u }_{\mathbb{W}^m_T} + \abs{ \dtan\psi_{\rm i} }_{\mathbb{W}_{{\rm b},T}^{m+1/2}} \leq C$. 
Moreover, the constants $C$ and $T^{-1}$ can be chosen as non-decreasing functions of $c_0^{-1}$ and $M_0$.
\end{theorem}

\subsection{Reduction to a quasilinear system}\label{sect:reduction}
Let $(u,\psi_{\rm i})$ be a smooth solution to the nonlinear wave-structure interaction problem \eqref{SWEs}--\eqref{ICs} and introduce new unknowns 
$(\widecheck{u},\widecheck{\psi}_{\rm i})$ as \eqref{newUNs}. 
Then, as in Proposition \ref{proplintransf} we see that $(\widecheck{u},\widecheck{\psi}_{\rm i})$ satisfies the equations 
\begin{equation}\label{quasilin1}
\begin{cases}
 S(u)\dt\widecheck{u} + G_j\partial_j\widecheck{u} = f(u,\widecheck{u}) &\mbox{in}\quad (0,T)\times\cE, \\
 N\cdot\widecheck{u}^{\rm II} - \Lambda\widecheck{\psi}_{\rm i} = 0 &\mbox{on}\quad (0,T)\times\mathit{\Gamma}, \\
 \dt\widecheck{\psi}_{\rm i} + \widecheck{u}^{\rm I} = 0 &\mbox{on}\quad (0,T)\times\mathit{\Gamma}, \\
 (\widecheck{u},\widecheck{\psi}_{\rm i})_{\vert_{t=0}} = (\widecheck{u}^{\rm in},\widecheck{\psi}_{\rm i}^{\rm in})
   &\mbox{in}\quad \cE\times\mathit{\Gamma},
\end{cases}
\end{equation}
where $G_j$ $(j=1,2)$ and $S(u)$ are the matrices defined as \eqref{defGj} and \eqref{defS} as before, 
whereas $f(u,\widecheck{u})=-({\rm d}_uS(u))[S(u)\widecheck{u}]\widecheck{u}$. 
Then, $(u,\psi_{\rm i})$ can be recovered from $(\widecheck{u},\widecheck{\psi}_{\rm i})$ by solving the following ODEs 
\begin{equation}\label{quasilin2}
\begin{cases}
 \dt u = S(u)\widecheck{u} &\mbox{in}\quad (0,T)\times\cE, \\
 \dt\psi_{\rm i} = \widecheck{\psi}_{\rm i} &\mbox{on}\quad (0,T)\times\mathit{\Gamma}, \\
 (u,\psi_{\rm i})_{\vert_{t=0}}=(u^{\rm in},\psi_{\rm i}^{\rm in}) &\mbox{in}\quad \cE\times\mathit{\Gamma}.
\end{cases}
\end{equation}
There two sets of equations form a quasilinear system of equations for the unknowns $(u,\widecheck{u},\psi_{\rm i},\widecheck{\psi}_{\rm i})$. 
The following proposition shows the equivalence of the reduced quasilinear problem and the original problem 
under appropriate assumptions on the initial data.

\begin{proposition}\label{prop:equi}
Let $T>0$ and suppose that Assumptions \ref{ass:hi} and \ref{ass:subcritical} are satisfied. 
If $(\widecheck{u},\widecheck{\psi}_{\rm i})$ and $(u,\psi_{\rm i})$ solve \eqref{quasilin1} and \eqref{quasilin2} and if the initial data satisfy 
\begin{equation}\label{CCequi}
\begin{cases}
 S(u^{\rm in})\widecheck{u}^{\rm in} + \partial_j\mathcal{F}_j(u^{\rm in}) = 0 &\mbox{in}\quad \cE, \\
 \mathcal{F}_{\rm nor}(u^{\rm in})^{\rm I} - \Lambda \psi_{\rm i}^{\rm in} = 0 &\mbox{on}\quad \mathit{\Gamma}, \\
 \widecheck{\psi}_{\rm i}^{\rm in} + N\cdot\mathcal{F}_{\rm nor}(u^{\rm in})^{\rm II} = 0 &\mbox{on}\quad \mathit{\Gamma},
\end{cases}
\end{equation}
then $(u,\psi_{\rm i})$ solves \eqref{SWEs}--\eqref{ICs}. 
\end{proposition}

\begin{remark}\label{rem:equi}
In terms of the notations introduced in Section \ref{sect:state-ET}, the first and the third conditions in \eqref{CCequi} can be written as 
$\widecheck{u}^{\rm in} = \Sigma(u^{\rm in})u_1^{\rm in}$ and $\widecheck{\psi}_{\rm i}^{\rm in}=\psi_{{\rm i},1}^{\rm in}$, see also \eqref{newUNs}, 
whereas the second condition is nothing but the compatibility condition at order $0$ for the problem \eqref{SWEs}--\eqref{ICs}. 
\end{remark}

\begin{proof}[Proof of Proposition \ref{prop:equi}]
We recall that \eqref{SWEs} and \eqref{BCs} can be written in the form \eqref{PB1F} and \eqref{PB2F}, respectively. 
It follows from \eqref{quasilin1} and \eqref{quasilin2} that 
\begin{align*}
\dt(\dt u + \partial_j\mathcal{F}_j(u))
&= \dt(S(u)\widecheck{u}) + \partial_j(G_j\Sigma(u)\dt u) \\
&= S(u)\dt\widecheck{u} + ({\rm d}S(u))[S(u)\widecheck{u}]\widecheck{u} + G_j\partial_j\widecheck{u} \\
&=0 \quad\mbox{in}\quad (0,T)\times\cE.
\end{align*}
Similarly, we have 
\[
\begin{cases}
 \dt( {\mathcal F}_{\rm nor}(u)^{\rm I}-\Lambda \psi_{\rm i} ) = 0 &\mbox{on}\quad (0,T)\times\mathit{\Gamma}, \\
 \dt( \partial_t \psi_{\rm i}+N\cdot{\mathcal F}_{\rm nor}(u)^{\rm II} ) = 0 &\mbox{on}\quad (0,T)\times\mathit{\Gamma}.
\end{cases}
\]
On the other hand, \eqref{CCequi} can be written as 
\[
\begin{cases}
 ( \dt u + \partial_j\mathcal{F}_j(u) )_{\vert_{t=0}} = 0 &\mbox{in}\quad \cE, \\
 ( {\mathcal F}_{\rm nor}(u)^{\rm I}-\Lambda \psi_{\rm i} )_{\vert_{t=0}} = 0 &\mbox{on}\quad (0,T)\times\mathit{\Gamma}, \\
 ( \partial_t \psi_{\rm i}+N\cdot{\mathcal F}_{\rm nor}(u)^{\rm II} )_{\vert_{t=0}} = 0 &\mbox{on}\quad (0,T)\times\mathit{\Gamma}.
\end{cases}
\]
Therefore, we obtain the desired result. 
\end{proof}

\subsection{Existence of the solution under additional conditions}\label{sect:exist}
In this subsection, we prove Theorem \ref{th:exist} under additional assumptions that $u^{\rm in}\in H^{m+1}(\cE)$, 
$\dtan\psi_{\rm i}^{\rm in}\in H^{m+1-1/2}(\mathit{\Gamma})$, and that the initial data satisfy the compatibility conditions up to order $m$, 
by constructing a solution to the reduced quasilinear problem \eqref{quasilin1}--\eqref{quasilin2}. 
Moreover, this solution is constructed as the limit of a sequence of approximate solutions; 
we first need to construct carefully the first iterate of the sequence so that the compatibility conditions are propagated by the iterative scheme.

\medskip
\noindent
{\bf Step 1.} Construction of the first iterate $u^{(0)}$. 
We use the same notations $\{(u_j^{\rm in},\psi_{{\rm i},j}^{\rm in})\}_{j=0}^{m+1}$ as in Section \ref{sect:state-ET}. 
Also, assuming that $(\widecheck{u},\widecheck{\psi}_{\rm i})$ is defined as \eqref{newUNs} from $(u,\psi_{\rm i})$ we put 
$\widecheck{u}_j^{\rm in}=(\dt^j\widecheck{u})_{\vert_{t=0}}$ and $\widecheck{\psi}_{{\rm i},j}^{\rm in}=(\dt\widecheck{\psi}_{\rm i})_{\vert_{t=0}}$, 
which can be written explicitly as 
\[
\begin{cases}
 (\widecheck{u}_j^{\rm in})^{\rm I} = \gr\zeta_{j+1}^{\rm in} + \frac12\sum_{k=0}^{j+1}\binom{j+1}{k}v_{j+1-k}^{\rm in}\cdot v_k^{\rm in}, \\
 (\widecheck{u}_j^{\rm in})^{\rm II} = \sum_{k=0}^{j+1}\binom{j+1}{k} h_{j+1-k}^{\rm in}v_k^{\rm in}, \\
 \widecheck{\psi}_{{\rm i},j}^{\rm in} = \psi_{{\rm i},j+1}^{\rm in}
\end{cases}
\]
for $j=0,1,\ldots,m$. 
We see easily that $\widecheck{u}_j^{\rm in} \in H^{m-j}(\cE)$ and $\widecheck{\psi}_{{\rm i},j}^{\rm in} \in H^{m-j+1/2}(\mathit{\Gamma})$ 
as well as $u_j^{\rm in} \in H^{m+1-j}(\cE)$ and $\psi_{{\rm i},j}^{\rm in} \in H^{m+1-j+1/2}(\mathit{\Gamma})$ for $j=0,1,\ldots,m$. 
Therefore, we can take positive constants $K_1$ and $K_2$ such that 
\[
\begin{cases}
 \sum_{j=0}^{m} \Abs{ u_j^{\rm in} }_{H^{m-j}(\cE)} \leq K_1, \\
 \sum_{j=0}^{m} \bigl( \Abs{ \widecheck{u}_j^{\rm in} }_{H^{m-j}(\cE)}
  + \abs{ \dtan\widecheck{\psi}_{{\rm i},j}^{\rm in} }_{H^{m-j-1/2}(\mathit{\Gamma})} \bigr) \leq K_2.
\end{cases}
\]
For $T,M>0$ we denote by $S_T$ the set of all functions $u=(\zeta,v^{\rm T})^{\rm T}\in\mathbb{W}_T^{m}$ which satisfy 
\[
\begin{cases}
 (\dt^j u)_{\vert_{t=0}}=u_j^{\rm in} \quad\mbox{for}\quad j=0,1,\ldots,m, \\
 \gr h(t,x) - \abs{v(t,x)}^2 \geq c_0 \quad\mbox{for}\quad (t,x)\in\overline{(0,T)\times\cE}, \\
 \Abs{ u }_{\mathbb{W}_T^{m}} \leq 2K_1,
\end{cases}
\]
and by $\widecheck{S}_{T,M}$ the set of all functions $\widecheck{u}\in\mathbb{W}_T^{m}$ which satisfy 
\[
\begin{cases}
 (\dt^j\widecheck{u})_{\vert_{t=0}}=\widecheck{u}_j^{\rm in} \quad\mbox{for}\quad j=0,1,\ldots,m, \\
 \Abs{ \widecheck{u} }_{\mathbb{W}_T^{m}} \leq M.
\end{cases}
\]
Here, we note that for any $u\in S_T$ we have $\Abs{ u }_{C^1_{\rm b}(\overline{(0,T)\times\cE})} \leq C(K_1)$ and 
$\partial S(u) \in \mathbb{W}_T^{m-1} \cap \mathbb{W}_T^{1,p}$ for any $p\in[2,\infty)$ with 
$\Abs{ \partial S(u) }_{\mathbb{W}_T^{m-1} \cap \mathbb{W}_T^{1,p}} \leq C\bigl(\frac{1}{c_0},K_1,p\bigr)$. 
It is easy to check that there exist sufficiently small $T_0$ and large $M_0$ such that $S_{T_0}\ne\emptyset$ and 
$\widecheck{S}_{T_0,M_0}\ne\emptyset$. 
We take $u^{(0)} \in S_{T_0}$ and $\widecheck{u}^{(0)} \in \widecheck{S}_{T_0,M_0}$ arbitrarily and fix them.

\medskip
\noindent
{\bf Step 2.} Construction of a sequence of approximate solutions. 
Given $u^{(n)}\in S_T$ and $\widecheck{u}^{(n)} \in \widecheck{S}_{T,M}$, we consider the initial value problem for unknowns 
$(\widecheck{u}^{(n+1)},\widecheck{\psi}_{\rm i}^{(n+1)})$
\[
\begin{cases}
 S(u^{(n)})\dt\widecheck{u}^{(n+1)} + G_j\partial_j\widecheck{u}^{(n+1)} = f(u^{(n)},\widecheck{u}^{(n)}) &\mbox{in}\quad (0,T)\times\cE, \\
 N\cdot(\widecheck{u}^{(n+1)})^{\rm II} - \Lambda\widecheck{\psi}_{\rm i}^{(n+1)} = 0 &\mbox{on}\quad (0,T)\times\mathit{\Gamma}, \\
 \dt\widecheck{\psi}_{\rm i}^{(n+1)} + (\widecheck{u}^{(n+1)})^{\rm I} = 0 &\mbox{on}\quad (0,T)\times\mathit{\Gamma}, \\
 (\widecheck{u}^{(n+1)},\widecheck{\psi}_{\rm i}^{(n+1)})_{\vert_{t=0}} = (\widecheck{u}_0^{\rm in},\widecheck{\psi}_{{\rm i},0}^{\rm in})
  &\mbox{in}\quad \cE\times\mathit{\Gamma}.
\end{cases}
\]
We see that the initial data $(\widecheck{u}_0^{\rm in},\widecheck{\psi}_{{\rm i},0}^{\rm in}) \in H^{m}(\cE)\times H^{m+1/2}(\mathit{\Gamma})$ 
for this problem satisfy the compatibility conditions up to order $m-1$, so that we can apply Theorem \ref{th:linexist} to obtain the unique solution 
$(\widecheck{u}^{(n+1)},\widecheck{\psi}_{\rm i}^{(n+1)}) \in \mathbb{W}_T^{m}\times\mathbb{W}_{{\rm b},T}^{m+1/2}$. 
Moreover, by Proposition \ref{prop:UEE} the solution satisfies
\begin{align*}
\Abs{ \widecheck{u}^{(n+1)} }_{\mathbb{W}_T^{m}}
&\leq C_1e^{C_1T}\bigl( \opnorm{ \widecheck{u}^{(n+1)}(0) }_{m} + \abs{ \dtan\widecheck{\psi}_{\rm i}^{(n+1)}(0) }_{H_{(m)}^{-1/2}}
 + \Abs{ f(u^{(n)},\widecheck{u}^{(n)}) }_{H^{m}(0,T)\times\cE} \bigr) \\
&\leq C_1e^{C_1T}( K_2 + C_2\sqrt{T} ),
\end{align*}
where $C_1=C_1\bigl( \frac{1}{c_0}, K_1 \bigr)$ and $C_2=C_2(K_1,M)$. 
Then, we define $u^{(n+1)}$ as the unique solution to 
\[
\begin{cases}
 \dt u^{(n+1)} = S(u^{(n)})\widecheck{u}^{(n)} &\mbox{in}\quad (0,T)\times\cE, \\
 {u^{(n+1)}}_{\vert_{t=0}}=u^{\rm in} &\mbox{in}\quad \cE.
\end{cases}
\]
We see that $u^{(n+1)}\in\mathbb{W}_T^{m}$ satisfies 
\begin{align*}
\Abs{ u^{(n+1)} }_{\mathbb{W}_T^{m}}
&\leq \opnorm{ u^{(n+1)}(0) }_{m} + \int_0^T\opnorm{ S(u^{(n)}(t))\widecheck{u}^{(n)}(t) }_{m}{\rm d}t \\
&\leq K_1 + C_3T,
\end{align*}
where $C_3=C_3(K_1,M)$. 
Therefore, if $T>0$ is so small that $C_3T\leq K_1$, then we have $\Abs{ u^{(n+1)} }_{\mathbb{W}_T^{m}} \leq 2K_1$. 
Under this restriction on $T$, we also have 
\begin{align*}
\gr h^{(n+1)}(t,x) - \abs{ v^{(n+1)}(t,x) }^2
&\geq \gr h^{\rm in}(x) - \abs{ v^{\rm in}(x) }^2 - C_4t \\
&\geq 2c_0 - C_4t \quad\mbox{for}\quad (t,x)\in\overline{(0,T)\times\cE}, 
\end{align*}
where $C_4=C_4(K_1)$. 
Therefore, if we take $M>0$ so large that $M\geq \max\{ M_0, 2C_1K_2 \}$ and $T>0$ so small that 
$C_1T\leq\log2$, $C_2\sqrt{T}\leq K_2$, $C_3T\leq K_1$, and $C_4T\leq c_0$, 
then we see that $u^{(n+1)} \in S_T$ and $\widecheck{u}^{(n+1)} \in \widecheck{S}_{T,M}$. 
Hence, we have constructed a sequence of approximated solutions $\{(u^{(n)},\widecheck{u}^{(n)},\widecheck{\psi}_{\rm i}^{(n)})\}_{n=1}^\infty$.

\medskip
\noindent
{\bf Step 3.} Convergence and regularity of the limit. 
We classically obtain that $\{(u^{(n)},\widecheck{u}^{(n)})\}_{n=1}^\infty$ and $\{\widecheck{\psi}_{\rm i}^{(n)}\}_{n=1}^\infty$ converge 
in $\mathbb{W}_T^{m-1}$ and $\mathbb{W}_{{\rm b},T}^{m-1+1/2}$, respectively. 
We denote by $(u,\widecheck{u})$ and by $\widecheck{\psi}_{\rm i}$ their limits and define $\psi_{\rm i}$ as 
$\psi_{\rm i}(t,\cdot)=\psi_{\rm i}^{\rm in}+\int_0^t\widecheck{\psi}_{\rm i}(t',\cdot){\rm d}t'$. 
Then, we see that $(u,\widecheck{u},\psi_{\rm i},\widecheck{\psi}_{\rm i})$ is a unique solution to the quasilinear system \eqref{quasilin1} 
and \eqref{quasilin2} with the initial data 
$(\widecheck{u}^{\rm in},\widecheck{\psi}_{\rm i}^{\rm in}) = (\widecheck{u}_0^{\rm in},\widecheck{\psi}_{{\rm i},0}^{\rm in})$. 
Particularly, it follows form Proposition \ref{prop:equi} and Remark \ref{rem:equi} that $(u,\psi_{\rm i})$ is a unique solution 
to the nonlinear wave-structure interaction problem \eqref{SWEs}--\eqref{ICs}. 
This solution satisfies $u,\widecheck{u}\in\mathbb{W}_T^{m-1}$ and, by a standard compactness argument, 
$\opnorm{ u(t) }_{m} \leq 2K_1$ and $\opnorm{ \widecheck{u}(t) }_{m} \leq M$ for any $t\in[0,T]$. 
Since $\dt u = S(u)\widecheck{u}$, we can recover the continuity of $u$ in $t$ in the norm $\opnorm{ \cdot }_{m}$, that is, $u \in \mathbb{W}_T^{m}$. 
Moreover, we have $f(u,\widecheck{u}) \in H^{m}((0,T)\times\cE)$, 
so that by regarding \eqref{quasilin1} as a linear problem for $(\widecheck{u},\widecheck{\psi}_{\rm i})$ 
we can apply Theorem \ref{th:linexist} to obtain $\widecheck{u}\in\mathbb{W}_T^{m}$ and $\widecheck{\psi}_{\rm i}\in\mathbb{W}_{{\rm b},T}^{m+1/2}$. 
By $\dt u = S(u)\widecheck{u}$ again, we also obtain $\dt u\in \mathbb{W}_T^{m}$.

\subsection{Completion of the proof of the existence theorem}\label{sect:exist2}
In this last subsection, we complete the proof of Theorem \ref{th:exist} without assuming the additional assumptions in the previous subsection. 
As in the proof of Proposition \ref{prop:appData2}, we can approximate the initial data $(u^{\rm in},\psi_{\rm i}^{\rm in})$ 
by a sequence of more regular data 
$\{ (u^{{\rm in}(n)},\psi_{\rm i}^{{\rm in}(n)}) \}_{n=1}^\infty \subset H^{m+1}(\cE)\times H^{m+1+1/2}(\mathit{\Gamma})$ such that each initial data 
$(u^{{\rm in}(n)},\psi_{\rm i}^{{\rm in}(n)})$ for the nonlinear wave-structure interaction problem \eqref{SWEs}--\eqref{BCs} satisfy 
the compatibility conditions up to order $m$ and that they converge to the original data in $H^{m}(\cE)\times H^{m+1/2}(\mathit{\Gamma})$. 
Without loss of generality, we can assume that 
\[
\begin{cases}
 \Abs{ u^{{\rm in}(n)} }_{H^m(\cE)} + \abs{ \dtan\psi_{\rm i}^{{\rm in}(n)} }_{H^{m-1/2}(\mathit{\Gamma})} \leq 2M_0, \\
 \gr h^{{\rm in}(n)}(x) - \abs{ v^{{\rm in}(n)}(x) }^2 \geq c_0 \quad\mbox{for}\quad x\in\cE,
\end{cases}
\]
for $n=1,2,\ldots$. 
By the result in the previous subsection, for each $n$ there exists a positive time $T_n$ such that the nonlinear wave-structure interaction problem 
\eqref{SWEs}--\eqref{BCs} under the initial conditions $(u,\psi_{\rm i})_{\vert_{t=0}}=(u^{{\rm in}(n)},\psi_{\rm i}^{{\rm in}(n)})$ 
has a unique solution $(u^{(n)},\psi_{\rm i}^{(n)}) \in \mathbb{W}_{T_n}^m\times\mathbb{W}_{{\rm b},T_n}^{m+1/2}$. 
Moreover, by the a priori estimate given in Theorem \ref{th:APE} there exist a positive time $T$ and a positive constant $C$ independent of $n$ 
such that these approximate solutions can be extended on the time interval $[0,T]$ and satisfy 
\[
\begin{cases}
 \opnorm{ u^{(n)}(t) }_{m} + \abs{ \dtan\psi_{\rm i}^{(n)}(t) }_{H_{(m)}^{-1/2}} \leq C, \\
 \gr h^{(n)}(t,x) - \abs{ v^{(n)}(t,x) }^2 \geq \frac12c_0 \quad\mbox{for}\quad (t,x)\in\overline{(0,T)\times\cE},
\end{cases}
\]
for $n=1,2,\ldots$. 
To show the convergence of these approximate solutions, we use the quasilinear equations \eqref{quasilin1} and \eqref{quasilin2}. 
In view of \eqref{newUNs} we put $\widecheck{u}^{(n)}=\Sigma(u^{(n)})\dt u^{(n)}$ and $\widecheck{\psi}_{\rm i}^{(n)}=\dt\psi_{\rm i}^{(n)}$, 
which satisfy the quasilinear equations with the corresponding initial data. 
Then, we have $\Abs{ \widecheck{u}^{(n)} }_{\mathbb{W}_T^{m-1}} \leq C_1$ and $\partial S(u^{(n)}) \in \mathbb{W}_T^{m-1} \cap \mathbb{W}_T^{1,p}$ 
for any $p\in[2,\infty)$ with 
$\Abs{ \partial S(u^{(n)}) }_{\mathbb{W}_T^{m-1} \cap \mathbb{W}_T^{1,p}} \leq C_1$, where $C_1$ is a constant independent of $n$. 
Therefore, by the same reasoning in the previous subsection we can show that $\{(u^{(n)},\widecheck{u}^{(n)})\}_{n=1}^\infty$ and 
$\{\widecheck{\psi}_{\rm i}^{(n)}\}_{n=1}^\infty$ converge in $\mathbb{W}_T^{m-2}$ and $\mathbb{W}_{{\rm b},T}^{m-2+1/2}$, respectively. 
We denote by $(u,\widecheck{u})$ and by $\widecheck{\psi}_{\rm i}$ their limits and define $\psi_{\rm i}$ as 
$\psi_{\rm i}(t,\cdot)=\psi_{\rm i}^{\rm in}+\int_0^t\widecheck{\psi}_{\rm i}(t',\cdot){\rm d}t'$, as before. 
Then, we see that $(u,\widecheck{u},\psi_{\rm i},\widecheck{\psi}_{\rm i})$ satisfy the quasilinear equations and that $(u,\psi_{\rm i})$ solves the 
nonlinear wave-structure interaction problem \eqref{SWEs}--\eqref{BCs}.

It remains to show that this solution satisfies the regularity \eqref{Reg}. 
By a standard compactness argument, the solution also satisfies $\opnorm{ u(t) }_m + \abs{ \dtan\psi_{\rm i}(t) }_{H_{(m)}^{-1/2}} \leq C$ 
and $\opnorm{ \widecheck{u}(t) }_{m-1} \leq C_1$ for $t\in[0,T]$. 
Since $\dt u=S(u)\widecheck{u}$, we have $u\in\mathbb{W}_T^{m-1}$. 
By the Gagliardo--Nirenberg inequality $\Abs{f}_{L^p(\cE)} \lesssim \Abs{ f }_{L^2(\cE)}^{2/p} \Abs{ f }_{H^1(\cE)}^{1-2/p}$, 
which holds for any $p\in[2,\infty)$, we have $\opnorm{ u(t) }_{2,p} \lesssim \opnorm{ u(t) }_2^{2/p} \opnorm{ u(t) }_3^{1-2/p}$, 
so that $u\in\mathbb{W}_T^{2,p}$ for any $p\in[2,\infty)$. 
Particularly, we have $\partial S(u) \in \mathbb{W}_T^{m-2}\cap\mathbb{W}_T^{1,p}$. 
We see also that the source term $f(u,\widecheck{u})$ in the quasilinear equations belongs to $H^{m-1}((0,T)\times\cE)$. 
Therefore, we can apply Theorem \ref{th:linexist} to obtain 
$(\widecheck{u},\widecheck{\psi}_{\rm i})\in\mathbb{W}_T^{m-1}\times\mathbb{W}_{{\rm b},T}^{m-1+1/2}$, 
which yields $\dt u\in \mathbb{W}_T^{m-1}$ and $\dt\psi_{\rm i}\in\mathbb{W}_{{\rm b},T}^{m-1+1/2}$. 
Here, we emphasize that this is the place where we use essentially the wider space $\mathbb{W}_T^{1,p}$ rather than $\mathbb{W}_T^2$ 
in Theorems \ref{theogenorderm} and \ref{th:linexist}. 
For the moment, we do not know whether $\partial S(u)\in\mathbb{W}_T^2$ holds or not in the critical case $m=3$; see also Remark \ref{rem:space}. 
Next, we put $\widecheck{u}=\Sigma(u)(\chi_{\rm b}\dtan u)$ and $\widecheck{\psi}_{\rm i}=\dtan\psi_{\rm i}$. 
Although we use the same notation for different functions, it would not bring any confusion. 
By Proposition \ref{proplintransf}, we see that 
\[
\begin{cases}
 S(u)\dt\widecheck{u} + G_j\partial_j\widecheck{u} = f &\mbox{in}\quad (0,T)\times\cE, \\
 N\cdot\widecheck{u}^{\rm II} - \Lambda\widecheck{\psi}_{\rm i} = g_1 &\mbox{on}\quad (0,T)\times\mathit{\Gamma}, \\
 \dt\widecheck{\psi}_{\rm i} + \widecheck{u}^{\rm I} = g_2 &\mbox{on}\quad (0,T)\times\mathit{\Gamma},
\end{cases}
\]
where $f=-[\chi_{\rm b}\dtan,\partial_j]\mathcal{F}_j(u)-(\dt S(u))\widecheck{u}$, $g_1=[\dtan,\Lambda]\psi_{\rm i}$, and 
$g_2=-(\dtan N)\cdot\mathcal{F}_{\rm nor}(u)^{\rm II}$. 
By Proposition \ref{propDN} and Lemma \ref{lem:Moser}, 
we see that $f\in H^{m-1}((0,T)\times\cE)$, $(1+\abs{D})^{-1/2}g_1 \in H^m((0,T)\times\mathit{\Gamma})$, 
and $(1+\abs{D})^{1/2}g_2 \in H^{m-1}((0,T)\times\mathit{\Gamma})$. 
Moreover, by a straightforward calculation we see also that the initial data 
$(\widecheck{u}(0,\cdot),\widecheck{\psi}_{\rm i}(0,\cdot)) \in H^{m-1}(\cE)\times H^{m-1+1/2}(\mathit{\Gamma})$ 
satisfy the compatibility conditions up to order $m-2$. 
Therefore, we can apply Theorem \ref{th:linexist} again to obtain 
$(\widecheck{u},\widecheck{\psi}_{\rm i})\in\mathbb{W}_T^{m-1}\times\mathbb{W}_{{\rm b},T}^{m-1+1/2}$, 
which yields $\chi_{\rm b}\dtan u\in \mathbb{W}_T^{m-1}$ and $\dtan\psi_{\rm i}\in\mathbb{W}_{{\rm b},T}^{m-1+1/2}$. 
Particularly, we obtain $\psi_{\rm i}\in\mathbb{W}_{{\rm b},T}^{m+1/2}$.

We proceed to see a regularity of $\chi_{\rm b}\dnor u$ by reasoning as in the proof of Lemma \ref{lemmapf4}, that is, 
we express the normal derivative in terms of the time and the tangential derivatives by using the nonlinear shallow water equations, 
which can be written as $\dt u + A_j(u)\partial_ju=0$ with $A_j(u)$ given by \eqref{defAj}. 
Particularly, we have 
\[
\dt u + A_{\rm tan}\dtan u + A_{\rm nor}\dnor u = 0 \quad\mbox{in}\quad (0,T)\times(\cE\cap U_\mathit{\Gamma}).
\]
Although the normal matrix $A_{\rm nor}$ is not invertible, we have 
\[
\begin{pmatrix}
 (A_{\rm nor}\dnor u)^{\rm I} \\
 N\cdot (A_{\rm nor}\dnor u)^{\rm II}
\end{pmatrix}
=
\begin{pmatrix}
 N\cdot v & h \\
 \gr & N\cdot v
\end{pmatrix}
\begin{pmatrix}
 \dnor\zeta \\
 N\cdot(\dnor v)
\end{pmatrix}
+
\begin{pmatrix}
 0 \\
 (N^\perp\cdot v)(N^\perp\cdot(\dnor v))
\end{pmatrix}.
\]
These equations imply that $\dnor\zeta$ and $N\cdot(\dnor v)$ can be expressed in terms of $u$, $\dt u$, $\dtan u$, and $N^\perp\cdot(\dnor v)$. 
It also follows from the nonlinear shallow water equations that the vorticity is conserved in time, that is, 
$\nabla^\perp\cdot v(t,\cdot)=\nabla^\perp\cdot v^{\rm in}$, so that we have 
$N^\perp\cdot(\dnor v) = \nabla^\perp\cdot v^{\rm in} - \abs{T}^{-2}T^\perp\cdot(\dtan v)$. 
As a result, we can express $\dnor u$ in terms of $u$, $\dt u$, and $\dtan u$, leading to $\chi_{\rm b}\dnor u\in \mathbb{W}_T^{m-1}$. 
Therefore, we obtain $\chi_{\rm b}\partial u \in \mathbb{W}_T^{m-1}$. 
Similarly and more easily we can show $(1-\chi_{\rm b})\partial u \in \mathbb{W}_T^{m-1}$, leading to $u\in\mathbb{W}_T^m$. 
The proof is complete.


\end{document}